\documentclass[11pt,reqno]{amsart}
\usepackage{}
\usepackage{esint}
\usepackage{bbm}
\usepackage{bm}
\usepackage{amssymb}
\usepackage{mathrsfs}
\usepackage{amsfonts}
\usepackage{amsfonts,amssymb,amsmath,amsthm}
\usepackage{url}
\usepackage{enumerate}
\usepackage[pdftex,bookmarksnumbered]{hyperref}
\usepackage{graphicx,xcolor}
\usepackage{pgfplots}
\usepackage{float}
\usepackage{caption}
\usepackage{ulem}
\usepackage{amsmath}
\usepackage{mathtools}
\makeatletter 
\def\@captype{figure} 
\makeatother  

\usetikzlibrary{patterns,arrows,positioning,calc,fadings,shapes,decorations.markings}
\usepackage{tikz}

\usepackage{multicol}

\newtheorem{theorem}{Theorem}[section]
\newtheorem{lemma}[theorem]{Lemma}

\newtheorem{corollary}[theorem]{Corollary}
\theoremstyle{definition}
\newtheorem{definition}[theorem]{Definition}
\newtheorem{remark}{Remark}

\numberwithin{equation}{section}

\usepackage[left=2.0cm, right=2.0cm, top=2.0cm, bottom=2.0cm]{geometry}

\usepackage{graphicx}%
\usepackage{color}
\usepackage{cite}
\usepackage{hyperref}
\hypersetup{
	colorlinks = true,%
	citecolor = blue,
	filecolor=red,%
	linkcolor = [rgb]{0.65,0.0,0.0},%
	anchorcolor = red,
	pagecolor = red,
	urlcolor= [rgb]{0.65,0.0,0.0}
	linktocpage=true,
	pdfpagelabels=true,
	bookmarksnumbered=true,
}

\makeatother

\makeatletter
\@namedef{subjclassname@2020}{\textup{2020} Mathematics Subject Classification}
\makeatother

\DeclareMathOperator{\m}{\mathfrak{m}}
\DeclareMathOperator{\dm}{d\mathfrak{m}}

\DeclareMathOperator{\dd}{d}

\DeclareMathOperator{\supp}{supp}

\DeclareMathOperator{\dii}{div}

\DeclareMathOperator{\FMMM}{\mathsf{FMMM}}

\newcommand{\Lip}{\mathsf{Lip}}
\newcommand{\dil}{\mathsf{dil}}

\author{Benling Li}
\address{
School of Mathematics and Statistics\\
Ningbo University\\
315211 Ningbo, China}
\email{libenling@nbu.edu.cn}

\author{Zhenkai Xu}
\address{
School of Mathematics and Statistics\\
Ningbo University\\
315211 Ningbo, China}
\email{xuzhenkai666@163.com}

\date{\today}


\thanks{B. Li is supported by Natural Science Foundation of Ningbo (No.~2024J017) and  National Natural Science Foundation of China (No.~12471045). }

\begin{document}

\title[Sharp Constants and Failure Threshold for Uncertainty Principles on Finsler Manifolds]{Sharp Constants, Threshold, and Failure of Uncertainty Principles on Non-trivial Flat Finsler Manifolds}

\begin{abstract}
By uncertainty principles we mean the generalized Caffarelli--Kohn--Nirenberg (CKN) inequality and the Heisenberg--Pauli--Weyl (HPW) uncertainty principle, with the latter being an endpoint case and the Hardy inequality a further limiting case.
We study these \(L^p\)-type inequalities \((1<p<n)\) on \(n\)-dimensional non-trivial flat Finsler manifolds, which arise from Hilbert's fourth problem. These manifolds are forward complete but not backward complete, projectively flat, with zero flag curvature, and have infinite reversibility, making them highly asymmetric. They are the simplest non-Riemannian, non-Minkowskian solutions to Hilbert's fourth problem, providing a natural framework in which classical inequalities exhibit new behavior.

For the full class, we establish a sharp threshold at the parameter value \(s=2\): the generalized CKN inequality and the HPW uncertainty principle hold for \(s>2\) but fail completely for \(s\le2\). When the generating Minkowski norm is reversible (while the Finsler metric itself remains irreversible with infinite reversibility), we determine the sharp constants for \(s>2\), which were previously unknown.  These constants are approached in the limit by compactly supported radial functions. As a corollary, our results cover Berwald's metric space and resolve the open questions on its sharp constants. These findings reveal that the threshold phenomenon is structural for the entire class and is a genuinely non-Riemannian feature with no analogue in the Riemannian flat case.
\end{abstract}
\maketitle

\section{Introduction}
The Caffarelli--Kohn--Nirenberg (CKN) inequality and the Heisenberg--Pauli--Weyl (HPW) uncertainty principle (also called the HPW inequality) are two cornerstones of modern analysis and mathematical physics. In $\mathbb{R}^n$, the classical CKN inequality \cite{CKN} unifies the Hardy inequality, the Sobolev inequality, and the HPW principle within a single framework. In its sharp form,
\begin{equation}\label{eq:CKNclassic}
\left( \int_{\mathbb{R}^n} |\nabla u|^2 {\dd}x \right)^{\frac{1}{2}}
\left( \int_{\mathbb{R}^n} \frac{|u|^{2(a-1)}}{|x|^{2(q-1)}} {\dd}x \right)^{\frac{1}{2}}
\ge \frac{n-q}{a} \int_{\mathbb{R}^n} \frac{|u|^a}{|x|^q} {\dd}x,
\end{equation}
for $0<q<2<a$ and $2<n<\frac{2(a-q)}{a-2}$. The constant $\frac{n-q}{a}$ is sharp, with extremals $u(x)=\left(C+|x|^{2-q}\right)^{\frac{1}{2-a}}$, $C>0$, up to scalar multiplication. As $a\to 2$ and $q\to 0$, \eqref{eq:CKNclassic} degenerates to the HPW principle, whose sharp constant $n/2$ is attained by Gaussians.

In the Riemannian setting, the sharp constants of the CKN and HPW inequalities are well understood. Carron \cite{Carron} initiated the study of weighted Hardy inequalities on manifolds, and Krist\'aly \cite{Kristaly2018} proved that on Riemannian Cartan--Hadamard manifolds the sharp Euclidean constants persist, with extremals forcing flatness. For further developments on Hardy and CKN-type inequalities in curved spaces, we refer to \cite{Berchio,CKZ,CFL,CFLL,Dambrosio,FFK, Kombe1,Kombe2,KKPZ,Yang,KristalyRepovs,doCarmo,Xia,KristalyOhta}.

As Chern \cite{Ch1} remarked, Finsler geometry is just Riemannian geometry without the quadratic restriction. In this broader setting, the picture is fundamentally different. The asymmetry of the metric leads to new phenomena with no Riemannian analogue.
Huang--Krist\'aly--Zhao \cite{HKZ} showed that reversibility, $S$-curvature, and flag curvature govern the validity of uncertainty principles and the existence of extremals. More related results on the validity of CKN/HPW inequalities can be found in \cite{Kristaly2015,KristalyRepovs,MST2020,MPV2020,WeiWu2020,HanXu2024, ZhaoHardy}.
More recent progress on sharp \(L^p\)-uncertainty principles on general Finsler measure spaces can be found in the preprint \cite{XiaLi2026}.
However, on some Finsler manifolds these inequalities fail completely. The first such example was found by Krist\'aly--Rudas \cite{KR} on the Funk metric space over the Euclidean unit ball, where the classical Hardy, HPW, and CKN inequalities fail; this was later extended by Krist\'aly--Li--Zhao \cite{KLZ} to general Funk metric spaces over bounded convex domains and to more general Finsler settings. The underlying reasons are a dominance condition of the $S$-curvature over the Ricci curvature and the infinite reversibility of the metric.

These results leave open a basic question: in the regime of infinite reversibility, can sharp functional inequalities still hold, and what are their optimal constants? The known counterexamples occur on Funk metric spaces, whose geodesics are straight and whose flag curvature is constant and negative. This makes the zero flag curvature case a natural testing ground: there the Riemannian picture is rigid, but the Finsler framework is remarkably rich. This is precisely the setting of Hilbert's fourth problem, which asks for a classification of metric geometries in which straight lines are shortest paths. His original formulation sought geometries that are flat in a weak sense, extending beyond the rigidity of Euclidean and Minkowski spaces. In the smooth setting, this is exactly the class of projectively flat Finsler manifolds. Unlike the Riemannian case---where Beltrami's theorem equates projective flatness with constant sectional curvature---the Finsler framework admits projectively flat metrics of nonconstant flag curvature, as well as constant curvature metrics that are not projectively flat.

Building on the classical works of Berwald \cite{Be1,Be2}, Bryant \cite{Bryant1,Bryant2}, Shen \cite{Sh1}, and Li \cite{Li},  Li--Zhao \cite{LiZhao_Hilbert4} recently obtained a complete classification of forward complete, projectively flat Finsler manifolds with constant flag curvature.  Moreover, the same classification shows that, in the zero flag curvature case, the entire metric is generated by two positively $1$-homogeneous functions at a reference point (without loss of generality, the origin $\mathbf{0}$): the Minkowski norm $\psi(y):=F(\mathbf{0},y)$ and the projective factor $\phi(y):=P(\mathbf{0},y)$. More precisely, the projective factor $P(x,y)$ is the unique solution to
\begin{equation}\label{Pgene}
P(x,y)=\phi(y+xP(x,y)),
\end{equation}
and the Finsler metric is reconstructed by
\begin{equation}\label{Fgene}
F(x,y)=\psi(y+xP(x,y))\{1+x^kP_{y^k}(x,y)\}.
\end{equation}
For a metric that is forward complete but not backward complete and satisfies $\phi(x)>0$ for all $x\neq\mathbf{0}$, the domain $\Omega$ is necessarily of the form
\[
\Omega=\{x\in\mathbb{R}^n:\phi(x)<1\}.
\]
Thus the classification reveals a fundamental dichotomy: either the manifold is a Minkowski space, or it is a genuinely non-Riemannian, forward complete but \emph{not backward complete} object. The latter are precisely the simplest solutions to Hilbert's fourth problem that are neither Minkowski nor Riemannian; we call them  \textit{non-trivial flat} Finsler manifolds.
We denote by $\lambda_\psi$ the reversibility of $\psi$, and call $\psi$ the generating Minkowski norm of $F$; these quantities will appear naturally in the sharp constants below.

A striking example is Berwald's metric on the Euclidean unit ball $\mathbb{B}^n_1(\mathbf{0})$, introduced by Berwald \cite{Be1,Be2},
\begin{equation}\label{Berwald}
\mathsf{B}(x,y):= \frac{\bigl(\sqrt{(1-|x|^2)|y|^2+\langle x,y\rangle^2}+\langle x,y\rangle\bigr)^2}
{(1-|x|^2)^2\sqrt{(1-|x|^2)|y|^2+\langle x,y\rangle^2}},
\qquad (x,y)\in T\mathbb{B}^n_1(\mathbf{0}).
\end{equation}
This metric has constant flag curvature zero, and all its geodesics are straight, yet exhibits extreme asymmetry: for the radial direction $\mathsf{y}:=x/|x|$,
\[
\lim_{|x|\to 1^-}\mathsf{B}(x,\mathsf{y})=+\infty,\qquad
\lim_{|x|\to 1^-}\mathsf{B}(x,-\mathsf{y})=\frac14.
\]
Consequently it admits no isometric embedding into any Minkowski space, and the associated Sobolev spaces are nonlinear; classical inequalities such as the Hardy inequality and the HPW principle fail completely on this space \cite{LZ}. In fact, $\mathsf{B}(x,y)$ is generated by $\psi=\phi=|\cdot|$  (the Euclidean norm) through \eqref{Pgene}--\eqref{Fgene}; see \cite{Li,Sh1}.

Compared with its Euclidean counterpart, the Finsler generalization of the CKN inequality acquires an additional free parameter $s$ governing the weight of the potential term.
Li and Zhao \cite{LZ} showed that on Berwald's metric space this generalized inequality exhibits a \emph{sharp threshold}: it holds for $s>2$ but fails completely when $s\le 2$.
To our knowledge, this is the first instance of such a threshold phenomenon even in the Finsler framework.
Its occurrence is notable because, on Finsler manifolds with extreme asymmetry (i.e., infinite reversibility), the relevant functional inequalities generally cannot be expected to hold globally; see Huang--Krist\'aly--Zhao \cite{HKZ}, Kaj\'ant\'o \cite{Ka}, and Mester--Peter--Varga \cite{MPV2020}. For instance, on Funk metrics---which are projectively flat and forward complete but not backward complete---the generalized CKN inequality fails unconditionally \cite[Theorem 1.3]{LZ}. This shows that such global inequalities are far from automatic on Finsler manifolds with extreme asymmetry, making the threshold phenomenon on Berwald's metric space all the more striking.

Since Berwald's metric is only one member of the class of non-trivial flat Finsler manifolds identified in \cite{LiZhao_Hilbert4}, it is natural to ask whether the threshold phenomenon is a structural feature of the entire class or merely an artifact of this particular model. Moreover, even in the Berwald case, the sharp constants for the generalized CKN and HPW inequalities were not determined in \cite{LZ}. This motivates the following two questions:

\medskip

{\it (Q1) Is the threshold phenomenon characteristic of the entire class of non-trivial flat Finsler manifolds, or is it an isolated feature of the Berwald model?}

\medskip
{\it (Q2) What are the sharp constants, and do they exhibit the same structural uniformity across the entire class?}
\medskip

The present paper answers both questions. For the entire class of forward complete but not backward complete projectively flat Finsler manifolds with zero flag curvature, we prove that the generalized CKN and HPW inequalities fail whenever $s\le 2$ (Theorem~\ref{thm:CKNHPW_threshold}). Thus the breakdown below the threshold is a structural property of the whole class, not an artifact of the Berwald model. For the sharp constants, we determine them explicitly under the additional assumption that the generating Minkowski norm $\psi$ is reversible, i.e., $\lambda_\psi=1$ (Theorem~\ref{thm:CKNHPW_sharp}). In particular, the sharp constants are $\frac{s-2}{a}$ for the generalized CKN inequality and $\frac{s-2}{p}$ for the generalized HPW uncertainty principle.
This extends the Berwald-model results to a broader geometric framework and provides the first determination of the sharp constants for these inequalities on non-trivial flat Finsler manifolds with reversible generating Minkowski norm.

To state our results, let $F$ be a Finsler metric on a domain $\Omega\subset\mathbb{R}^n$ containing the origin $\mathbf{0}$, where $n\ge2$, and let
\begin{equation} \label{r}
r=r(x):=d_F(\mathbf{0},x)
\end{equation}
denote the distance function from the origin. We endow $\Omega$ with the Lebesgue measure $\mathscr{L}^n$, which is natural in this projectively flat setting since $\Omega$ is a domain in $\mathbb{R}^n$ and the geometric data are expressed in Euclidean coordinates.
In the Finsler setting, the natural generalization of the CKN inequality replaces $|\nabla u|$ and $|x|$ by $F(\nabla u)$ and $r(x)$ defined in \eqref{r}, respectively.
For $a\ge p>1$ and $0\le b<1$, the generalized CKN/HPW inequality takes the form
\[
\left(\int_\Omega r^{bps} F^p(\nabla f)\,{\dd}x \right)^{1/p}
\left(\int_\Omega |f|^{p'(a-1)} r^{(1-b)p's}\,{\dd}x \right)^{1/p'}
\ge C \int_\Omega |f|^a r^{s-1}\,{\dd}x,
\qquad f\in C_0^\infty(\Omega),
\]
where $p'=p/(p-1)$ is the H\"older conjugate of $p$. The parameter $s\in(1-n,\infty)$ plays the role of the threshold variable, with the critical value $s=2$ separating the valid regime $s>2$ from the failing regime $s\le2$.

Our first main result is the weighted Hardy-type inequality that underlies all subsequent estimates.
\begin{theorem}\label{main_thmgeneral1}
Let $(\Omega, F, \mathscr{L}^n)$ be a forward complete projectively flat Finsler manifold with zero flag curvature, endowed with the Lebesgue measure. Let $a>1$ and $s>1-n$. Then for any $f\in C_0^\infty(\Omega)$,
\begin{equation}\label{firsteq}
\int_\Omega |f|^{a-1} r^s F(\nabla f)\,{\dd}x
\ge \frac{s-2}{a\lambda_\psi} \int_\Omega |f|^a r^{s-1}\,{\dd}x
+ \frac{n+1}{a\lambda_\psi} \int_\Omega |f|^a (1-\phi(x)) r^{s-1}\,{\dd}x,
\end{equation}
where $\lambda_\psi$ is the reversibility of the generating Minkowski norm $\psi(y):=F(\mathbf{0},y)$, and $\phi$ is the projective factor at the origin. In particular, for $\phi(x)\equiv 0$,
\begin{equation}\label{firsteqphi0}
\int_\Omega |f|^{a-1} r^s F(\nabla f)\,{\dd}x
\ge \frac{n+s-1}{a\lambda_\psi} \int_\Omega |f|^a r^{s-1}\,{\dd}x.
\end{equation}
\end{theorem}
We emphasize that the assumption that the manifold is not backward complete is not needed in Theorem~\ref{main_thmgeneral1}; it also covers the Minkowski case \(\phi\equiv0\).
The two terms on the right-hand side of \eqref{firsteq} explain how these non-trivial flat Finsler manifolds differ from Euclidean and Minkowskian manifolds. The second term involves the factor $\phi$, which satisfies $0<\phi(x)<1$ for every $x\neq\mathbf{0}$, whereas on Euclidean and Minkowskian manifolds $\phi\equiv0$, so the inequality reduces to \eqref{firsteqphi0}.

The presence of $\lambda_\psi$ in the denominator reflects the role of asymmetry: the constant is inversely proportional to the reversibility of the generating norm. Combining Theorem~\ref{main_thmgeneral1} with H\"older's inequality immediately yields the generalized CKN inequality (when $a>p$) and the generalized HPW uncertainty principle (when $a=p$) in the regime $s>2$, and the same threshold phenomenon governs their failure. Our next result establishes this complete characterization for the general (non-reversible) case.

\begin{theorem}\label{thm:CKNHPW_threshold}
Let $(\Omega, F, \mathscr{L}^n)$ be a forward complete but not backward complete projectively flat Finsler manifold with zero flag curvature, endowed with the Lebesgue measure. Let $ 1< p <n$, $a>1$, and $0\le b<1$. Then the following hold:
\begin{enumerate}[\rm (i)]
\item  \label{1.2(i)}
If $s>2$,
then the following inequality holds for all $f\in C_0^\infty(\Omega)$:
\[
\left(\int_\Omega r^{bps} F^p(\nabla f)\,{\dd}x \right)^{1/p}
\left(\int_\Omega |f|^{p'(a-1)}\, r^{(1-b)p's}\,{\dd}x \right)^{1/p'}
\ge \frac{s-2}{a\lambda_\psi} \int_\Omega |f|^a\, r^{s-1}\,{\dd}x.
\]

\item \label{1.2(ii)}
 If $1-n<s\le 2$,  and $0\le b\le \frac{p+1}{2p}$ such that $bps+n>0$ and $(1-b)p's+n>0$, then the inequality fails completely:
\[
\inf_{f\in C_0^\infty(\Omega)\setminus\{0\}}
\frac{
\left(\int_\Omega r^{bps} F^p(\nabla f)\,{\dd}x\right)^{1/p}
\left(\int_\Omega |f|^{p'(a-1)}\, r^{(1-b)p's}\,{\dd}x\right)^{1/p'}
}{
\int_\Omega |f|^a\, r^{s-1}\,{\dd}x
} = 0.
\]
\end{enumerate}
\end{theorem}
The conditions $bps+n>0$ and $(1-b)p's+n>0$ are automatically satisfied when $s>0$; for $s\le0$ they impose restrictions on $b$ and $s$.

Theorem~\ref{thm:CKNHPW_threshold} establishes the threshold $s=2$ but does not determine the optimal constant in the valid regime. To obtain the sharp constant, one may assume that the generating Minkowski norm $\psi$ is reversible, i.e., $\lambda_\psi=1$. This is precisely the condition under which equality can be achieved. Our next result determines the sharp constant in this setting.

\begin{theorem}\label{thm:CKNHPW_sharp}
Under the assumptions of Theorem~\ref{thm:CKNHPW_threshold}, suppose in addition that $\psi$ is reversible ($\lambda_\psi=1$), $a\ge p$, and $s>2$. If
\begin{align}\label{condition_bthm1.4}
b< \frac{p+s-1}{p s} - \frac{(a-p)(p+s-3)}{aps},
\end{align}
then for every $f\in C_0^\infty(\Omega)$,
\begin{align} \label{CKNsharp}
\left(\int_\Omega r^{bps} F^p(\nabla f)\,{\dd}x\right)^{1/p}
\left(\int_\Omega |f|^{p'(a-1)}\, r^{(1-b)p's}\,{\dd}x\right)^{1/p'}
\ge \frac{s-2}{a} \int_\Omega |f|^a\, r^{s-1}\,{\dd}x.
\end{align}
The constant $(s-2)/a$ is optimal. Moreover:
\begin{enumerate}[\rm (i)]
 \item \label{1.3(i)} For $a>p$, this is the generalized CKN inequality. Equality is approached in the limit by the family of compactly supported smooth radial functions constructed in Section~\ref{sec:sharp}.
 \item \label{1.3(ii)} For $a=p$, this reduces to the generalized HPW uncertainty principle with sharp constant $(s-2)/p$. Equality is approached in the limit by the family of compactly supported smooth radial functions constructed in Section~\ref{SharpHPW}.
\end{enumerate}
\end{theorem}

\begin{remark}
We emphasize that the condition \(\lambda_\psi=1\) concerns the generating Minkowski norm \(\psi(y)=F(\mathbf{0},y)\), not the Finsler metric \(F\) itself. In the non-trivial flat case, \(F\) is typically irreversible, and its reversibility constant \(\lambda_F\) is infinite. Thus the sharp constant is obtained in a genuinely non-Riemannian and highly asymmetric setting.
\end{remark}

As a direct consequence, we recover and sharpen the previously known results on Berwald's metric space defined in \eqref{Berwald}.
In this case, the generating norm is Euclidean ($\psi(y)=|y|$), hence reversible ($\lambda_\psi=1$), and the projective factor is $\phi(x)=|x|$. Theorem~\ref{thm:CKNHPW_sharp} therefore yields the following.

\begin{corollary}\label{cor:Berwald}
On Berwald's metric space $(\mathbb{B}^n_1(\mathbf{0}), \mathsf{B}, \mathscr{L}^n)$, the following hold:

\begin{enumerate}
\item[(i)] The generalized CKN inequality holds with sharp constant $(s-2)/a$ for $s>2$ and $b$ satisfying \eqref{condition_bthm1.4}, and fails completely for $1- n< s\le 2$ and $b$ satisfying the conditions in Theorem \ref{thm:CKNHPW_threshold}\eqref{1.2(ii)}.

\item[(ii)] The generalized HPW uncertainty principle holds with sharp constant $(s-2)/p$ for $s>2$ and $b$ satisfying \eqref{condition_bthm1.4}, and fails completely for $1- n< s\le 2$ and $b$ satisfying the conditions in Theorem \ref{thm:CKNHPW_threshold}\eqref{1.2(ii)}.
\end{enumerate}
\end{corollary}

This corollary resolves the open questions left in \cite{LZ} regarding the optimal constants on the Berwald's metric model, and shows that the threshold phenomenon observed there is a structural feature of the entire class of non-trivial flat Finsler manifolds with reversible generating norm.

The present results demonstrate that the non-trivial flat Finsler manifolds arising from Hilbert's fourth problem are not merely geometric curiosities. They provide a natural analytic setting in which classical functional inequalities exhibit a systematic, parameter-dependent threshold: the inequality holds for \(s>2\) and fails for \(s\le 2\) across the entire class. Moreover, for \(s>2\) the sharp constant \((s-2)/a\) is established in the case where the generating Minkowski norm \(\psi\) is reversible. We emphasize that this condition concerns \(\psi\), not the Finsler metric \(F\) itself, whose reversibility is infinite. These findings reveal that the threshold phenomenon is a structural feature of the whole class rather than an accident of a particular model.

The paper is organized as follows. Section~\ref{prelimain} collects the necessary preliminaries from Finsler geometry, including the classification theorem needed here, estimates for the distance function and $S$-curvature, and two newly proved essential lemmas. Section~\ref{sec:sharp}
proves  the weighted Hardy-type inequality (Theorem~\ref{main_thmgeneral1}) and the generalized CKN inequality with sharp constant  for $s > 2$ (Theorem~\ref{thm:CKNHPW_threshold}\eqref{1.2(i)}, Theorem~\ref{thm:CKNHPW_sharp}\eqref{1.3(i)}).
Section~\ref{SharpHPW} is devoted to  proving the corresponding generalized HPW inequality with sharp constant (Theorem~\ref{thm:CKNHPW_sharp}\eqref{1.3(ii)}).
Section~\ref{section:K=0} proves the failure for $s\le 2$ (Theorem~\ref{thm:CKNHPW_threshold}\eqref{1.2(ii)}).

\vskip 10mm

\section{Preliminaries}\label{prelimain}

\subsection{Finsler geometry basics}
We assume the reader is familiar with the basic notions of smooth manifolds, tangent and cotangent bundles; see, for example, Warner \cite{Warner} for details. In this section we recall only those definitions and properties from Finsler geometry that are needed in the sequel; for a comprehensive treatment we refer to Bao--Chern--Shen \cite{BCS}, Ohta \cite{Ohta1} and Shen \cite{ShenSpray,ShenLecture}. Einstein's summation convention is used throughout the paper.

Let $V$ be an $n$-dimensional vector space with $n \geq 2$. A {\it pseudo-norm} on $V$ is a nonnegative function $\psi: V \rightarrow [0, +\infty)$ satisfying:
\begin{enumerate}[{\rm (i)}]
\item\label{minkownor1} $\psi(y)\geq 0$ with equality if and only if $y=\mathbf{0}$;
\item\label{minkownor2} $\psi$ is positively $1$-homogeneous, i.e., $\psi(\alpha y) = \alpha \psi(y)$ for all $\alpha > 0$;
\item  \begin{equation} \label{phiquasiconvex0}
\psi(y_1+y_2)\leq \psi(y_1)+\psi(y_2), \quad \forall y_1,y_2\in V.
\end{equation}
\end{enumerate}
A pseudo-norm $\psi$ is  a {\it weak Minkowski norm} if it is smooth on $V\backslash\{\mathbf{0}\}$ and
the equality in \eqref{phiquasiconvex0} holds only for
\[
\alpha_1 y_1=\alpha_2y_2, \quad \alpha_1\geq 0, \quad \alpha_2\geq 0, \quad \alpha_1+\alpha_2>0.
\]
It is a {\it Minkowski norm} if, in addition, the Hessian matrix $\left( [\psi^2]_{y^i y^j}(y) \right)$ is positive definite for all $y \in V \backslash\{\mathbf{0}\}$.
The pair $(V, \psi)$ is then called a ({\it weak}) {\it Minkowski space}.

A {\it Finsler metric} on a smooth manifold $M$ is a function $F: TM   \to [0, +\infty)$ such that $F(x,\cdot)$ is a Minkowski norm on $T_xM$ for every $x \in M$. The pair $(M,F)$ is called a {\it Finsler manifold}.
In particular, when $M$ is taken to be a vector space $V$ and $F(x,y) = \psi(y)$ is independent of $x$, the above definition reduces to that of a Minkowski space $(V,\psi)$. This elementary case will serve as a recurring reference in the classification of flat projectively flat structures in Section \ref{subsectionK=0}.

The {\it fundamental form} $g_y$ on $T_xM$ is
\[ g_y(u,v) = g_{ij}(x,y) u^i v^j, \quad \forall u, v \in T_xM, \]
where $g_{ij}(x,y)$
is the {\it fundamental tensor} of $F$ defined by
\[ g_{ij}(x,y):=\tfrac12[F^2]_{y^iy^j}(x,y), \qquad \forall y \in T_xM \backslash \{\mathbf{0}\}.\]
 It satisfies $F^2=g_{ij}(x,y)y^iy^j$ by Euler's theorem for homogeneous functions.
The Riemannian case corresponds to the special situation where $g_{ij}$ is independent of $y$. In general, however, $g_{ij}$ depends on the direction $y$ and the point $x$.

The possible failure of symmetry is a distinctive feature of Finsler metrics. Its extent is measured by the {\it reversibility}
\begin{equation*}
\lambda_F(M):=\sup_{x\in M}\lambda_F(x) \ \ {\rm with}\ \ \lambda_F(x)= \sup_{y\in T_xM\setminus\{\mathbf{0}\}} \frac{F(x,-y)}{F(x,y)},
\end{equation*}
 see Rademacher \cite{Rade1,Rade}.
One has \(\lambda_F(M)=1\) if and only if $F $ is {\it reversible}, i.e., $F(x, y) = F(x, -y)$; this class includes all Riemannian metrics, but is strictly larger in the Finsler setting.

The {\it Legendre transformation} $\mathfrak{L} : TM \rightarrow T^*M$ is defined by \[
\mathfrak{L}(x,y) := g_{ij}(x,y)\, y^i dx^j \quad \text{for } y \neq 0,
\]
and extended to the zero section by $\mathfrak{L}(x,0) := 0$. By the strong convexity of $F$, $\mathfrak{L}$ restricts to a diffeomorphism from $TM\setminus\{0\}$ to $T^*M\setminus\{0\}$.
For a smooth function $f$, its {\it gradient}  is given by \[\nabla f := \mathfrak{L}^{-1}({\dd}f).\]  In particular, \(\nabla f = 0\) on the set \(\{df = 0\}\); consequently, ${\dd}f(X) = g_{\nabla f} (\nabla f,X)$ on the set $\{ {\dd}f \neq 0\}$.

The {\it co-metric} (or {\it dual metric}) $F^*$ on $T^*M$ is defined as the support function of the unit ball of $F$:
\begin{equation}\label{coFinsler}
F^*(x,\eta):=\underset{y\in T_xM\backslash \{\mathbf{0}\}}{\sup}\frac{\langle\eta,y\rangle}{F(x,y)}, \qquad
\forall \eta\in T_x^*M.
\end{equation}
It is itself a Finsler metric on the cotangent bundle \(T^*M\). From this definition one obtains the duality relation
\begin{equation*}
F^*(\mathfrak L(x,y))=F(x,y),\qquad \forall(x,y)\in TM,
\end{equation*}
which implies
\begin{equation}\label{granotwoff}
		F^*(x,{\dd}f)=F(x,\nabla f),  \quad (F^*({\dd}f)=F(\nabla f) \text{ for short}).
\end{equation}
In the sequel, when stating our results we shall use the gradient form
\(F(\nabla f)\), while in the proofs we shall often work with
\(F^*(\mathrm{d}f)\); the two are interchangeable by \eqref{granotwoff}.

The definition \eqref{coFinsler} also gives the elementary inequality
\begin{equation}\label{dualff*}
\langle \eta,y\rangle\le F^*(x,\eta)F(x,y),
\end{equation}
with equality if and only if \(\eta=\alpha\mathfrak L(x,y)\) for some \(\alpha>0\). Since \(F\) need not be symmetric, we shall also use the following inequality:
\begin{equation}\label{absolute_dualff*}
|\langle\eta,y\rangle| \leq F^*(x,\eta) \max \{ F(x,y), F(x,-y)\}, \qquad  \forall \eta\in T^*_xM,\ y\in T_xM,
\end{equation}
which will be essential in the proof of Theorem \ref{main_thmgeneral1}.

A geodesic is a smooth curve \(\gamma:I\to M\) whose coordinate functions \(\gamma^i(t)\) satisfy the second-order system
\[
\ddot{\gamma}^i + 2G^i(\gamma,\dot{\gamma})=0,
\]
where \(G^i(x,y)\) are the geodesic coefficients, given by
\[
G^i(x,y):=\frac14 g^{il}(x,y)\left(2\frac{\partial g_{jl}}{\partial x^k}(x,y)-\frac{\partial g_{jk}}{\partial x^l}(x,y)\right)y^jy^k.
\]
These coefficients determine the metric connection and encode the directional dependence of the geometry. Geodesics are locally minimizing and have constant speed (with respect to $F$).

The {\it flag curvature} $\mathbf K$ is the Finslerian analogue of sectional curvature. For a non-zero $y\in T_xM$ and a transverse vector $v\in T_xM$, it is defined by
\[
\mathbf K(x,y,v):=\frac{g_y(R_y(v),v)}{g_y(y,y)g_y(v,v)-g_y(y,v)^2},
\]
where \(R_y\) is the {\it Riemannian curvature}
 of $F$, which is a family of linear transformations on tangent spaces. Explicitly, it is given by
$R_y:=R^i_k(x,y)\frac{\partial}{\partial x^i}\otimes {\dd}x^k$, where
\[
R^i_{\,k}(x,y):=2\frac{\partial G^i}{\partial x^k}-y^j\frac{\partial^2G^i}{\partial x^j\partial y^k}+2G^j\frac{\partial^2 G^i}{\partial y^j \partial y^k}-\frac{\partial G^i}{\partial y^j}\frac{\partial G^j}{\partial y^k}.
\]
The {\it Ricci curvature} is then the trace
\[
\mathbf{Ric}(x,y):=\frac1{F^2(x,y)}\sum_{i=1}^{n-1}\mathbf K(x,y,e_i),
\]
with \(\{e_i\}\) a \(g_y\)-orthonormal basis of \(T_xM\) chosen so that \(e_n=y/F(x,y)\).

A Finsler manifold is {\it forward complete} if every geodesic $\gamma:[0,T)\to M$ with $T<+\infty$ admits an extension to $[0,+\infty)$; backward completeness is defined analogously.
Let $\zeta:[0,1]\rightarrow M$ be a Lipschitz continuous path.
The associated {\it distance function}
\[
d_F(x_1,x_2):=\inf_{\zeta}\int_0^1 F(\zeta(t),\dot\zeta(t))\,{\dd}t
\]
is generally asymmetric unless $F$ is reversible.

	The {\it forward} and {\it backward metric balls} centered at $x\in M$ with radii $R>0$ are defined respectively by
\begin{equation*}
	B^+_R(x):=\{z\in M:\, d_F(x,z)<R\},\qquad B^-_R(x):=\{z\in M:\, d_F(z,x)<R\}.
\end{equation*}
Throughout this paper, unless otherwise specified, \(B_R(x):=B_R^+(x)\) denotes the forward metric ball.

Let $\mathfrak m$ be a smooth positive measure on $M$, written in local coordinate system as
\[{\dd}\mathfrak{m}=\sigma(x){\dd}x^1\wedge\ldots\wedge {\dd}x^n,\]
where $\sigma= \sigma(x)$ denotes the {\it density function} of $\dm$.
 A triple $(M,F,\mathfrak m)$ will be called a {\it Finsler metric measure manifold} ($\FMMM$) if $(M,F)$ is forward complete and endowed with smooth positive measure $\m$.

For a \(\FMMM\) $(M,F,\mathfrak m)$, the {\it distortion} is defined as
\begin{equation*}
\tau(x,y):=\ln \frac{\sqrt{\det g_{ij}(x,y)}}{\sigma(x)}, \qquad  \forall y\in T_xM\backslash\{0\},
\end{equation*}
The associated {\it S-curvature} is the derivative of \(\tau\) along geodesics:
\begin{equation}\label{Scurvaturedef}
\mathbf S(x,y):=\left.\frac{{\dd}}{{\dd}t}\right|_{t=0}\tau(\gamma_y(t),\dot\gamma_y(t)).
\end{equation}
Equivalently, by Shen \cite[Lemma 3.2.3]{ShenLecture},
$
\mathbf S(x,y)=\frac{\partial G^i}{\partial y^i}(x,y)-y^i\frac{\partial}{\partial x^i}\ln \sigma(x).
$
We say that \(\mathbf S\ge h\) if
\[ \mathbf S(x,y)\geq h\,F(x,y), \quad \forall (x,y)\in TM\backslash\{0\}. \]

We now introduce polar coordinates. Fix \(o\in M\) and let \(S_oM:=\{z\in T_oM:F(o,z)=1\}\). For \(y\in S_oM\), define
\[
i_y:=\sup\{t>0:\gamma_y|_{[0,t]} \text{ is globally minimizing}\},\qquad \mathfrak i_o:=\inf_{y\in S_oM}i_y>0,
\]
where \(\gamma_y\) is the unit-speed geodesic starting at \(o\) with \(\dot\gamma_y(0)=y\). A point \(x\in M\) is represented in polar coordinates \((r,y)\) by \(r=r(x)=d_F(o,x)\) and \(y=y(x)\in S_oM\), with
 \begin{equation}\label{geommeaingofr}
r\leq i_y,\quad\gamma_y(r)=(r,y),\quad \dot{\gamma}_y(r)=\nabla r|_{(r,y)}.
\end{equation}
By \cite[Lemma 3.2.3]{ShenLecture}, this representation satisfies
\[
F(\nabla r)=F^*({\dd}r)=1 \quad \text{for }  \mathfrak{m} \text{-almost every }  x\in M.
\]
In these coordinates, the measure decomposes as
\begin{equation}\label{volumeexprsso}
\dm|_{(r,y)}=:\hat{\sigma}_o(r,y)\,{\dd}r \wedge {\dd}\nu_o(y),
\end{equation}
where $\hat{\sigma}_o(r,y)$ is the density function and  \({\dd}\nu_o\) is the volume form on the indicatrix \(S_oM\). Hence, for any \(f\in L^1(M,\mathfrak m)\),
\[
\int_M f\,{\dd}\mathfrak m=\int_{S_oM}{\dd}\nu_o(y)\int_0^{i_y}f(r,y)\,\hat\sigma_o(r,y)\,{\dd}r.
\]
We shall use the {\it integral of distortion} (cf.~Huang--Krist\'aly--Zhao \cite{HKZ})
\begin{equation}\label{integralofdistortion}
\mathscr{I}_{\m}(o):=\int_{S_oM}e^{-\tau(o,y)}\,{\dd}\nu_o(y),
\end{equation}
which is finite. In our proof of the results in this paper, we mainly use $\mathscr{I}_{\m}(\mathbf{0})$, $\tau(\mathbf{0},y)$, and $\nu_{\mathbf{0}}(y)$ since the distance function $r$ in the polar coordinates here is from the origin $\mathbf{0}$.

The volume comparison theorem of Zhao--Shen \cite[Theorem 3.6]{ZS} will be needed. Set
\[
\mathfrak s_k(t):=
\begin{cases}
\frac{\sin(\sqrt{k}t)}{\sqrt{k}}, & k>0,\\[3pt]
t, & k=0,\\[3pt]
\frac{\sinh(\sqrt{-k}t)}{\sqrt{-k}}, & k<0.
\end{cases}
\]
\begin{theorem}[\cite{ZS}]\label{bascivolurcompar} Let $(M,F,\m)$ be an $n$-dimensional $\FMMM$. If $\mathbf{Ric}\geq (n-1)k$ for some $k \in \mathbb{R}$, then for every $o\in M$ and $y\in  S_oM$, the function
\begin{equation} \label{densesim}
H_y(r):=\frac{\hat{\sigma}_o(r,y)}{ e^{-\tau\big(\gamma_y(r),\dot{\gamma}_y (r)\big)}  \mathfrak{s}_k^{n-1}(r)},\qquad \forall  r\in (0,i_y),
\end{equation}
is monotonically decreasing in $r$ and satisfies $\lim_{r \to 0^+} H_y(r) = 1$,
where $(r,y)$ is the polar coordinate system around $o$ and $\gamma_y(t)$ is a geodesic with $\dot{\gamma}_y(0)=y$.
\end{theorem}

\begin{remark}\label{rem:volume_comparison_equality}
In the equality case of the Zhao--Shen volume comparison theorem \cite[Theorem 3.6]{ZS}, which occurs when $\mathbf K=0$, the density function satisfies the exact identity
\begin{equation}\label{simga_0}
\hat{\sigma}_{\mathbf{0}}(r,y)= e^{-\tau(\gamma_y(r),\dot{\gamma}_y(r))} r^{n-1}.
\end{equation}
This will be used in Lemma \ref{lem_hat_sigma} to derive the two-sided estimate \eqref{hat_sigma}.
\end{remark}

Since \(S_oM\) is compact and $\lim_{r \to 0^+} H_y(r) = 1$, under the assumption of the above theorem, there exists \(\epsilon\in(0,\min\{1,\mathfrak i_o\})\) such that
\begin{equation}\label{limcaes}
2^{-1} e^{-\tau(o,y)}r^{n-1}\leq \hat{\sigma}_o(r,y)\leq 2 e^{-\tau(o,y)}r^{n-1}, \quad \forall r\in (0,\epsilon), \ y\in S_oM.
\end{equation}

Let \(p\in[1,+\infty)\). For \(u\in C_0^\infty(M)\), define the
Sobolev asymmetric norm by
\begin{equation*}
	\|u\|_{W_{\mathfrak m}^{1,p}}
	:=
	\left(
	\int_M |u|^p\,{\dd}\mathfrak m
	\right)^{1/p}
	+
	\left(
	\int_M F^p(x,\nabla u)\,{\dd}\mathfrak m
	\right)^{1/p}.
\end{equation*}
The Sobolev space \(W_0^{1,p}(M,F,\mathfrak m)\) is defined as the
closure of \(C_0^\infty(M)\) with respect to the backward topology, i.e.,
\[
u\in W^{1,p}_0(M,F,\m) \quad \Longleftrightarrow \quad \exists \, (u_k) \subset C^\infty_0(M) \text{ \ with }\lim_{k\rightarrow \infty}\|u-u_k\|_{W^{1,p}_{\m}}= 0.
\]
Since $F^*(x,{\dd}f)=F(x,\nabla f),$
the Sobolev space  admits the norm
\begin{align}\label{Sobolev_pseudonorm}
\|u\|_{W^{1,p}_{\mathfrak m}}
& = \left(
		\int_M |u|^p\,{\dd}\mathfrak m
		\right)^{1/p}
		+
		\left(
		\int_M F^{*p}({\dd}u)\,{\dd}\mathfrak m
		\right)^{1/p}.
\end{align}
Although this norm behaves linearly when \(F\) is Riemannian, in the Finsler setting the gradient map is nonlinear; it is therefore often preferable to work with the differential. The norm \(\|\cdot\|_{W^{1,p}_{\mathfrak m}}\) satisfies the usual axioms:
\[
\|u\|_{W^{1,p}_{\mathfrak m}}\ge0,\quad \|\lambda u\|_{W^{1,p}_{\mathfrak m}}=\lambda\|u\|_{W^{1,p}_{\mathfrak m}}\ (\lambda\ge0),\quad
\|u+v\|_{W^{1,p}_{\mathfrak m}}\le \|u\|_{W^{1,p}_{\mathfrak m}}+\|v\|_{W^{1,p}_{\mathfrak m}}, \quad  \forall u,v\in W_0^{1,p}.
\]

A function \(f:U\to\mathbb R\) is called {\it Lipschitz} if there exists \(C\geq 0\) such that
\[
|f(x_1)-f(x_2)|\leq  C\,d_F(x_1,x_2),\quad \forall x_1,x_2\in U.
\]
The smallest such constant is the {\it Lipschitz constant} $\dil(f)$. We denote by \(\Lip_0(U)\) the space of compactly supported Lipschitz functions on \(U\). 

The following standard approximation result (see \cite[Proposition 3.3]{KLZ}) will be used in proving Lemma~\ref{lemma:Cpab_well}.
\begin{lemma}\cite{KLZ}\label{lipsconverppax}
Let $(M,F,\m)$ be an $\FMMM$. For every $u\in \Lip_0(M)$,   there exists a sequence of smooth Lipschitz functions $(u_k)\subset C^\infty_0(M)\cap \Lip_0(M)$  such that
\begin{align*}
&\supp u\cup \supp u_k\subset \mathscr{K},\quad |u_k|\leq C,\quad F^*({\dd} u_k)\leq   \dil(u_k)\leq C,\\
&\quad u_k \rightrightarrows u, \qquad \|u_k-u \|_{W^{1,p}_{\m}}+\|u-u_k \|_{W^{1,p}_{\m}}\rightarrow 0,
\end{align*}
where $\mathscr{K}\subset M$ is a compact set,  $C>0$ is  a constant and ``$\rightrightarrows$" means uniform convergence. In particular, if $u$ is nonnegative, so are the \(u_k\).
\end{lemma}

\subsection{Background of projectively flat Finsler manifolds with $\mathbf{K} =0$}\label{subsectionK=0}
In this subsection we collect the classification results for projectively flat Finsler manifolds with $\mathbf K=0$ from \cite{LiZhao_Hilbert4}, together with two auxiliary lemmas needed in the proofs of our main theorems.

A Finsler metric is called {\it projectively flat} if, locally, its geodesics can be represented as straight lines in some coordinate system. In this case, the spray coefficients take the special form
\[
G^i(x,y)=P(x,y)y^i,
\]
where the scalar function \(P(x,y)\), homogeneous of degree one in \(y\), is the {\it projective factor}.

Let \(\Omega\subset\mathbb R^n\) be a simply connected domain containing the origin, and let \((\Omega,F)\) be a Finsler manifold.
 Define
\begin{equation}\label{psiphi}
\psi(y):=F(\mathbf{0},y),\qquad \phi(y):=P(\mathbf{0},y),
\end{equation}
where \(P = F_{x^i} y^i /(2 F) \). The following classification result of Li and Zhao characterizes all forward complete projectively flat Finsler manifolds with vanishing flag curvature.

\begin{theorem}\cite[Theorem 1.2]{LiZhao_Hilbert4} \label{thmK=0globalintro}
	Let $(\Omega,F)$ be a Finsler manifold and set $\psi(y), \phi(y)$ as in \eqref{psiphi}.
Then $(\Omega,F)$ is forward complete, projectively flat, and has $\mathbf K=0$ if and only if the following hold:
	\begin{enumerate}[\rm (i)]
		\item\label{K=0basciciont}  $\phi$ is non-negative, $\Omega = \big\{ x \in \mathbb{R}^n \,:\, \phi(x) < 1 \big\}$ is either bounded or $\mathbb{R}^n$, the projective factor $P$ is the unique solution to
		\begin{equation}\label{P_phi}
			P(x,y)=\phi(y+xP(x,y)),\quad \forall (x,y)\in \Omega\times \mathbb{R}^n,
		\end{equation}
		and the Finsler metric $F$ is given by
		\begin{equation}\label{K=0region}
			F(x,y) = \psi ( y + x P(x,y)) \,\left\{ 1 +   x^k  P_{y^k}(x,y)\right\},    \quad \forall (x,y)\in \Omega\times \mathbb{R}^n;
		\end{equation}
		
		\item\label{K=0additoncomp} exactly one of the following occurs:
		\begin{enumerate}[{\rm (1)}]
			\item  \label{globalK=0(i)}
			$\phi=0$, in which case $(\Omega,F)=(\mathbb{R}^n,\psi)$ is a Minkowski space (backward complete with $\lambda_F(\Omega)<+\infty$);

			\item \label{globalK=0(ii)}
			$\phi$ is a weak Minkowski norm, in which case
			$(\Omega, F)$ is not backward complete with $\lambda_F(\Omega)=+\infty$ and particularly,  $\Omega$ is a bounded domain in $\mathbb{R}^n$ and $P$ is the unique weak Funk metric on $\Omega$.
		\end{enumerate}
	\end{enumerate}
\end{theorem}
\begin{remark}
By this theorem, the metric $F$ is generated and determined by $\psi$ and $\phi$.
Except for Theorem~\ref{main_thmgeneral1}, in the rest of this paper we focus exclusively on the forward complete but not backward complete case, i.e. case \eqref{globalK=0(ii)} in Theorem \ref{thmK=0globalintro}.
\end{remark}

The following lemma is a consequence of Theorem \ref{thmK=0globalintro}; it will be used in the proof of Theorem  \ref{main_thmgeneral1} and in the sharpness argument.
\begin{lemma}\label{maxF}
	Let $(\Omega,F)$ be an $n$-dimensional forward complete projectively flat Finsler manifold  with $\mathbf{K}=0$.
	Define $\psi(y)$ and $\phi(y)$ as in \eqref{psiphi}. Then
\begin{align} \label{maxF(x,x,-x)}
	\max \{F(x,x), F(x,-x) \} = \max \left\{ \frac{\psi(x)}{[1-\phi(x)]^2}, \frac{\psi(-x)}{[1+\phi(-x)]^2} \right\} \leq  \frac{\lambda_\psi \psi(x)}{[1-\phi(x)]^2}, \end{align}
where $\lambda_{\psi}$ is the reversibility of $\psi$.
\end{lemma}
\begin{proof}
	By Theorem~\ref{thmK=0globalintro}\eqref{K=0basciciont} and the non-negativity of
	\(\phi\), we have \(P(x,y)\geq0\) for every
	\((x,y)\in\Omega\times\mathbb{R}^n\). We first prove that
	\(P(x,-x)<1\).
	
	Indeed, if \(P(x,-x)=1\), then
	\[
	1=P(x,-x)
	=\phi\bigl(-x+xP(x,-x)\bigr)
	=\phi(\mathbf{0})=0,
	\]
	which is impossible. If \(P(x,-x)>1\), then the positive
	homogeneity of \(\phi\) gives
	\[
	P(x,-x)
	=\phi\bigl(-x+xP(x,-x)\bigr)
	=\bigl(P(x,-x)-1\bigr)\phi(x).
	\]
	Consequently,
	\[
	\phi(x)
	=\frac{P(x,-x)}{P(x,-x)-1}>1,
	\]
	which contradicts \(x\in\Omega\). Therefore,
	\[
	0\leq P(x,-x)<1.
	\]

	Using Theorem \ref{thmK=0globalintro}\eqref{K=0basciciont} again and the positive homogeneity of $\phi$
we have
\begin{align*}
	P(x,x) & = \phi(x + x P(x,x) ) = (1+P(x,x)) \phi(x), \\
	P(x,-x) & = \phi(-x + x P(x,-x) )  = (1 - P(x,-x)) \phi(-x),
\end{align*}
Hence,
\begin{align}\label{P(x,x)_P(x,-x)}
	P(x,x) = \frac{\phi(x)}{1 - \phi(x)}, \qquad P(x,-x) = \frac{\phi(-x)}{1 + \phi(-x)}.
\end{align}

Using \eqref{K=0region},  $x^k P_{y^k}(x,y)|_{y=x} =  P(x,x)$, $x^k P_{y^k}(x,y)|_{y=-x} = -  P(x,-x)$,  \eqref{P(x,x)_P(x,-x)}, and
the homogeneity of $\psi$,
 we have
\begin{align*}
	F(x,x) & = \psi ( x + x P(x,x)) \left\{ 1  +  P(x,x)\right\} = [1  +  P(x,x)]^2 \psi(x) = \frac{\psi(x)}{[1-\phi(x)]^2}, \\
	F(x,-x) & = \psi ( -x + x P(x,-x)) \left\{ 1  -  P(x,-x)\right\} = [1  -  P(x,-x)]^2 \psi(-x) = \frac{\psi(-x)}{[1+\phi(-x)]^2}.
\end{align*}
Hence,
	\[ \max \{F(x,x), F(x,-x) \} = \max \left\{ \frac{\psi(x)}{[1-\phi(x)]^2}, \frac{\psi(-x)}{[1+\phi(-x)]^2} \right\}
\leq \frac{\max \{ \psi(x),\psi(-x) \}}{[1-\phi(x)]^2}. \]
Since \(\lambda_\psi \geq 1\) is the reversibility of \(\psi\), we have \(\psi(-x)\le\lambda_\psi\psi(x)\).  It yields \eqref{maxF(x,x,-x)}. This completes the proof.
\end{proof}

We shall use the following two results from \cite{LiZhao_Hilbert4}: the distance formula and the gradient/$S$-curvature estimates.
\begin{theorem}\cite[Theorem 3.6]{LiZhao_Hilbert4} \label{K=0lemdx_0xdxx_0}Let $(\Omega,F)$ be a projectively flat Finsler manifold  with $\mathbf{K}=0$, where $\Omega\subset \mathbb{R}^n$ is a  strictly convex domain containing the origin. Define $\psi(y)$ and $\phi(y)$ as in \eqref{psiphi}.
	Then  for any $x_1,x_2\in \Omega$,
	\begin{equation*}
		d_F(x_1,x_2) =  \frac{F(x_1, x_2-x_1)}{1-P(x_1, x_2-x_1)}.
	\end{equation*}
	In particular, this implies
	\begin{equation}\label{coreqd0xdx0}
		d_F(\mathbf{0},x) =  \frac{\psi(x)}{1-\phi(x)}, \qquad d_F(x, \mathbf{0})= \frac{\psi(-x)}{1+ \phi(-x)}.
	\end{equation}
\end{theorem}

\begin{lemma}\label{K=0nabla} 
	Let $(\Omega,F)$ be an $n$-dimensional forward complete projectively flat Finsler manifold  with $\mathbf{K}=0$.
	Define $\psi(y)$ and $\phi(y)$ as in \eqref{psiphi}.  Then the following hold:
	\begin{enumerate}[\rm (i)]
		\item \label{nabla_r}
		the gradient of $r(x)$  is given by
		\begin{equation}\label{eqnabla_r}
			\nabla r(x) = \frac{  \psi(x) }{(\psi(x) + r(x) \phi(x))^2} x  = \frac{[1 - \phi(x)]^2}{\psi(x)}x,\quad  \forall x\in \Omega\backslash\{\mathbf{0}\};
		\end{equation}
		\item \label{K=0_S}
		the $S$-curvature $\mathbf{S}$ with respect to the Lebesgue measure $\mathscr{L}^n$ satisfies
		$\mathbf{S} =  (n+1) P\geq 0$. In particular,
		\begin{equation}\label{Socntk=0}
			0 \leq \mathbf{S}(x,\nabla r(x)) = (n+1) \frac{\phi(x)}{r(x)}, \quad \forall x\in  \Omega\backslash \{\mathbf{0}\}.
		\end{equation}
	\end{enumerate}
\end{lemma}
\begin{proof}
If $\phi\equiv 0$, then by Theorem~\ref{thmK=0globalintro}, $(\Omega,F)=(\mathbb{R}^n,\psi)$ is a Minkowski space. In this case $P\equiv 0$, $F(x,y)=\psi(y)$, and $r(x)=\psi(x)$. We first prove the gradient formula. Set $V(x):=x/\psi(x)$. Then $\psi(V(x))=1$. By Euler's theorem for the positively homogeneous function $\psi$, for any $X\in\mathbb{R}^n$,
	\[
	g_{V(x)}(V(x),X)
	= g_{ij}(V(x)) V^i(x) X^j
	= \psi(V(x)) \psi_{y^j}(V(x)) X^j
	= \psi_{y^j}(V(x)) X^j.
	\]
	Since $\psi_{y^j}$ is positively homogeneous of degree zero,
	$\psi_{y^j}(V(x)) = \psi_{y^j}(x).$
	Hence
	\[
	g_{V(x)}(V(x),X)
	= \psi_{y^j}(x) X^j
	= {\dd}\psi(x)(X)
	= {\dd}r(x)(X).
	\]
	By the definition of the Legendre transformation $\mathfrak{L}$,
	\[
	\mathfrak{L}(x,V(x))(X)
	= g_{V(x)}(V(x),X)
	= {\dd}r(x)(X),
	\]
	so
	\[
	\mathfrak{L}(x,V(x)) = {\dd}r(x).
	\]
	Since $\mathfrak{L}$ is a diffeomorphism from $T_xM\backslash\{0\}$ to $T_x^*M\backslash\{0\}$, we conclude
	\[
	V(x) = \mathfrak{L}^{-1}({\dd}r(x)) = \nabla r(x).
	\]
	This proves \eqref{eqnabla_r} when $\phi\equiv 0$. The $S$-curvature formula follows from $P\equiv 0$, hence $\mathbf{S}=(n+1)P=0$, and the right-hand side of \eqref{Socntk=0} is also zero since $\phi\equiv 0$.

	Now assume $\phi\not\equiv 0$. Then by Theorem~\ref{thmK=0globalintro}, the manifold is not backward complete. The proof of \eqref{eqnabla_r} and \eqref{Socntk=0} in this case is exactly the same as in \cite[Lemma 6.6]{LiZhao_Hilbert4}. For the reader's convenience, we sketch the main steps. From \eqref{coreqd0xdx0}, we have $r=\psi/(1-\phi)$, or equivalently $\psi=r(1-\phi)$. Using the expression \eqref{K=0region} for $F$ and the homogeneity of $\psi$ and $\phi$, one computes
	\[
	\nabla r(x)=\frac{[1-\phi(x)]^2}{\psi(x)}x,
	\]
	which is the second equality in \eqref{eqnabla_r}; the first equality follows from $r=\psi/(1-\phi)$. The $S$-curvature formula $\mathbf{S}=(n+1)P$ is standard for projectively flat metrics (see \cite[Lemma 6.6]{LiZhao_Hilbert4}), and evaluating it at $\nabla r$ using $P(x,\nabla r(x))=\phi(x)/r(x)$ gives \eqref{Socntk=0}. This completes the proof.
\end{proof}
\begin{remark}\label{rem:K=0nabla_general}
	In \cite[Lemma 6.6]{LiZhao_Hilbert4}, the above lemma was stated under the additional assumption that the manifold is forward complete but not backward complete. However, the proof there does not use the backward incompleteness; it only relies on the classification theorem (Theorem~\ref{thmK=0globalintro}) and the homogeneity of $\psi$ and $\phi$. In the Minkowski case $\phi\equiv 0$, the formulas reduce to the standard identities $\nabla r=x/\psi(x)$ and $\mathbf S=0$, which we verified directly above. Therefore the lemma remains valid in the more general setting stated here.
\end{remark}

The following estimates for the $S$-curvature will be used repeatedly in the integral estimates of  Sections~\ref{sec:sharp}, \ref{SharpHPW}, and \ref{section:K=0}.
\begin{lemma}\label{lem:S_estimates}
Let $(\Omega,F)$ be a forward complete but not backward complete projectively flat Finsler manifold with $\mathbf{K}=0$, and let $\psi,\phi$ be as in \eqref{psiphi}.
Let
\begin{equation} \label{rhomaxmin}
\bar{\varrho}:=\sup_{y\in \mathbb{R}^n\backslash\{\mathbf{0}\}}\frac{\psi(y)}{\phi(y)},\qquad
\hat{\varrho}:=\inf_{y\in \mathbb{R}^n\backslash\{\mathbf{0}\}}\frac{\psi(y)}{\phi(y)}.
\end{equation}
Then \(0<\hat{\varrho}\le\bar{\varrho}<+\infty\), and
\begin{align} \label{phi_rho_minmax}
	\frac{r}{\bar{\varrho} + r} \leq \phi \leq \frac{r}{\hat{\varrho} + r}.
	\end{align}
Consequently,
\begin{equation}\label{S>}
	\frac{n+1}{\bar{\varrho} + r(x)} \leq \mathbf{S}(x,\nabla r(x)) \leq \frac{n+1}{\hat{\varrho} + r(x)}.
	\end{equation}
\end{lemma}
\begin{proof}
Since the manifold is not backward complete, the classification theorem (Theorem \ref{thmK=0globalintro}) ensures 
$0 < \phi(y) <1$ for all $y\in \Omega\backslash\{\mathbf{0}\}$, so the constants $\bar{\varrho}$ and $\hat{\varrho}$ are well-defined. By the positive homogeneity of \(\psi\) and \(\phi\), the ratio \(\psi(y)/\phi(y)\) is independent of the scaling of \(y\), hence depends only on the direction. Since the unit sphere is compact, the supremum and infimum are finite and positive, so \(0<\hat\varrho\le\bar\varrho< +\infty\).

From \eqref{coreqd0xdx0}, we have \(r=\psi/(1-\phi)\), or equivalently \(\psi=r(1-\phi)\). Substituting this into the inequality \(\hat\varrho\phi\le\psi\le\bar\varrho\phi\) yields
\[
\hat\varrho\phi\le r(1-\phi)\le \bar\varrho\phi.
\]
The right-hand inequality \(r(1-\phi)\le\bar\varrho\phi\) gives
$r\le(\bar\varrho+r)\phi$,
hence
$
\phi\ge  r/({\bar\varrho+r}).
$
The left-hand inequality \(\hat\varrho\phi\le r(1-\phi)\) gives
$
(\hat\varrho+r)\phi\le r,
$
hence
$
\phi\le  {r}/{\hat\varrho+r}.
$
Since $0 \leq \phi < 1$, dividing by \(\bar\varrho+r>0\) is legitimate.
Thus \eqref{phi_rho_minmax} is proved.

Substituting \eqref{phi_rho_minmax} into \(\mathbf S(x,\nabla r(x))=(n+1)\phi(x)/r(x)\) from \eqref{Socntk=0} gives
\eqref{S>}.
This completes the proof.
\end{proof}

\vskip 10mm
\section{Hardy-type inequalities and Sharp constant for the generalized CKN inequality}\label{sec:sharp}

This section is devoted to the proof of  the weighted Hardy-type inequality and the generalized CKN inequality with sharp constant. The main results are Theorem~\ref{main_thmgeneral1}, Theorem~\ref{thm:CKNHPW_threshold}\eqref{1.2(i)}, and  Theorem~\ref{thm:CKNHPW_sharp}\eqref{1.3(i)}.

The section is organized as follows. In \S\ref{sec:aux} we collect auxiliary lemmas needed throughout this section and next two sections. In \S\ref{sec:sharp_Hardy} we first prove the weighted Hardy-type inequality (Theorem~\ref{main_thmgeneral1}) and its sharpness when $\lambda_{\psi}=1$ (Corollary~\ref{maincorollary}) on projectively flat zero-curvature Finsler manifolds, and give the  proof of Theorem~\ref{thm:CKNHPW_threshold}\eqref{1.2(i)}. In \S\ref{sec:sharp_CKN} we apply H\"older's inequality  and establish the sharp constants for the generalized CKN inequality (Theorem~\ref{thm:CKNHPW_sharp}\eqref{1.3(i)}) based on Theorems~\ref{sbeerHolder}.

\subsection{Auxiliary estimates and functionals}\label{sec:aux}
For the fixed origin $\mathbf{0}$, we write $\hat{\sigma}_{\mathbf{0}}(r,y)$ in place of $\hat{\sigma}_{o}(r,y)$ when $o = \mathbf{0}$.
\begin{lemma}\label{lem_hat_sigma}
Let $(\Omega, F, \mathscr{L}^n)$ be a forward complete but not backward complete projectively flat Finsler metric manifold with $\mathbf{K}=0$  endowed with the Lebesgue measure.
Let $(r,y)$ be the polar coordinate system around $\mathbf{0}$. Then the measure density function $\hat{\sigma}_{\mathbf{0}}(r,y)$  satisfies
\begin{align} \label{hat_sigma}
{\hat{\varrho}}^{n+1} e^{-\tau(\mathbf{0},y)} r^{n-1}\,(r+\hat{\varrho})^{-(n+1)}
\leq  \hat{\sigma}_{\mathbf{0}}(r,y)  \leq {\bar{\varrho}}^{n+1} e^{-\tau(\mathbf{0},y)} r^{n-1}\,(r+\bar{\varrho})^{-(n+1)},
\end{align}
where $\hat{\varrho}$ and $\bar{\varrho}$ are defined in \eqref{rhomaxmin}.
Moreover,
let $B_{t_1}(\mathbf{0})$ and $B_{t_2}(\mathbf{0})$ be two geodesic balls with radius $ 0 \leq t_1 \leq t_2 \leq +\infty$, then
for any non-negative radial function $f = f(r) \in L^1(\Omega,\mathscr{L}^n)$,
\begin{align} \label{inte_ine}
\mathscr{I}_{\m}(\mathbf{0})  {\hat{\varrho}}^{n+1} \int_{t_1}^{t_2} f(r)\, \frac{r^{n-1}}{(r+\hat{\varrho})^{n+1}}\, {\dd}r \leq \int_{B_{t_2}(\mathbf{0})\backslash B_{t_1}(\mathbf{0})} f(r)\,{\dd}x \leq  \mathscr{I}_{\m}(\mathbf{0})  {\bar{\varrho}}^{n+1} \int_{t_1}^{t_2} f(r)\, \frac{r^{n-1}}{(r+\bar{\varrho})^{n+1}}\, {\dd}r.
\end{align}
\end{lemma}
\begin{proof}
Let $(r,y)$ be the polar coordinate system around $\mathbf{0}$.
By  \eqref{Scurvaturedef}, \eqref{geommeaingofr}, and \eqref{S>}, we have
\[
\tau(\gamma_y(r), \dot{\gamma}_y(r)) - \tau(\mathbf{0}, y) = \int_0^r \mathbf{S}(\gamma_y(t), \dot{\gamma}_y(t)) \, \mathrm{d}t \geq \int_0^r  \frac{n+1}{\bar{\varrho}+t} \mathrm{d}t =   \ln(\bar{\varrho}+r)^{n+1} - \ln{\bar{\varrho}}^{n+1}.
\]
\[
\tau(\gamma_y(r), \dot{\gamma}_y(r)) - \tau(\mathbf{0}, y) = \int_0^r \mathbf{S}(\gamma_y(t), \dot{\gamma}_y(t)) \, \mathrm{d}t \leq \int_0^r  \frac{n+1}{\hat{\varrho}+t} \mathrm{d}t =   \ln(\hat{\varrho}+r)^{n+1} - \ln{\hat{\varrho}}^{n+1}.
\]
Together with Theorem \ref{bascivolurcompar} and \eqref{simga_0} in Remark \ref{rem:volume_comparison_equality} this yields the key comparison
\begin{align*}
\hat{\sigma}_{\mathbf{0}}(r,y) &=  e^{-\tau(\gamma_y(r),\dot{\gamma}_y(r))} r^{n-1}\leq {\bar{\varrho}}^{n+1} e^{-\tau(\mathbf{0},y)} r^{n-1}\,(r+\bar{\varrho})^{-(n+1)}, \\
	\hat{\sigma}_{\mathbf{0}}(r,y) &=  e^{-\tau(\gamma_y(r),\dot{\gamma}_y(r))} r^{n-1}\geq {\hat{\varrho}}^{n+1} e^{-\tau(\mathbf{0},y)} r^{n-1}\,(r+\hat{\varrho})^{-(n+1)}, \quad \forall  r\in (0,+\infty), \ y\in S_{\mathbf{0}}\Omega.
\end{align*}
This proves \eqref{hat_sigma}.

By \eqref{integralofdistortion}, for any \(f = f(r)\in L^1(\Omega,\mathscr{L}^n)\) and  $ 0 \leq t_1 \leq t_2 \leq +\infty$, there holds
\[
\int_{B_{t_2}(\mathbf{0})\backslash B_{t_1}(\mathbf{0})} f(r)\,{\dd}x =\int_{S_{\mathbf{0}}M}{\dd}\nu_\mathbf{0}(y)\int_{t_1}^{t_2} f(r)\,\hat\sigma_\mathbf{0}(r,y)\,{\dd}r,
\]
which together with \eqref{hat_sigma} and $f\geq 0$
yields \eqref{inte_ine} and hence completes the proof.
\end{proof}

\begin{lemma}\label{lemma_u^p_r^p}
Let $(\Omega, F, \mathscr{L}^n)$ be a forward complete but not backward complete projectively flat Finsler metric manifold with $\mathbf{K}=0$ and the generating Minkowski norm $\psi$ reversible, endowed with the Lebesgue measure.
	Let  $a > p >1$, $s>2$, and $0 \leq b < 1$ satisfying
\begin{align}\label{condition_b*}
b< \frac{p+s-1}{p s} - \frac{(a-p)(s-2)}{aps}.
\end{align}
Define
\begin{equation}\label{func_h}
u_{\kappa} = - \eta \bigl(\kappa +r^{\frac{p+s-1-bps}{p-1}}\bigr)^{\frac{1-p}{a-p}},
\end{equation}
 where
 \[\eta :=  \left[ \frac{p+s-1-bps}{a-p} \right]^{\frac{p-1}{a-p}}, \quad \kappa := R^{\frac{p+s-1-bps}{p-1}}, \quad R> \max\{ 2,  \hat{\varrho}\}. \]
  Then the following statements hold$:$
	\begin{enumerate}[\rm (i)]
\item \label{lemma_ui}
\[\int_{\Omega} r^{bps} F^{*p}({\dd}u_\kappa)\, {\dd}x <+\infty; \]
\item \label{lemma_u0}
\begin{align} \label{threeintegral}
 \int_{\Omega} r^{bps} F^{*p}({\dd}u_\kappa)\, {\dd}x  =  \int_{\Omega} |u_\kappa|^{p'(a-1)} r^{(1-b)p's}\, {\dd}x
=
\int_{\Omega} {|u_\kappa|^{a-1}}{r^{s}}  F^*({\dd}u_\kappa)\, {\dd}x;
\end{align}
\item \label{lemma_uii}
there exists some positive constant $C = C(a,b,n,p,s, \hat{\varrho})$ such that
\begin{equation}\label{CKNintlowerboundR}
\begin{split}
\int_{\Omega} r^{bps} F^{*p}({\dd}u_\kappa)\, {\dd}x
& \geq C \mathscr{I}_{\m}(\mathbf{0})\kappa^{\frac{(1-a)p}{a-p}+1}  R^{ s-2};
\end{split}
\end{equation}

\vskip 3mm
\item \label{lemma_uiii}
\begin{equation} \label{eqlemma_uiii}
 \lim_{R\rightarrow +\infty} \frac{ \int_{\Omega \setminus B_{e^{R}}(\mathbf{0})} r^{bps} F^{*p}({\dd}u_\kappa)\, {\dd}x}{ \int_{\Omega} r^{bps} F^{*p}({\dd}u_\kappa) {\dd}x }   = 0;
\end{equation}
\item \label{lemma_uiv}
\begin{equation} \label{eqlemma_uiv}
\lim_{R\rightarrow +\infty} \frac{\int_{B_{e^{-R}}(\mathbf{0})} r^{bps} F^{*p}({\dd}u_\kappa)\, {\dd}x}{ \int_{\Omega} r^{bps} F^{*p}({\dd}u_\kappa) {\dd}x }  = 0.
 \end{equation}
	\end{enumerate}
\end{lemma}
\begin{proof}
Before proving \eqref{lemma_ui}-\eqref{lemma_uiv}, we first do some preparation.
For the assumption   $a > p$ and \eqref{condition_b*} which implies
\begin{equation} \label{condition_b1}
b <  \frac{p+s-1}{p s} - \frac{(a-p) (s-2)}{a p s} <   \frac{p+s-1}{p s},
\end{equation}
it can be verified directly that $u_{\kappa} < 0$ and
\begin{align}\label{u'_kappa}
u'_{\kappa} =  \eta^{\frac{a-1}{p-1}} \bigl(\kappa +r^{\frac{p+s-1-bps}{p-1}}\bigr)^{\frac{1-a}{a-p}} r^{\frac{s-bps}{p-1}}  > 0, \quad \forall r>0.
\end{align}
Together with the assumption on $a$, $b$ and $s$,  it follows that
\[
\frac{s-bps}{p-1} p + bps + n-1 = \frac{(1-b) ps}{p-1} + n-1 >0,
\]
and hence
\begin{align}
(u'_{\kappa})^p r^{bps + n-1} (r+\bar{\varrho})^{-(n+1)}  \leq
\begin{cases}
& \bar{\varrho}^{-(n+1)} \eta^{\frac{(a-1)p}{p-1}}   \kappa^{\frac{(1-a)p}{a-p}}, \quad r\in(0, 1], \\[4pt]
& \eta^{\frac{(a-1)p}{p-1}} r^{ \left( \frac{p+s-1-bps}{p-1} \frac{1-a}{a-p} + \frac{s-bps}{p-1} \right)p + bps -2}, \quad  r\in [1, +\infty).
\end{cases} \label{u'_kappa_r>1}
\end{align}
A direct computation shows that the power of $r$ in the upper bound when $r\in [1, +\infty)$ can be simplified as
\begin{align}
&\quad\left( \frac{p+s-1-bps}{p-1} \frac{1-a}{a-p} + \frac{s-bps}{p-1} \right)p + bps -2
 = \frac{abps}{a-p} - \frac{p+s-1}{a-p}a +s-3 < -1, \label{u'_kappa_r>1negative}
\end{align}
where the last inequality follows directly from \eqref{condition_b*}.
\medskip

\eqref{lemma_ui}
It follows from  \eqref{u'_kappa_r>1} and \eqref{u'_kappa_r>1negative} that
\[ \int_0^{+\infty} (u'_{\kappa})^p r^{bps + n-1} (r+\bar{\varrho})^{-(n+1)} {\dd}r  < +\infty. \]
Then it follows from $F^{*p}({\dd}u_\kappa) = (u'_\kappa)^p F^{*p}({\dd}r) =  (u'_\kappa)^p$ and \eqref{inte_ine} that
\begin{align}
\int_{\Omega} r^{bps} F^{*p}({\dd}u_\kappa) {\dd}x
   \leq \mathscr{I}_{\m}(\mathbf{0}) {\bar{\varrho}}^{n+1}  \int^{+\infty}_{0}  (u'_{\kappa})^p r^{bps + n-1} (r+\bar{\varrho})^{-(n+1)}  {\dd}r < + \infty. \label{u'_kappa_int}
\end{align}
This proves \eqref{lemma_ui}.

\medskip
\eqref{lemma_u0}
A direct computation shows that $u_\kappa$ satisfies
\begin{equation*}
r^{bps} (u'_{\kappa})^{p} =  |u_{\kappa}|^{(a-1)p'} r^{(1-b)p's} = {|u_{\kappa}|^{a-1}}{r^{s}} u'_{\kappa}.
\end{equation*}
Then by $F^{*p}({\dd}u_\kappa) = (u'_\kappa)^p F^{*p}({\dd}r) = (u'_\kappa)^p$, we have \eqref{threeintegral}.

\medskip

\eqref{lemma_uii}
By $F^{*p}({\dd}u_\kappa)  =  (u'_{\kappa})^p$ and \eqref{inte_ine}, we have
\begin{align}
\int_{\Omega} r^{bps} F^{*p}({\dd}u_\kappa) {\dd}x
&\geq  \mathscr{I}_{\m}(\mathbf{0}) {\hat{\varrho}}^{n+1}  \int^{+\infty}_{0}  (u'_{\kappa})^p r^{bps + n-1} (r+\hat{\varrho})^{-(n+1)} {\dd}r \notag \\
& \geq  \mathscr{I}_{\m}(\mathbf{0}) {\hat{\varrho}}^{n+1}  \int^{R}_{\hat{\varrho}}  (u'_{\kappa})^p r^{bps + n-1} (r+\hat{\varrho})^{-(n+1)} {\dd}r.
\label{lem_ukappa_ii}
\end{align}
By \eqref{u'_kappa}, $\kappa = R^{\frac{p+s-1-bps}{p-1}}$, $R>\hat{\varrho}$, and $a>p>1$, for $r\in [\hat{\varrho}, R]$,
\begin{align*}
& \quad (u'_{\kappa})^p r^{bps + n-1} (r+\hat{\varrho})^{-(n+1)} \\
 & \geq  \eta^{\frac{(a-1)p}{p-1}} \bigl(\kappa + \kappa \bigr)^{\frac{(1-a)p}{a-p}} (2r)^{-(n+1)} r^{\frac{(s-bps)p}{p-1} + bps + n-1}
 = : C_1(a,b,n,p,s, \hat{\varrho}) \kappa^{\frac{(1-a)p}{a-p}} r^{\frac{(1-b)ps}{p-1} -2},
\end{align*}
which together with \eqref{lem_ukappa_ii} yields
\begin{align*}
\int_{\Omega} r^{bps} F^{*p}({\dd}u_\kappa) {\dd}x
& \geq  \mathscr{I}_{\m}(\mathbf{0}) {\hat{\varrho}}^{n+1} C_1(a,b,n,p,s, \hat{\varrho}) \kappa^{\frac{(1-a)p}{a-p}} \int^{R}_{\hat{\varrho}} r^{\frac{(1-b)ps}{p-1} -2} {\dd}r \\
& = (\frac{(1-b)ps}{p-1} -1)^{-1} \mathscr{I}_{\m}(\mathbf{0}) {\hat{\varrho}}^{n+1} C_1(a,b,n,p,s, \hat{\varrho}) \kappa^{\frac{(1-a)p}{a-p}} \left[ R^{  \frac{(1-b)ps}{p-1} -1 } - {\hat{\varrho}}^{ \frac{(1-b)ps}{p-1} -1}   \right].
\end{align*}
By the assumption on $b$, we have $\frac{p+s-1-bps}{p-1} >0$ and hence
\[  \frac{(1-b)ps}{p-1} -1 = \frac{(p-1)(s-2) + p+s-1-bps}{p-1} > 0. \]
Then there exists some constant $C:=C(a,b,n,p,s, \hat{\varrho})$ such that
\begin{align*}
\int_{\Omega} r^{bps} F^{*p}({\dd}u_\kappa) {\dd}x
 \geq
 C \mathscr{I}_{\m}(\mathbf{0})\kappa^{\frac{(1-a)p}{a-p}}  R^{ \frac{(1-b)ps}{p-1} -1 }  > 0.
\end{align*}
Since $\kappa = R^{\frac{p+s-1-bps}{p-1}}$ and
 \[ \left( \frac{p+s-1-bps}{p-1}\right) \frac{(1-a)p}{a-p} + \frac{(1-b)ps}{p-1} -1 = \frac{p+s-1-bps}{p-1} \left[ \frac{(1-a)p}{a-p} +1 \right] + s-2, \]
it follows that
\begin{align*}
\int_{\Omega} r^{bps} F^{*p}({\dd}u_\kappa) {\dd}x
&\geq C \mathscr{I}_{\m}(\mathbf{0})\kappa^{\frac{(1-a)p}{a-p}+1}  R^{ s-2}.
\end{align*}
Thus, \eqref{lemma_uii} is proved.

\medskip

\eqref{lemma_uiii}
Again by $F^{*p}({\dd}u_\kappa)  =  (u'_\kappa)^p$ and \eqref{inte_ine}, we have
\begin{align}
\int_{\Omega \setminus B_{e^{R}}(\mathbf{0}) } r^{bps} F^{*p}({\dd}u_\kappa) {\dd}x
&\leq {\bar{\varrho}}^{n+1} \int_{S_{\mathbf{0}}\Omega}  \left( \int^{+ \infty}_{e^{R}}  e^{-\tau(\mathbf{0},y)}  (u'_{\kappa})^p r^{bps + n-1} (r+\bar{\varrho})^{-(n+1)} {\dd}r  \right) {\dd}\nu_{\mathbf{0}}(y) \notag \\
	& =  \mathscr{I}_{\m}(\mathbf{0}) {\bar{\varrho}}^{n+1}  \int^{+\infty}_{e^{R}}  (u'_{\kappa})^p r^{bps + n-1} (r+\bar{\varrho})^{-(n+1)} {\dd}r.\label{u'_kappa_int2}
\end{align}
Together with \eqref{u'_kappa_r>1} and \eqref{u'_kappa_r>1negative}, for $\theta: =\frac{abps}{a-p} - \frac{p+s-1}{a-p}a +s-3 <-1$, it follows that
\begin{align*}
\int_{\Omega \setminus B_{e^{R}}(\mathbf{0}) } r^{bps} F^{*p}({\dd}u_\kappa) {\dd}x
&\leq  \eta^{\frac{(a-1)p}{p-1}} {\bar{\varrho}}^{n+1} \mathscr{I}_{\m}(\mathbf{0})  \int^{+\infty}_{e^{R}}  r^{\theta} {\dd}r \\
&= \eta^{\frac{(a-1)p}{p-1}}  \mathscr{I}_{\m}(\mathbf{0}) {\bar{\varrho}}^{n+1}(-\theta-1)^{-1} e^{(\theta+1)R} .
\end{align*}
Hence, by \eqref{CKNintlowerboundR} and $\theta+1<0$, we have
\[ \lim_{R\rightarrow +\infty} \frac{ \int_{\Omega \setminus B_{e^{R}}(\mathbf{0})} r^{bps} F^{*p}({\dd}u_\kappa)\, {\dd}x}{ \int_{\Omega} r^{bps} F^{*p}({\dd}u_\kappa) {\dd}x }  = 0. \]

\eqref{lemma_uiv}
Similarly, by $F^{*p}({\dd}u_\kappa) =  (u'_\kappa)^p$ and \eqref{inte_ine}, we have
\begin{align}
\int_{B_{e^{-R}}(\mathbf{0})} r^{bps} F^{*p}({\dd}u_\kappa) {\dd}x
&\leq {\bar{\varrho}}^{n+1} \int_{S_{\mathbf{0}}\Omega}  \left( \int^{e^{-R}}_0  e^{-\tau(\mathbf{0},y)}  (u'_{\kappa})^p r^{bps + n-1} (r+\bar{\varrho})^{-(n+1)} {\dd}r  \right) {\dd}\nu_{\mathbf{0}}(y) \notag \\
& =  \mathscr{I}_{\m}(\mathbf{0}) {\bar{\varrho}}^{n+1}  \int^{e^{-R}}_{0}  (u'_{\kappa})^p r^{bps + n-1} (r+\bar{\varrho})^{-(n+1)} {\dd}r.
\end{align}
It together with \eqref{u'_kappa_r>1} for $r<1$ yields
\begin{align*}
\int_{B_{e^{-R}}(\mathbf{0})} r^{bps} F^{*p}({\dd}u_\kappa) {\dd}x
\leq \eta^{\frac{(a-1)p}{p-1}}  \kappa^{\frac{(1-a)p}{a-p}} \mathscr{I}_{\m}(\mathbf{0}) e^{-R},
\end{align*}
which together with \eqref{CKNintlowerboundR} proves \eqref{eqlemma_uiv}.

Therefore, the proof is completed.
\end{proof}

\subsection{Sharp weighted Hardy inequalities on projectively flat zero-curvature Finsler manifolds}\label{sec:sharp_Hardy}

We now prove the weighted Hardy-type inequality stated in Theorem \ref{main_thmgeneral1}. Then the proof of Theorem~\ref{thm:CKNHPW_threshold}\eqref{1.2(i)} and Corollary \ref{maincorollary} are followed.

\begin{proof}[Proof of Theorem \ref{main_thmgeneral1}]
For any $f = f(x) \in C^{\infty}_0 (\Omega) \backslash \{0\}$,
\eqref{eqnabla_r}, \eqref{absolute_dualff*}, and \eqref{maxF(x,x,-x)}  imply
\begin{align*}
\left| \left\langle {\dd}f, \nabla r    \right\rangle  \right|
& = \frac{[1 - \phi(x)]^2}{\psi(x)} \left| \left\langle {\dd}f, x    \right\rangle \right| \leq
\frac{[1 - \phi(x)]^2}{\psi(x)}  F^*(x,{\dd}f) \max\{ F(x,x), F(x,-x)\}\\
 & \leq \lambda_{\psi}  F^*(x,{\dd}f).
\end{align*}
Consequently, writing $F^*({\dd}f)$ for $F^*(x,{\dd}f)$, we have
\begin{align}
\int_{\Omega}  {|f|^{a-1}}{r^{s}} \left| \left\langle {\dd}f, \nabla r    \right\rangle \right| {\dd}x
& \leq  \lambda_{\psi} \int_{\Omega}  {|f|^{a-1}}{r^{s}}  F^*({\dd}f) {\dd}x.  \label{newFs>=2numer}
\end{align}

Let $\dii$ be the Euclidean divergence in $\mathbb{R}^n$.
By using \eqref{coreqd0xdx0}$_1$, \eqref{eqnabla_r} and the homogeneity of $\phi$ and $\psi$, we have
\begin{align}\label{newK=0_Laplace}
\Delta r & : =\dii\circ\nabla r=  \frac{\partial }{\partial x^i} \left[ \frac{[1-\phi(x)]^2}{\psi(x)} x^i \right]
=  \frac{[1-\phi(x)]^2}{\psi(x)} n  - \frac{2\phi(x)[1-\phi(x)]}{\psi(x)} - \frac{[1-\phi(x)]^2}{\psi(x)} \notag\\
& = \frac{[1-\phi(x)]^2}{\psi(x)} (n-1) - \frac{2 [1-\phi(x)]\phi(x)}{\psi(x)} = -\frac{2}{r} + \frac{(n+1)(1-\phi(x))}{ r}.
\end{align}

On the one hand, the divergence theorem together with   $\langle {\dd}r,\nabla r\rangle=1$ and \eqref{newFs>=2numer} implies
\begin{align}\label{newleftberwres}
&\quad \int_{\Omega} {|f|^{a}}{r^{s}}\Delta r {\dd}x + s \int_{\Omega} {|f|^{a}}{r^{s-1}} {\dd}x \notag\\
& =  - \int_{\Omega} \left\langle {\dd}\left( {|f|^{a}}{r^{s}}\right), \nabla r    \right\rangle{\dd}x + s \int_{\Omega} {|f|^{a}}{r^{s-1}} {\dd}x
 \leq a \int_{\Omega}  {|f|^{a-1}}{r^{s}}\left| \left\langle {\dd}f, \nabla r    \right\rangle \right| {\dd}x
\notag\\
&\leq a  \lambda_{\psi} \int_{\Omega}  {|f|^{a-1}}{r^{s}}  F^*({\dd}f) {\dd}x.
\end{align}
On the other hand, it follows from  \eqref{newK=0_Laplace} that
\begin{align}
 \quad \int_{\Omega} {|f|^{a}}{r^{s}} \Delta r {\dd}x  + s \int_{\Omega} {|f|^{a}}{r^{s-1}} {\dd}x
= (s-2) \int_{\Omega}  {|f|^{a}}{r^{s-1}}   {\dd}x  + (n+1) \int_{\Omega} {|f|^{a}} (1-\phi(x)) r^{s-1}  {\dd}x. \label{newCKNs>=2_secondpart}
\end{align}
Then \eqref{newleftberwres} and \eqref{newCKNs>=2_secondpart} imply
\[  \int_{\Omega}  {|f|^{a-1}}{r^{s}}  F^*({\dd}f) {\dd}x \geq \frac{s-2}{a  \lambda_{\psi}}  \int_{\Omega}  {|f|^{a}}{r^{s-1}}   {\dd}x  + \frac{n+1}{a  \lambda_{\psi}}\int_{\Omega} {|f|^{a}} (1-\phi(x)) r^{s-1}  {\dd}x, \]
which together with \eqref{granotwoff} (i.e. $F^*({\dd}f) = F(\nabla f)$) proves \eqref{firsteq}.
\end{proof}

As a direct consequence of Theorem \ref{main_thmgeneral1}, we now can give the proof of  Theorem \ref{thm:CKNHPW_threshold}\eqref{1.2(i)}.
\begin{proof}[Proof of Theorem \ref{thm:CKNHPW_threshold}\eqref{1.2(i)}]
By \eqref{granotwoff}, replace $F(\nabla f)$ by $F^*({\dd}f)$.
Applying the H\"older inequality to the left side of   \eqref{firsteq}, we get
\begin{align} \label{Thm1.2_Holder1}
\left( \int_{\Omega} r^{bps} F^{*p}({\dd}f) {\dd}x \right)^{1/{p}} \left( \int_{\Omega} |f|^{p'(a-1)} r^{(1-b)p's} {\dd}x  \right)^{1/{p'}}
\geq
\int_{\Omega} {|f|^{a-1}}{r^{s}}  F^*({\dd}f) {\dd}x,
\end{align}
which together with the right side of \eqref{firsteq} and $1-\phi(x)\geq 0$ gives
\begin{align*}
&\quad \left( \int_{\Omega} r^{bps} F^{*p}({\dd}f) {\dd}x \right)^{1/{p}} \left( \int_{\Omega} |f|^{p'(a-1)} r^{(1-b)p's} {\dd}x  \right)^{1/{p'}}
 \notag\\
&\geq   \frac{1}{a \lambda_\psi} \left[(s-2) \int_{\Omega}  {|f|^{a}}{r^{s-1}}   {\dd}x  + (n+1) \int_{\Omega} {|f|^{a}} (1-\phi(x)) r^{s-1} {\dd}x\right] \\
&\geq \frac{s-2}{a \lambda_\psi}   \int_{\Omega}  {|f|^{a}}{r^{s-1}}   {\dd}x.
\end{align*}
This completes the proof.
\end{proof}

\begin{corollary}\label{maincorollary}
Under the assumptions of Theorem~\ref{main_thmgeneral1}, with $\lambda_\psi=1$, for any $f\in C_0^\infty(\Omega)$,
\begin{align}\label{ineqmain1_reverpsi}
\int_\Omega |f|^{a-1} r^s F(\nabla f){\dd}x
\ge \frac{1}{a}\left[(s-2)\int_\Omega |f|^a r^{s-1}{\dd}x
+ (n+1)\int_\Omega |f|^a (1-\phi(x)) r^{s-1}{\dd}x\right].
\end{align}
The constant $1/a$ is sharp. Equality is approached by radial functions $f=f(r)\leq 0$ with $f'(r)\geq 0$ and $f\not\equiv 0$.
\end{corollary}
\begin{proof}
Since the  generating Minkowski norm $\psi$ is reversible, we have $\lambda_\psi=1$. Therefore, Theorem \ref{main_thmgeneral1} with $\lambda_\psi=1$ yields exactly the inequality \eqref{ineqmain1_reverpsi}. Thus the inequality itself is already proved.

It remains to show that equality can be achieved. Let $f=f(r)\in C_0^\infty(\Omega)$ be a radial function satisfying
\begin{equation}\label{Coro1.2fcondi}
f(r)\leq 0,\qquad f'(r) \geq 0,\qquad \forall\, r\ge0.
\end{equation}
We claim that such an $f$ attains equality in \eqref{ineqmain1_reverpsi}. Indeed, since $f$ is radial, we have $\langle {\dd}f,\nabla r\rangle = f'(r)$, and because $F^*({\dd}r)=1$ (which follows from the fact that $r$ is the distance function) and $F^*$ is positively 1-homogeneous, we get
\[
F^*({\dd}f) =F^*(f'(r) {\dd}r) =f'(r) F^*({\dd}r) =f'(r).
\]
Also, $|\langle {\dd}f,\nabla r\rangle|=f'(r)$. Hence the estimate
\[
\left|\left\langle {\dd}f,\nabla r\right\rangle\right| \le F^*({\dd}f)
\]
in \eqref{newFs>=2numer} becomes an equality (with $\lambda_\psi=1$), i.e., $\left|\left\langle {\dd}f,\nabla r\right\rangle\right| = F^*({\dd}f)$ and hence
\begin{align*}
\int_{\Omega}  {|f|^{a-1}}{r^{s}} \left| \left\langle {\dd}f, \nabla r    \right\rangle \right| {\dd}x
=  \int_{\Omega}  {|f|^{a-1}}{r^{s}}  F^*({\dd}f) {\dd}x.
\end{align*}
 Consequently, all inequalities in the proof of Theorem \ref{main_thmgeneral1} become equalities. In particular, the divergence theorem computation yields
\[
\int_{\Omega} {|f|^{a}}{r^{s}}\Delta r {\dd}x + s \int_{\Omega} {|f|^{a}}{r^{s-1}} {\dd}x  = -\int_{\Omega} \left\langle {\dd}\bigl(|f|^a r^s\bigr),\nabla r\right\rangle {\dd}x  + s \int_{\Omega} {|f|^{a}}{r^{s-1}} {\dd}x
= a\int_{\Omega} |f|^{a-1} r^s F^*({\dd}f) {\dd}x,
\]
which, upon using the formula for $\Delta r$ in \eqref{newK=0_Laplace}, gives exactly the equality
\[
\int_{\Omega} |f|^{a-1} r^s F^*({\dd}f){\dd}x
= \frac{1}{a} \left[ (s-2) \int_{\Omega} |f|^a r^{s-1}\,{\dd}x
+ (n+1)\int_{\Omega} |f|^a (1-\phi(x)) r^{s-1}\,{\dd}x \right].
\]
Since $F^*({\dd}f)=F(\nabla f)$ by the duality relation \eqref{granotwoff}, this is precisely \eqref{ineqmain1_reverpsi} with equality.

Thus the constant $1/a$ is sharp and is approached by the family of functions satisfying \eqref{Coro1.2fcondi}. This completes the proof of the corollary.
\end{proof}

The above proof shows that equality in \eqref{ineqmain1_reverpsi} is attained whenever a function \(f\in C_0^\infty(\Omega)\) satisfies
\eqref{Coro1.2fcondi}.
We now exhibit such a nontrivial function.

Let \(\chi\in C^\infty(\mathbb{R}) \) be the standard cut-off function defined by
\begin{equation}\label{chi}
\chi(t)=
\begin{cases}
0, & t\le 0,\\[4pt]
\dfrac{e^{-1/t}}{e^{-1/t}+e^{-1/(1-t)}}, & 0<t<1,\\[8pt]
1, & t\ge 1,
\end{cases}
\end{equation}
so that \(0\le\chi\le1\) and all derivatives of \(\chi\) vanish at \(0\) and \(1\). For any \(R>0\), define
\[
\varphi_R(r)=\chi(R+1-r),\qquad r\ge0.
\]
Then \(\varphi_R\in C^\infty([0,+\infty))\), \(\varphi_R(r)=1\) for \(0\le r\le R\), \(\varphi_R(r)=0\) for \(r\ge R+1\), and \(\varphi_R'(r)\le0\) for all \(r\).

Now set
\[
f_R(r)=-\varphi_R(r).
\]
Clearly \(f_R\in C_0^\infty(\Omega)\), \(f_R\le0\), and
\[
f_R'(r)=-\varphi_R'(r)\ge0.
\]
Thus \(f_R\) satisfies the above equality condition, and since \(f_R\not\equiv0\), it provides a nonzero admissible function for which equality in \eqref{ineqmain1_reverpsi} holds.
(Equivalently, one could take a truncated version of \(-e^{-r}\), but the above construction is simpler and directly verifies the required properties.)

\subsection{Sharp constants for the generalized CKN inequality}\label{sec:sharp_CKN}

We next prove that \( \frac{1}{a} \) remains the optimal constant for the inequality derived by applying H\"older's inequality (Theorem~\ref{sbeerHolder}) and then   Theorem~\ref{thm:CKNHPW_sharp}\eqref{1.3(i)} can be proved.

The proof of Theorem \ref{sbeerHolder} is organized as follows. After applying H\"older's inequality and constructing the approximating functions \(u_{\kappa,R}\), we split the integrals over four intervals. The main proof establishes the required limit, deferring the technical estimates on \(I_0,I_1,I_2\) to Lemmas \ref{lem:I0_estimates}, \ref{lem:I1_estimates}, and \ref{lem:I2_estimates}, which are stated and proved at the end of this subsection.

\begin{theorem}\label{sbeerHolder}
	Let $(\Omega, F, \mathscr{L}^n)$ be a forward complete but not backward complete projectively flat Finsler metric manifold with $\mathbf{K}=0$ and  the generating Minkowski norm $\psi$ reversible ($\lambda_\psi=1$), endowed with the Lebesgue measure.
	Let $1<p<n$, $a > p$, $s>2$, and $0 \leq b < 1$ satisfying
\begin{align}\label{condition_b}
b< \frac{p+s-1}{p s} - \frac{(a-p)(p+s-3)}{aps}.
\end{align}
Then for any $f \in C^\infty_0(\Omega)$, there holds
\begin{align}\label{sbeerHolder_ineq}
&\left( \int_{\Omega} r^{bps} F^{p}(\nabla f) {\dd}x \right)^{1/{p}} \left( \int_{\Omega} |f|^{p'(a-1)} r^{(1-b)p's} {\dd}x  \right)^{1/{p'}}
\geq \notag\\
&\qquad \qquad \frac{1}{a} \left[(s-2) \int_{\Omega}  {|f|^{a}}{r^{s-1}}   {\dd}x  + (n+1) \int_{\Omega} {|f|^{a}} (1-\phi(x)) r^{s-1} {\dd}x\right],
\end{align}
where the constant $1/a$ is sharp.
\end{theorem}
\begin{proof}
The inequality \eqref{sbeerHolder_ineq} holds by the proof of Theorem \ref{thm:CKNHPW_threshold}\eqref{1.2(i)} with $\lambda_\psi=1$, where
we used \eqref{Thm1.2_Holder1}, i.e.,
\begin{align} \label{Thm1.2_Holder11}
\left( \int_{\Omega} r^{bps} F^{*p}({\dd}f) {\dd}x \right)^{1/{p}} \left( \int_{\Omega} |f|^{p'(a-1)} r^{(1-b)p's} {\dd}x  \right)^{1/{p'}}
\geq
\int_{\Omega} {|f|^{a-1}}{r^{s}}  F^*({\dd}f) {\dd}x.
\end{align}
Now the remainder is to prove the sharpness of the inequality, i.e. the constant $1/a$ is sharp.

By the standard equality condition for H\"older's inequality, the equality in  \eqref{Thm1.2_Holder11} holds if and only if
the two integrands of the integrals on the left side are proportional, i.e.,
\begin{equation} \label{thm1.2_HolderFunc}
r^{bps} F^{*p}({\dd}f) = c |f|^{p'(a-1)} r^{(1-b)p's},
\end{equation}
where $c$ is a constant.
Under the assumptions   $a > p$, $s>2$, and \eqref{condition_b}, we have the chain of inequalities
\begin{equation} \label{condition_b1_holder}
b < \frac{p+s-1}{p s} - \frac{(a-p)(p+s-3)}{aps} <  \frac{p+s-1}{p s} - \frac{(a-p) (s-2)}{a p s} <   \frac{p+s-1}{p s}.
\end{equation}
Now we choose a function as \eqref{func_h}, i.e.,
\begin{equation*}
u_{\kappa} = - \eta \bigl(\kappa +r^{\frac{p+s-1-bps}{p-1}}\bigr)^{\frac{1-p}{a-p}},
\end{equation*}
where
\begin{align} \label{Holder_R_new}
\eta :=  \left[ \frac{p+s-1-bps}{a-p} \right]^{\frac{p-1}{a-p}},\quad \kappa := R^{\frac{ p+s-1-bps}{p-1}} > 1,   \quad R> \max\{2, \hat{\varrho} \}.
\end{align}
It can be verified directly that $u_{\kappa} < 0$, $u'_{\kappa} > 0$.
A direct computation shows that $u_{\kappa}$ satisfies \eqref{thm1.2_HolderFunc} with $c=1$, i.e.,
\begin{equation} \label{h'_h_1}
r^{bps} (u'_{\kappa})^p =  |u_{\kappa}|^{(a-1)p'} r^{(1-b)p's}.
\end{equation}
Then, $u_{\kappa}$ achieves the equality of \eqref{Thm1.2_Holder1}.
Here, although the theorem can be proved for any positive constant \(\kappa\), for convenience in proving Theorem \ref{thm:CKNHPW_sharp}\eqref{1.3(i)} below, we choose \(\kappa\) as in \eqref{Holder_R_new}.
Although $u_{\kappa}$ is not compactly supported and not smooth at $r=0$, we still can construct a smooth
function with compact support by $u_{\kappa}$ and two cut-off functions as follows.

Let \(\chi\) be the standard cut-off function defined as above in \eqref{chi},
so that \(0\le\chi\le1\) and all derivatives of \(\chi\) vanish at \(0\) and \(1\). For any \(R>2\), define
\begin{align}\label{varphi_eta(r)}
\eta_R(r): =\chi(1- e^R r),\qquad r\ge0,
\end{align}
\begin{align}\label{varphi_R(r)}
\varphi_R(r): =\chi(e^R+1-r),\qquad r\ge0.
\end{align}

Then $\eta_R\in C^\infty([0,+\infty))$, $\eta_R(0)=1$, $\eta_R(r)=0$ for $r\geq e^{-R}$, and $\eta_R'(r)\le0$ for all $r$;
 $\varphi_R\in C^\infty([0,+\infty))$, $\varphi_R(r)=1$ for $0\le r\le e^R$, $\varphi_R(r)=0$ for $r\ge e^R+1$, and $\varphi_R'(r)\leq 0$ for all $r$.

Now set
\begin{equation} \label{h_R}
u_{\kappa, R}(r)= \eta_R(r)u_{\kappa}(0) + (1-\eta_R(r)) \varphi_R(r) u_{\kappa}(r).
\end{equation}
This ensures that $ u_{\kappa, R}\in C_0^\infty(\Omega)$, $u_{\kappa, R}\leq 0$, and
\begin{align*}
	u_{\kappa,R}'(r)= \eta_R'(r) \bigl(u_\kappa(0)-\varphi_R(r)u_\kappa(r)\bigr) +(1-\eta_R(r))\varphi_R'(r)u_\kappa(r)
	+(1-\eta_R(r))\varphi_R(r)u_\kappa'(r).
\end{align*}
Because $\eta_R'\le0$, $\varphi_R'\le0,\ u_{\kappa}\le0$ and $\eta_R\ge0$, $\varphi_R\ge0$, we have $u_{\kappa,R}' \ge0$.
Moreover, a direct computation gives the $i$-th derivative vanishes at $r=0$, i.e.,
\[
u^{(i)}_{\kappa,R}(0)=0,\qquad i\in\mathbb{N}^+,
\]
which remedies the possible non-smoothness of $u_{\kappa}$ at $r=0$.
Then $u_{\kappa, R}$ satisfies \eqref{Coro1.2fcondi} in Corollary \ref{maincorollary}, and hence
$u_{\kappa, R}$ achieves the equality of \eqref{ineqmain1_reverpsi}, i.e.,
\begin{align*}
\int_\Omega |u_{\kappa, R}|^{a-1} r^s F(\nabla u_{\kappa, R}){\dd}x
= \frac{1}{a}\left[(s-2)\int_\Omega |u_{\kappa, R}|^a r^{s-1}{\dd}x
+ (n+1)\int_\Omega |u_{\kappa, R}|^a (1-\phi(x)) r^{s-1}{\dd}x\right].
\end{align*}
Therefore, to prove that $1/a$ is the sharp constant in \eqref{sbeerHolder_ineq}, it suffices to show
\begin{equation}\label{Q_R}
\lim_{R\to +\infty}\frac{(A_R)^{1/p}(\tilde A_R)^{1/p'}}{V_R}=1,
\end{equation}
where
\[
A_R:=\int_\Omega r^{bps}F^{*p}({\dd}u_{\kappa,R})\,{\dd}x,\quad
\tilde A_R:=\int_\Omega |u_{\kappa,R}|^{p'(a-1)}r^{(1-b)p's}\,{\dd}x,\quad
V_R:=\int_\Omega |u_{\kappa,R}|^{a-1}r^sF^*({\dd}u_{\kappa,R})\,{\dd}x.
\]

Split $[0,+\infty)$ into four subintervals:
\begin{equation}\label{CKNinterval}
I_0=[0,e^{-R}],\qquad
I_1=[e^{-R},e^R],\qquad
I_2=[e^R,e^R+1],\qquad
I_3=[e^R+1,+\infty).
\end{equation}
For simplicity, for $i=0, 1, 2, 3$, define
\begin{equation} \label{u_kappaRi}
u_{\kappa, R, i} := u_{\kappa, R}|_{r\in I_i}, \quad u'_{\kappa, R, i} := u'_{\kappa, R}|_{r\in I_i},
\end{equation}
\[
A^{(i)}_R: = \int_{\Omega} |u'_{\kappa, R, i}|^{p}\, r^{bps} {\dd}x,\qquad
\tilde{A}^{(i)}_R: = \int_{\Omega}   |u_{\kappa, R, i}|^{(a-1)p'}\,r^{(1-b)p' s}\,{\dd}x,\qquad
V^{(i)}_R: = \int_{\Omega} { |u_{\kappa, R, i}|^{a-1}}{r^{s}} |u'_{\kappa, R, i}| {\dd}x.
\]
On $I_3$, $u_{\kappa,R}\equiv0$, so the corresponding integrals vanish.
Then
\begin{equation} \label{AAV}
A_R =A^{(0)}_R + A^{(1)}_R + A^{(2)}_R, \quad \tilde{A}_R = \tilde{A}^{(0)}_R + \tilde{A}^{(1)}_R + \tilde{A}^{(2)}_R, \quad V_R =  V^{(0)}_R + V^{(1)}_R + V^{(2)}_R.
\end{equation}
For these integrals, we have the following estimates, whose proofs are given in Lemmas \ref{lem:I0_estimates}, \ref{lem:I1_estimates}, and \ref{lem:I2_estimates} below:
\[
\lim_{R\to +\infty} \frac{ A^{(0)}_R}{A} =\lim_{R\to +\infty}\frac{ \tilde A^{(0)}_R}{A}= \lim_{R\to +\infty} \frac{V^{(0)}_R}{A}=0,\qquad
\lim_{R\to +\infty} \frac{A^{(1)}_R}{A} =\lim_{R\to +\infty} \frac{\tilde A^{(1)}_R}{A} =\lim_{R\to +\infty}  \frac{V^{(1)}_R}{A} = 1,
\]
\[
\lim_{R\to +\infty} \frac{A^{(2)}_R}{A} = \lim_{R\to +\infty} \frac{\tilde A^{(2)}_R}{A} =\lim_{R\to +\infty} \frac{V^{(2)}_R}{A} = 0,
\]
where $A:=\int_\Omega r^{bps}|u_\kappa'|^p{\dd}x$ is finite and positive for fixed $R$ as shown in Lemma \ref{lemma_u^p_r^p}. Combining these estimates and \eqref{AAV} yields
\[
\lim_{R\to +\infty}  \frac{A_R}{A} =  \lim_{R\to +\infty}  \frac{\tilde A_R}{A} = \lim_{R\to +\infty} \frac{V_R}{A} = 1,
\]
and consequently
\[
\lim_{R\to +\infty}\frac{(A_R)^{1/p}(\tilde A_R)^{1/p'}}{V_R}=1.
\]
This proves the required limit and hence the sharpness of $1/a$.

For readability, the three lemmas containing the interval estimates are stated at the end of this subsection.
\end{proof}

We now prove Theorem \ref{thm:CKNHPW_sharp}\eqref{1.3(i)} based on the method and function used in the above theorem.

\begin{proof}[Proof of Theorem \ref{thm:CKNHPW_sharp}\eqref{1.3(i)}]
Note that \eqref{condition_bthm1.4} is stronger than \eqref{condition_b*}, since $p+s-3 > s-2$.
Since the assumptions are the same as those in Theorem \ref{sbeerHolder} and $1-\phi(x) >0$, the inequality \eqref{CKNsharp} holds naturally by
\eqref{sbeerHolder_ineq}.
Hence, we only need to show $(s-2)/a$ is the sharp constant.
To prove it, we still use the same function $u_{\kappa}$ and the two cut-off functions $\eta_R$, $\varphi_R$ to compose $u_{\kappa, R}$ as in \eqref{h_R}.
By the proof of Theorem \ref{sbeerHolder}, the equation \eqref{Q_R}  holds for ${u}_{\kappa, R}(r)$.
By this definition, since $|u_{\kappa}(0)| = \eta \kappa^{\frac{1-p}{a-p}}  \geq |u_{\kappa}(r)|$,  the following estimate holds:
\begin{equation}\label{CKNubound}
|u_{\kappa, R}| =
\begin{cases}
|\eta_R(r)u_{\kappa}(0)| + |(1-\eta_R(r)) u_{\kappa}(r)| \leq \eta \kappa^{\frac{1-p}{a-p}}, &\qquad \text{if } r\in [0, e^{-R}], \\[3pt]
|u_{\kappa}(r)| = \eta \bigl(\kappa +r^{\frac{p+s-1-bps}{p-1}}\bigr)^{\frac{1-p}{a-p}}, &\qquad \text{if } r\in [e^{-R}, e^R],\\[3pt]
|\varphi_R(r) u_{\kappa}(r)| \leq |u_{\kappa}(r)| = \eta \bigl(\kappa +r^{\frac{p+s-1-bps}{p-1}}\bigr)^{\frac{1-p}{a-p}}, &\qquad \text{if } r\in [e^R, e^{R}+1], \\[3pt]
0, &\qquad \text{if } r\in [e^{R}+1, +\infty).
\end{cases}
\end{equation}
Thus it remains to show
\begin{equation}\label{CKNsharptarget}
\lim_{R \rightarrow +\infty}  \frac{ \int_{\Omega} {| {u}_{\kappa, R}|^{a}} (1-\phi(x)) r^{s-1}\, {\dd}x}{ \int_{\Omega}  {| {u}_{\kappa, R}|^{a}} r^{s-1}\,  {\dd}x  }
=0.
\end{equation}
Now we divide the proof into two steps to estimate the two integrals in the target limit \eqref{CKNsharptarget}.

\medskip
\noindent
\paragraph{\emph{Step 1: Numerator estimate.}}
The inequality  \eqref{phi_rho_minmax} implies
\[ \frac{\hat{\varrho}}{\hat{\varrho} +r} \leq  1-\phi(x) \leq   \frac{\bar{\varrho}}{\bar{\varrho} +r}.  \]
Combining with  \eqref{inte_ine}, it follows that
\begin{align}
& \quad \int_{\Omega} {| {u}_{\kappa, R}|^{a}} (1-\phi(x)) r^{s-1}\, {\dd}x \notag \\
& \leq {\bar{\varrho}}^{n+2}  \mathscr{I}_{\m}(\mathbf{0})
\int^{+\infty}_0  {| {u}_{\kappa, R}|^{a}} r^{s+ n-2} (r+\bar{\varrho})^{-(n+2)} {\dd}r
\notag\\
& = {\bar{\varrho}}^{n+2}  \mathscr{I}_{\m}(\mathbf{0})
\left\{ \int^{1}_0 + \int^{R}_1 + \int^{e^R +1}_R   {| {u}_{\kappa, R}|^{a}} r^{s+ n-2} (r+\bar{\varrho})^{-(n+2)} {\dd}r  \right\}.
\label{CKN_key0}
\end{align}
The estimate \eqref{CKNubound} yields that $|u_{\kappa, R}| \leq \eta \kappa^{\frac{1-p}{a-p}}$ for $r\in [0, 1]$, and  $|u_{\kappa, R}| \leq  \eta \bigl(\kappa +r^{\frac{p+s-1-bps}{p-1}}\bigr)^{\frac{1-p}{a-p}}$ for $r\in [1, e^R +1]$. Then
\begin{align}
\int^{1}_0 {| {u}_{\kappa, R}|^{a}} r^{s+ n-2} (r+\bar{\varrho})^{-(n+2)} {\dd}r & \leq \bar{\varrho}^{-(n+2)} \eta^a \kappa^{\frac{(1-p)a}{a-p}}, \label{CKNsharp_1st_1}\\
\int^{R}_1 {| {u}_{\kappa, R}|^{a}} r^{s+ n-2} (r+\bar{\varrho})^{-(n+2)} {\dd}r & \leq  \eta^a \kappa^{\frac{(1-p)a}{a-p}} \int^{R}_1 r^{s-4} {\dd}r. \label{CKNsharp_2nd_1}
\end{align}
For $r\in [R, e^R +1]$, since $ | u_{\kappa, R}| \leq  \eta r^{ -\frac{p+s-1-bps}{a-p}}$, we have
\begin{align}
\int^{e^R+1}_R {| {u}_{\kappa, R}|^{a}} r^{s+ n-2} (r+\bar{\varrho})^{-(n+2)} {\dd}r & \leq  \eta^a \int^{e^R+1}_R  r^{-\frac{(p+s-1-bps)a}{a-p} + s-4} {\dd}r. \label{CKNsharp_3rd}
\end{align}
By a direct computation, the condition on $b$ \eqref{condition_b} is equivalent to
\begin{align*}
 \frac{(p+s-1 - bps)a}{a-p} -  (p+s-3) > 0,
\end{align*}
which together with $p>1$ yields
\[
- \frac{(p+s-1 - bps)a}{a-p} + s-4 < -2.
\]
Then it follows from \eqref{CKNsharp_3rd} and $\kappa^{\frac{(1-p)a}{a-p}} = R^{-\frac{(p+s-1-bps)a}{a-p}}$ that
\begin{align}
&\quad \int^{e^R+1}_R {| {u}_{\kappa, R}|^{a}} r^{s+ n-2} (r+\bar{\varrho})^{-(n+2)} {\dd}r \notag\\
&
\leq  \eta^a \left( \frac{(p+s-1-bps)a}{a-p} - s+3 \right)^{-1} \left[ R^{-\frac{(p+s-1-bps)a}{a-p} + s-3} - \left(e^R+1\right)^{-\frac{(p+s-1-bps)a}{a-p} + s-3}\right] \notag\\
& \leq \eta^a \left( \frac{(p+s-1-bps)a}{a-p} - s+3 \right)^{-1} \kappa^{\frac{(1-p)a}{a-p}} R^{s-3}.
 \label{CKNsharp_3rd_1}
\end{align}
Thus,
combining \eqref{CKNsharp_1st_1}, \eqref{CKNsharp_2nd_1}, and \eqref{CKNsharp_3rd_1}, the inequality \eqref{CKN_key0} implies
\begin{align}
& \quad \int_{\Omega} {| {u}_{\kappa, R}|^{a}} (1-\phi(x)) r^{s-1}\, {\dd}x  \notag \\
& \leq  {\bar{\varrho}}^{n+2}  \mathscr{I}_{\m}(\mathbf{0}) \eta^a \kappa^{\frac{(1-p)a}{a-p}}
\left\{ \bar{\varrho}^{-(n+2)} + \int^{R}_1 r^{s-4} {\dd}r +  \left( \frac{(p+s-1-bps)a}{a-p} - s+3 \right)^{-1} R^{s-3}  \right\}. \label{CKNsharpnumer}
\end{align}

\medskip
\noindent
\paragraph{\emph{Step 2: Denominator estimate.}}
Combining with  \eqref{inte_ine}, it follows that
\begin{align}
 \int_{\Omega}  {| {u}_{\kappa, R}|^{a}} r^{s-1}\,  {\dd}x
&\geq {\hat{\varrho}}^{n+1} \mathscr{I}_{\m}(\mathbf{0}) \int^{+\infty}_0  {| {u}_{\kappa, R}|^{a}} r^{s+n-2} (r+\hat{\varrho})^{-(n+1)} \, {\dd}r
\notag\\
& \geq {\hat{\varrho}}^{n+1} \mathscr{I}_{\m}(\mathbf{0}) \int^{R}_1 {| {u}_{\kappa, R}|^{a}} r^{s+n-2} (r+\hat{\varrho})^{-(n+1)} \, {\dd}r.
\label{CKNsharpdenomi}
\end{align}
Since ${u}_{\kappa, R} = u_{\kappa}$ on $[1, R] \in I_1$ as defined in \eqref{CKNinterval}, and $\kappa = R^{\frac{ p+s-1-bps}{p-1}}$,  we have
\[ | {u}_{\kappa, R}| = |u_{\kappa}| \geq \eta (\kappa + \kappa)^{\frac{1-p}{a-p}} = \eta  (2\kappa)^{\frac{1-p}{a-p}}, \quad \forall r\in [1, R]. \]
Then together with the above estimate  and $R> \max\{2, \hat{\varrho} \}$, it follows from \eqref{CKNsharpdenomi} that
\begin{align}
\int_{\Omega}  {| {u}_{\kappa, R}|^{a}} r^{s-1}\,  {\dd}x &\geq {\hat{\varrho}}^{n+1} \mathscr{I}_{\m}(\mathbf{0})\, \eta^a (2 \kappa)^{\frac{(1-p)a}{a-p}} (2R)^{-(n+1)} \int_1^{R} r^{s+n-2} \,{\dd}r
\notag\\
& =(s+n-1)^{-1} {\hat{\varrho}}^{n+1} \mathscr{I}_{\m}(\mathbf{0})\, \eta^a (2 \kappa)^{\frac{(1-p)a}{a-p}} (2R)^{-(n+1)} \left( R^{s+n-1} -1
\right) \notag\\
& \geq 2^{\frac{(1-p)a}{a-p}-(n+2)} (s+n-1)^{-1}{\hat{\varrho}}^{n+1} \mathscr{I}_{\m}(\mathbf{0}) \,\eta^a \kappa^{ \frac{(1-p)a}{a-p} }  R^{s-2}.
\label{CKN_key2}
\end{align}

Since  $s>2$, a direct computation shows
\[
\lim_{R \rightarrow +\infty} \frac{  \bar{\varrho}^{-(n+2)} + \int^{R}_1 r^{s-4} {\dd}r +  \left( \frac{(p+s-1-bps)a}{a-p} - s+3 \right)^{-1} R^{s-3} }{ R^{s-2} }
=0.
\]
Thus, together with \eqref{CKNsharpnumer} and \eqref{CKN_key2}, it  yields
\begin{align*}
0 &\leq \lim_{R \rightarrow +\infty}  \frac{ \int_{\Omega} {| {u}_{\kappa, R}|^{a}} (1-\phi(x)) r^{s-1}\, {\dd}x}{ \int_{\Omega}  {| {u}_{\kappa, R}|^{a}} r^{s-1}\,  {\dd}x  } \\
& \leq \lim_{R \rightarrow +\infty}
\left[ \frac{\bar{\varrho}^{n+2}(s+n-1) }{2^{\frac{(1-p)a}{a-p}-(n+2)} {\hat{\varrho}}^{n+1}} \right] \frac{   \bar{\varrho}^{-(n+2)} + \int^{R}_1 r^{s-4} {\dd}r +  \left( \frac{(p+s-1-bps)a}{a-p} - s+3 \right)^{-1} R^{s-3} }{ R^{s-2} }
=0.
\end{align*}
This completes the proof.
\end{proof}

Now we prove the three lemmas containing the interval estimates.
\begin{lemma}\label{lem:I0_estimates}
Assume the hypotheses of Theorem \ref{sbeerHolder}.
Let $u_{\kappa,R}$,  $A^{(0)}_R$, $\tilde A^{(0)}_R$, $V^{(0)}_R$, and $I_0=[0,e^{-R}]$ be defined in the proof of Theorem \ref{sbeerHolder}.
Then
\[
\lim_{R\to +\infty} \frac{ A^{(0)}_R}{A} =\lim_{R\to +\infty}\frac{ \tilde A^{(0)}_R}{A}= \lim_{R\to +\infty} \frac{V^{(0)}_R}{A}=0.
\]
\end{lemma}
\begin{proof}
On $I_0$, $u_{\kappa,R}=\eta_R(r)u_\kappa(0)+(1-\eta_R(r))u_\kappa(r)$. Since $\eta_R$ is smooth with $\eta_R'\le0$ and $|\eta_R'|\le C_1 e^R$ for some constant $C_1>0$, and $|u_\kappa(0)-u_\kappa(r)|\leq |u_\kappa(0)| \leq \eta \kappa^{\frac{1-p}{a-p}}$, we have
\begin{align}\label{CKNu'esti}
|u'_{\kappa,R}|
\leq |\eta_R'(r)|\, |u_\kappa(0)-u_\kappa(r)| +|1-\eta_R(r)| \, |u_\kappa'(r)|  \leq \eta C_1  e^R \kappa^{\frac{1-p}{a-p}} + |u_\kappa'(r)|,
\end{align}
which together with  the convexity of $t\mapsto t^p$ yields
\[ |u'_{\kappa,R}|^p \leq  2^{p-1} \left( \eta^p C_1^p e^{pR} \kappa^{\frac{(1-p)p}{a-p}} + |u_\kappa'(r)|^p \right). \]
Then
combining with $e^{-R} < 1$ and \eqref{u'_kappa_r>1}, it follows that
\begin{align*}
|u'_{\kappa,R}|^p  r^{bps+n-1}(r+\bar\varrho)^{-(n+1)}
 & \leq 2^{p-1} \eta^p C_1^p e^{pR} \kappa^{\frac{(1-p)p}{a-p}} r^{bps+n-1}(r+\bar\varrho)^{-(n+1)} + 2^{p-1} \bar{\varrho}^{-(n+1)} \eta^{\frac{(a-1)p}{p-1}}   \kappa^{\frac{(1-a)p}{a-p}} \\
 & \leq 2^{p-1} \bar\varrho^{-(n+1)} \eta^p C_1^p e^{pR} \kappa^{\frac{(1-p)p}{a-p}} r^{bps+n-1} + 2^{p-1} \bar{\varrho}^{-(n+1)} \eta^{\frac{(a-1)p}{p-1}}   \kappa^{\frac{(1-a)p}{a-p}} \\
 & = 2^{p-1} \bar\varrho^{-(n+1)} \kappa^{\frac{(1-a)p}{a-p}} \left( \eta^p C_1^p e^{pR} r^{bps+n-1} \kappa^{p} + \eta^{\frac{(a-1)p}{p-1}}  \right).
\end{align*}
Since $r \in I_0 = [0, e^{-R}]$, $bps+n-1 >0$, and $p<n$, we have
\begin{align*}
|u'_{\kappa,R}|^p  r^{bps+n-1}(r+\bar\varrho)^{-(n+1)}
& \leq 2^{p-1} \bar\varrho^{-(n+1)} \kappa^{\frac{(1-a)p}{a-p}} \left( \eta^p C_1^p e^{pR} e^{-(bps+n-1)R} \kappa^{p} + \eta^{\frac{(a-1)p}{p-1}}   \right) \\
& \leq 2^{p-1} \bar\varrho^{-(n+1)} \kappa^{\frac{(1-a)p}{a-p}} \left( \eta^p C_1^p e^{-(bps-1)R} \kappa^{p} + \eta^{\frac{(a-1)p}{p-1}}\right).
\end{align*}
Now combining the volume estimate \eqref{inte_ine}, it follows that
\begin{align*}
A^{(0)}_R
&= \int_{S_{\mathbf{0}}\Omega}\int_0^{e^{-R} } |u'_{\kappa,R}|^p r^{bps}\hat\sigma_{\mathbf{0}}(r,y){\dd} r{\dd}\nu_\mathbf{0}(y) \\
& \leq \bar\varrho^{n+1} \mathscr{I}_{\m}(\mathbf{0})  \int_0^{e^{-R}}  |u'_{\kappa,R}|^p r^{bps+n-1}(r+\bar\varrho)^{-(n+1)}{\dd}r \\
& \leq  2^{p-1} \mathscr{I}_{\m}(\mathbf{0}) \kappa^{\frac{(1-a)p}{a-p}} \left( \eta^p C_1^p e^{-(bps-1)R} \kappa^{p} + \eta^{\frac{(a-1)p}{p-1}}\right) e^{-R} \\
& = 2^{p-1} \mathscr{I}_{\m}(\mathbf{0}) \kappa^{\frac{(1-a)p}{a-p}}  \left( \eta^p C_1^p e^{-bpsR} \kappa^{p} + \eta^{\frac{(a-1)p}{p-1}} e^{-R} \right).
\end{align*}
Since  the exponential term decays faster than any positive power of \(R\),    together with $bps \geq 0$, $\kappa = R^{\frac{p+s-1-bps}{p-1}}$, and \eqref{CKNintlowerboundR}, it follows \[\lim_{R\to +\infty} \frac{A^{(0)}_R}{A} = 0.\]

Similarly, by \eqref{u'_kappa_r>1} and
\begin{equation}\label{CKN_I0_u_kappaR_esti}
|u_{\kappa,R}| \leq  \eta_R(r) |u_\kappa(0)|+ (1-\eta_R(r))|u_\kappa(0)|= \eta \kappa^{(1-p)/(a-p)},
\end{equation}
we have
\begin{align*}
	\tilde{A}^{(0)}_R & = \int_{S_{\mathbf{0}}\Omega}\left( \int^{e^{-R}}_0  |u_{\kappa, R}|^{p'(a-1)} r^{(1-b)p's} \hat{\sigma}_{\mathbf{0}}(r,y)\ {\dd}r \right){\dd}\nu_{\mathbf{0}}(y) \\
	& \leq   {\bar{\varrho}}^{n+1} \eta^{p'(a-1)} \mathscr{I}_{\m}(\mathbf{0})  \kappa^{\frac{(1-a)p}{a-p}}   \int^{e^{-R}}_0  r^{(1-b)p's + n-1} (r+\bar{\varrho})^{-(n+1)} {\dd}r  \\
	& \leq \eta^{p'(a-1)} \mathscr{I}_{\m}(\mathbf{0}) \kappa^{\frac{(1-a)p}{a-p}}  e^{- [(1-b)p's + n] R}.
\end{align*}
Since $(1-b)p's+n>0$, the exponent in the above estimates are negative, so together with \eqref{CKNintlowerboundR}, it yields
\[ \lim_{R\to +\infty} \frac{\tilde A^{(0)}_R}{A} = 0. \]

Now we estimate $V^{(0)}_R$. Since $ 0 \leq r\leq e^{-R} <1$, we have
\[ r^s  |u'_{\kappa}| =  \eta^{\frac{a-1}{p-1}} \bigl(\kappa +r^{\frac{p+s-1-bps}{p-1}}\bigr)^{\frac{1-a}{a-p}} r^{\frac{s p(1-b)}{p-1} } \leq
\eta^{\frac{a-1}{p-1}} \kappa^{\frac{1-a}{a-p}}. \]
It follows from \eqref{CKNu'esti} and  $a>p$ that
\[
r^s |u'_{\kappa, R}|  \leq \eta\, C_1  e^R \kappa^{\frac{1-p}{a-p}}{r^{s}} + \eta^{\frac{a-1}{p-1}} \kappa^{\frac{1-a}{a-p}}
\leq \left(\eta C_1  e^R \kappa r^{s}  + \eta^{\frac{a-1}{p-1}}\right) \kappa^{\frac{1-a}{a-p}},
\]
which, combined with \eqref{CKN_I0_u_kappaR_esti}, yields
\[
|u_{\kappa, R}|^{a-1} {r^{s}} |u'_{\kappa, R}| \leq \left(\eta^a C_1  e^R  \kappa r^{s}  + \eta^{\frac{(a-1)p}{p-1}}\right) \kappa^{\frac{(1-a)p}{a-p}} \leq  \left(\eta^a C_1   + \eta^{\frac{(a-1)p}{p-1}}\right) \kappa^{\frac{(1-a)p}{a-p} +1},
 \]
where the last inequality follows from $ 0 \leq r\leq e^{-R}$, $s>2$, and $\kappa > 1$.
Then it follows from \eqref{inte_ine} that
\begin{align*}
	V^{(0)}_R & = \int_{S_{\mathbf{0}}\Omega}\left( \int^{e^{-R}}_0   |u_{\kappa, R}|^{a-1} {r^{s}} |u'_{\kappa, R}| \hat{\sigma}_{\mathbf{0}}(r,y)\ {\dd}r \right){\dd}\nu_{\mathbf{0}}(y) \\
	& \leq {\bar{\varrho}}^{n+1}  \mathscr{I}_{\m}(\mathbf{0}) \left(\eta^a C_1  + \eta^{\frac{(a-1)p}{p-1}}\right) \kappa^{\frac{(1-a)p}{a-p}+ 1}
 \int^{e^{-R}}_0 r^{ n-1} (r+\bar{\varrho})^{-(n+1)} \, {\dd}r  \\
	& \leq   \mathscr{I}_{\m}(\mathbf{0}) \left(\eta^a C_1  + \eta^{\frac{(a-1)p}{p-1}}\right) \kappa^{\frac{(1-a)p}{a-p}+1}
 \int^{e^{-R}}_0 r^{ n-1}   \, {\dd}r \\
	& \leq \mathscr{I}_{\m}(\mathbf{0}) \left(\eta^a C_1  + \eta^{\frac{(a-1)p}{p-1}}\right) \kappa^{\frac{(1-a)p}{a-p}+1} e^{-n R}.
\end{align*}
So together with \eqref{CKNintlowerboundR}, it  yields
\[ \lim_{R\to +\infty} \frac{V^{(0)}_R}{A} = 0. \]
This completes the proof.
\end{proof}

\begin{lemma}\label{lem:I1_estimates}
Assume the hypotheses of Theorem \ref{sbeerHolder}.
Let $u_{\kappa,R}$, $A$, $A^{(1)}_R$, $\tilde A^{(1)}_R$, $V^{(1)}_R$, and $I_1=[e^{-R}, e^R]$ be defined in the proof of Theorem \ref{sbeerHolder}.
Then
\begin{equation} \label{A/A1/V}
\lim_{R\rightarrow +\infty} \frac{ A^{(1)}_R}{A} = \lim_{R\rightarrow +\infty} \frac{\tilde{A}^{(1)}_R}{A} = \lim_{R\rightarrow +\infty} \frac{V^{(1)}_R}{A} = 1.
\end{equation}
\end{lemma}
\begin{proof}
On $I_1$, $u_{\kappa,R}=u_\kappa$, so $u'_{\kappa,R}=u'_\kappa$. Hence
\begin{align*}
A - A^{(1)}_R & = \int_{S_{\mathbf{0}}\Omega}\left( \int^{+\infty}_{e^R} r^{bps} |u'_{\kappa}|^p \hat{\sigma}_{\mathbf{0}}(r,y)\ {\dd}r \right){\dd}\nu_{\mathbf{0}}(y)
+ \int_{S_{\mathbf{0}}\Omega}\left( \int^{e^{-R} }_0 r^{bps} |u'_{\kappa}|^p \hat{\sigma}_{\mathbf{0}}(r,y)\ {\dd}r \right){\dd}\nu_{\mathbf{0}}(y) \\
& =\int_{\Omega \setminus B_{e^{R}}(\mathbf{0})} r^{bps} F^{*p}({\dd}u_\kappa)\, {\dd}x + \int_{B_{e^{-R}}(\mathbf{0})} r^{bps} F^{*p}({\dd}u_\kappa)\, {\dd}x,
\end{align*}
which together with \eqref{eqlemma_uiii} and \eqref{eqlemma_uiv} in Lemma \ref{lemma_u^p_r^p} yields
\begin{align*}
\lim_{R\rightarrow +\infty} \frac{  A - A^{(1)}_R }{A} =0 \Rightarrow \lim_{R\rightarrow +\infty}\frac{ A^{(1)}_R}{A} =1.
\end{align*}
Since $A^{(1)}_R  = \tilde{A}^{(1)}_R = V^{(1)}_R$, it yields \eqref{A/A1/V}.
\end{proof}

\begin{lemma}\label{lem:I2_estimates}
Assume the hypotheses of Theorem \ref{sbeerHolder}.
Let $u_{\kappa,R}$, $A$, $A^{(2)}_R$, $\tilde A^{(2)}_R$, $V^{(2)}_R$, and $I_2=[e^R,e^R+1]$ be defined in the proof of Theorem \ref{sbeerHolder}.
Then
\[
\lim_{R\to +\infty} \frac{A^{(2)}_R}{A} = \lim_{R\to +\infty} \frac{\tilde A^{(2)}_R}{A} =\lim_{R\to +\infty} \frac{V^{(2)}_R}{A} = 0.
\]
\end{lemma}
\begin{proof}
On $I_2$, $u_{\kappa,R}=u_\kappa\varphi_R$, where $\varphi_R(r)=\chi(e^R+1-r)$ satisfies $|\varphi_R'|\le C_2$ for some constant $C_2>0$. Let
\[
\gamma:=-\frac{p+s-1-bps}{a-p}.
\]
From the definition of $u_\kappa$ in \eqref{func_h}, we have for $r\ge e^R >1$,
\[
|u_\kappa(r)|\le \eta r^{\gamma},\qquad |u'_\kappa(r)|\le C_3 r^{\gamma-1},
\]
for some constant $C_3>0$. Hence, on $I_2$,
\[
|u_{\kappa,R}|\le \eta r^{\gamma},\qquad
|u'_{\kappa,R}|=|u'_\kappa\varphi_R+u_\kappa\varphi_R'|
\le (C_3+\eta C_2) r^{\gamma}=:\tilde C r^{\gamma}.
\]
Then, it follows from \eqref{inte_ine} that
\begin{align}
A^{(2)}_R & = \int_{B_{e^R+1}(\mathbf{0}) \backslash B_{e^R}(\mathbf{0}) } |u'_{\kappa, R}|^{p}\, r^{bps} {\dd}x \leq
\tilde{C}^p  \int_{B_{e^R+1}(\mathbf{0}) \backslash B_{e^R}(\mathbf{0}) } \, r^{p(\gamma+bs)} {\dd}x \notag\\
& \leq  {\bar{\varrho}}^{n+1} \tilde{C}^p  \mathscr{I}_{\m}(\mathbf{0})  \int^{e^R+1}_{e^R}   r^{p(\gamma+bs) + n-1} (r+\bar{\varrho})^{-(n+1)} {\dd}r \notag\\
& \leq {\bar{\varrho}}^{n+1} \tilde{C}^p  \mathscr{I}_{\m}(\mathbf{0})   \int^{e^R +1}_{e^R}    r^{p(\gamma+bs) -2}  {\dd}r.
  \label{CKNLemmaA2_bound}
\end{align}
Since $\gamma = -\frac{p+s-1-bps}{a-p}$, we have
\[ p(\gamma+bs) -2 =  \frac{aps}{a-p} \left[ b - \frac{p+s-1}{ps} + \frac{(a-p)(p+s-3)}{aps} \right], \]
which together with the condition on $b$ \eqref{condition_b} gives
\[ \tilde{\theta} := p(\gamma+bs) -2 <0. \]
Then it follows from  \eqref{CKNLemmaA2_bound} that
\[A^{(2)}_R \leq {\bar{\varrho}}^{n+1} \tilde{C}^p  \mathscr{I}_{\m}(\mathbf{0}) e^{\tilde{\theta} R},\]
which together with
\eqref{CKNLemmaA2_bound} and  \eqref{CKNintlowerboundR}  yields
\begin{align*}
\lim_{R\to +\infty} \frac{A^{(2)}_R}{A}  =0.
\end{align*}

For $\tilde{A}^{(2)}_R$ and $V^{(2)}_R$, an almost identical procedure gives
\begin{align}
\tilde{A}^{(2)}_R & = \int_{B_{e^R+1}(\mathbf{0}) \backslash B_{e^R}(\mathbf{0}) }   |u_{\kappa, R}|^{(a-1)p'}\,r^{(1-b)p' s}\,{\dd}x
\leq \eta^{(a-1)p'} \int_{B_{e^R+1}(\mathbf{0}) \backslash B_{e^R}(\mathbf{0}) }  r^{\gamma (a-1)p' + (1-b)p' s}\,{\dd}x  \notag\\
& \leq \eta^{(a-1)p'} {\bar{\varrho}}^{n+1} \mathscr{I}_{\m}(\mathbf{0}) \int^{e^R+1}_{e^R} r^{\gamma (a-1)p' + (1-b)p' s +n-1} (r+\bar{\varrho})^{-(n+1)} {\dd}r \notag\\
& \leq \eta^{(a-1)p'} {\bar{\varrho}}^{n+1} \mathscr{I}_{\m}(\mathbf{0}) \int^{e^R+1}_{e^R} r^{\gamma (a-1)p' + (1-b)p' s -2} {\dd}r,
\label{A(2)}
\end{align}
and
\begin{align}
	V^{(2)}_R & = \int_{B_{e^R+1}(\mathbf{0}) \backslash B_{e^R}(\mathbf{0}) }   |u_{\kappa, R}|^{a-1} \,r^{s} |u'_{\kappa, R, 2}| \,{\dd}x
	\leq \eta^{a-1} \tilde{C} \int_{B_{e^R+1}(\mathbf{0}) \backslash B_{e^R}(\mathbf{0}) }  r^{a\gamma + s}\,{\dd}x \notag\\
	& \leq {\bar{\varrho}}^{n+1} \eta^{a-1} \tilde{C}  \mathscr{I}_{\m}(\mathbf{0}) \int^{e^R+1}_{e^R}  r^{a\gamma + s + n-1} (r+\bar{\varrho})^{-(n+1)} {\dd}r \notag\\
	& \leq {\bar{\varrho}}^{n+1} \eta^{a-1} \tilde{C}  \mathscr{I}_{\m}(\mathbf{0}) \int^{e^R+1}_{e^R}  r^{a\gamma + s -2} {\dd}r.
\label{V(2)}
\end{align}

A direct simplification using \eqref{condition_b} shows that the exponents of
$r$ in the above integrals are negative. In fact
\begin{align*}
	\quad (a-1)p' \gamma + (1-b)p' s -2  &=  - \frac{(p+s-1-bps)(a-1)p}{(a-p)(p-1)} + \frac{(1-b)ps}{p-1} -2\\
& = p (\gamma + bs) - 2 - \frac{(p+s-1-bps)p}{p-1} + \frac{(1-bp)ps}{p-1}, \\
& = \tilde{\theta} - p < 0,
\end{align*}
\begin{align*}
	& a\gamma + s -2 = p (\gamma + bs) - 2 + (a-p)\gamma + s-bps = \tilde{\theta} - p+1 < 0.
\end{align*}
Hence, it follows from \eqref{A(2)} and \eqref{V(2)} that
\begin{align*}
\tilde{A}^{(2)}_R & \leq \eta^{(a-1)p'} {\bar{\varrho}}^{n+1} \mathscr{I}_{\m}(\mathbf{0})\,  e^{ [\gamma (a-1)p' + (1-b)p' s -2] R}, \\
V^{(2)}_R & \leq {\bar{\varrho}}^{n+1} \eta^{a-1} \tilde{C}  \mathscr{I}_{\m}(\mathbf{0}) \,e^{ R (a\gamma + s -2) }.
\end{align*}
These inequalities, together with   \eqref{CKNintlowerboundR} then yield
\[ \lim_{R\rightarrow +\infty} \frac{\tilde{A}^{(2)}_R}{A} =\lim_{R\rightarrow +\infty} \frac{V^{(2)}_R}{A}  =0. \]
This completes the proof.
\end{proof}

\vskip 10mm

\section{Sharp constant for the generalized HPW inequality} \label{SharpHPW}

This section is devoted to the sharpness of the generalized Heisenberg--Pauli--Weyl (HPW)  inequality. As noted in the introduction, the HPW inequality corresponds to the special case $a=p$ in the generalized CKN framework. While the generalized CKN setting ($a>p$) is governed by power-law extremals, the generalized HPW case is driven by exponential-type extremals. Our aim is to prove that the constant $(s-2)/p$ in the HPW-type inequality is sharp.

The proof follows the same truncation scheme as in Section \ref{sec:sharp}, with the power-law test function replaced by an exponential one. We first establish the integral properties of the exponential extremal.

\begin{lemma}\label{lemma_u_exp_HPW}
Let $(\Omega, F, \mathscr{L}^n)$ be a forward complete but not backward complete projectively flat Finsler metric manifold with $\mathbf{K}=0$ and the generating Minkowski norm $\psi$ reversible, endowed with the Lebesgue measure. Let  $p >1$, $s>2$, and $0 \leq b < 1$ satisfying
\begin{align}\label{condition_b*HPW}
b< \frac{p+s-1}{p s}.
\end{align}
Define
\begin{equation}\label{v_HPW}
 v_\kappa(r) := - e^{ -\frac{r^{\alpha+1}}{\kappa(\alpha+1) }},
\end{equation}
where
\[
\alpha: = \frac{s(1-bp)}{p-1},\quad \kappa = R^{\alpha+1}, \quad R>\max\{2, \hat{\varrho}\}.
\]
Then the following hold:
\begin{enumerate}[\rm (i)]
\item \label{exp_uiHPW}
\[
\int_\Omega r^{bps} F^{*p}({\dd}v_\kappa){\dd}x	\, < +\infty;
\]
\item \label{exp_u0HPW}
\begin{align} \label{threeintegralHPW}
 \int_{\Omega} r^{bps} F^{*p}({\dd}v_\kappa)\, {\dd}x  = \frac{1}{\kappa^p} \int_{\Omega} |v_\kappa|^{p} r^{(1-b)p's}\, {\dd}x
=
\frac{1}{\kappa^{p-1}} \int_{\Omega} {|v_{\kappa}|^{p-1}}{r^{s}}  F^*({\dd}v_\kappa)\, {\dd}x;
\end{align}
\item \label{exp_uiiHPW}
there exists some positive constant $C = C(b,n, p, s, \hat\varrho)$ such that
\begin{align}
 \int_{\Omega} r^{bps} F^{*p}({\dd}v_\kappa)\, {\dd}x
\geq  C  \hat\varrho^{n+1} \mathscr{I}_{\m}(\mathbf{0}) \kappa^{-p}
R^{p\alpha + bps -1}  > 0;   \label{exp_uiiHPW_ineq}
 \end{align}
\vskip 3mm
\item \label{exp_uiiiHPW}
\begin{equation} \label{eqlemma_uiiiHPW}
\lim_{R\to +\infty} \frac{\int_{\Omega \setminus B_{e^R}(\mathbf{0})} r^{bps} F^{*p}({\dd}v_\kappa) \, {\dd}x }{\int_{\Omega} r^{bps} F^{*p}({\dd}v_\kappa) \, {\dd}x  } =0;
\end{equation}
\item \label{exp_uivHPW}
\begin{equation} \label{eqlemma_uivHPW}
\lim_{R\to +\infty} \frac{\int_{B_{e^{-R}}(\mathbf{0})} r^{bps} F^{*p}({\dd}v_\kappa)\, {\dd}x  }{\int_{\Omega} r^{bps} F^{*p}({\dd}v_\kappa) \, {\dd}x  } =0.
 \end{equation}
\end{enumerate}
\end{lemma}
\begin{proof}
Before proving \eqref{exp_uiHPW}-\eqref{exp_uivHPW}, we first do some preparation.
Since $F^*({\dd}r)=1$ and $v_\kappa$ is radial, we have $F^*({\dd}v_\kappa)=v'_\kappa$.
By definition and the assumption on $b$, we have
\begin{equation}\label{HPWalpha+1}
p \alpha  = (1-b)p's - bps, \quad \alpha +1 =\frac{p+s-1-bps}{p-1}>0,
\end{equation}
\begin{equation}\label{HPWv_kappa'}
F^*({\dd}v_\kappa) = v'_\kappa = \frac{1}{\kappa} r^\alpha  e^{ -\frac{r^{\alpha+1}}{\kappa(\alpha+1)}} =\frac{1}{\kappa} |v_\kappa| r^\alpha.
\end{equation}
Although $\alpha$ may be negative, the combined exponent $p \alpha + bps + n$
 is positive, which is the one used in the final estimates.

\medskip
For \eqref{exp_uiHPW},
by \eqref{HPWv_kappa'} and the upper bound in  \eqref{inte_ine}, we have
\[
\int_\Omega r^{bps} F^{*p}({\dd}v_\kappa) {\dd}x
\le \bar\varrho^{n+1} \mathscr{I}_{\m}(\mathbf{0}) \kappa^{-p}
\int_0^\infty e^{-\frac{p r^{\alpha+1}}{
\kappa(\alpha+1)}}
r^{p\alpha + bps + n-1} (r+\bar\varrho)^{-(n+1)} {\dd}r < +\infty,
\]
where the last inequality holds because $p \alpha + bps + n-1 = (1-b)p's + n-1 >0$ and the exponential decay dominates any polynomial.

\medskip
For \eqref{exp_u0HPW},
the equality \eqref{threeintegralHPW}  follows from \eqref{exp_uiHPW} and
\[ r^{bps} F^{*p}({\dd}v_\kappa)  = \frac{1}{\kappa^p}  |v_\kappa|^{p} r^{(1-b)p's}
=
\frac{1}{\kappa^{p-1}} {|v_{\kappa}|^{p-1}}{r^{s}}  F^*({\dd}v_\kappa), \]
which is from \eqref{HPWalpha+1} and \eqref{HPWv_kappa'}.

\medskip
For \eqref{exp_uiiHPW},
by the lower bound in \eqref{inte_ine},  we have
\begin{align} \label{HPWintlowerbound}
 \int_\Omega r^{bps} F^{*p}({\dd}v_\kappa) {\dd}x
& \geq  \hat\varrho^{n+1} \mathscr{I}_{\m}(\mathbf{0}) \kappa^{-p}
\int_0^{+\infty} e^{-\frac{p r^{\alpha+1}}{\kappa(\alpha+1) }}
r^{p\alpha + bps + n-1} (r+\hat\varrho)^{-(n+1)} {\dd}r.
\end{align}
Noting that $\kappa = R^{\alpha+1}$,
we set $t: = r/R$. Then
\begin{align*}
& \quad \int_0^{+\infty} e^{-\frac{p r^{\alpha+1}}{\kappa(\alpha+1) }}
r^{p\alpha + bps + n-1} (r+\hat\varrho)^{-(n+1)} {\dd}r \\
&= R^{p\alpha + bps + n} \int_0^{+\infty} e^{-\frac{p t^{\alpha+1}}{\alpha+1}}
t^{p\alpha + bps + n-1} (t R+\hat\varrho)^{-(n+1)} {\dd}t \\
& \geq R^{p\alpha + bps -1} \int_0^{+\infty} e^{-\frac{p t^{\alpha+1}}{\alpha+1}}
t^{p\alpha + bps + n-1} (t +\hat\varrho)^{-(n+1)} {\dd}t = :  R^{p\alpha + bps -1} {\hat\varrho}^{-n-1} C(b,n, p, s, \hat\varrho) ,
\end{align*}
 where the last integral is convergent since $p \alpha + bps + n-1  >0$ and $\alpha+1 >0$. Then it follows from \eqref{HPWintlowerbound} that
 \begin{align}
 \int_\Omega r^{bps} F^{*p}({\dd}v_\kappa) {\dd}x
& \geq  C(b,n, p, s, \hat\varrho) \mathscr{I}_{\m}(\mathbf{0}) \kappa^{-p}
R^{p\alpha + bps -1} > 0,
\end{align}
which is exactly  \eqref{exp_uiiHPW_ineq}.

\medskip
For \eqref{exp_uiiiHPW}, it follows from \eqref{HPWv_kappa'} and the upper bound in  \eqref{inte_ine} that
\begin{align}\label{HPWlemtailupperbound}
\int_{\Omega \setminus B_{e^R}(\mathbf{0})} r^{bps} F^{*p}({\dd}v_\kappa) \, {\dd}x
& \leq
\bar\varrho^{n+1} \mathscr{I}_{\m}(\mathbf{0}) \kappa^{-p} \int_{e^R}^\infty  e^{-\frac{p r^{\alpha+1}}{\kappa(\alpha+1)} } r^{p \alpha + bps  + n-1} (r+\bar\varrho)^{-(n+1)} {\dd}r.
\end{align}
Noting that $\kappa = R^{\alpha+1}$,
we set $t: = r/R$. Then
\begin{align*}
& \quad \int_{e^R}^\infty  e^{-\frac{p r^{\alpha+1}}{\kappa(\alpha+1)} } r^{p \alpha + bps  + n-1} (r+\bar\varrho)^{-(n+1)} {\dd}r \\
& = R^{p \alpha + bps  + n} \int_{\frac{e^R}{R}}^\infty  e^{-\frac{p t^{\alpha+1}}{\alpha+1} } t^{p \alpha + bps  + n-1} (t R+\bar\varrho)^{-(n+1)} {\dd}t \\
& \leq R^{p \alpha + bps  + n} e^{-R (n+1)} \int_{\frac{e^R}{R}}^\infty  e^{-\frac{p t^{\alpha+1}}{\alpha+1} } t^{p \alpha + bps  + n-1} {\dd}t,
\end{align*}
which together with \eqref{HPWlemtailupperbound} yields
\begin{align} \label{HPWlemtailupperbound_1}
\int_{\Omega \setminus B_{e^R}(\mathbf{0})} r^{bps} F^{*p}({\dd}v_\kappa) \, {\dd}x
 \leq \bar\varrho^{n+1} \mathscr{I}_{\m}(\mathbf{0}) \kappa^{-p} R^{p \alpha + bps  + n} e^{-R (n+1)} \int_{\frac{e^R}{R}}^\infty  e^{-\frac{p t^{\alpha+1}}{\alpha+1} } t^{p \alpha + bps  + n-1} {\dd}t.
\end{align}
Since $\int_{0}^\infty  e^{-\frac{p t^{\alpha+1}}{\alpha+1} } t^{p \alpha + bps  + n-1} {\dd}t < + \infty$, we have
\[ \lim_{R\to +\infty}\int_{\frac{e^R}{R}}^\infty  e^{-\frac{p t^{\alpha+1}}{\alpha+1} } t^{p \alpha + bps  + n-1} {\dd}t =0.   \]
Hence, combining  \eqref{HPWlemtailupperbound_1} and \eqref{exp_uiiHPW_ineq}, it follows that
\[
\lim_{R\to +\infty} \frac{\int_{\Omega \setminus B_{e^R}(\mathbf{0})} r^{bps} F^{*p}({\dd}v_\kappa) \, {\dd}x }{\int_{\Omega} r^{bps} F^{*p}({\dd}v_\kappa) \, {\dd}x  } =0.
\]

\medskip
For \eqref{exp_uivHPW}, it follows from \eqref{HPWv_kappa'} and \eqref{inte_ine} that
\begin{align}
\int_{B_{e^{-R}}(\mathbf{0})} r^{bps} F^{*p}({\dd}v_\kappa)\, {\dd}x
& \leq
\bar\varrho^{n+1} \mathscr{I}_{\m}(\mathbf{0}) \kappa^{-p} \int_0^{e^{-R}}  e^{-\frac{p r^{\alpha+1}}{\kappa(\alpha+1)}} r^{p \alpha + bps  + n-1} (r+\bar\varrho)^{-(n+1)} {\dd}r.
\end{align}
Then, using the same substitution as in the proof of \eqref{exp_uiiiHPW}, we get
\begin{align*}
& \quad \int_0^{e^{-R}}  e^{-\frac{p r^{\alpha+1}}{\kappa(\alpha+1)}} r^{p \alpha + bps  + n-1} (r+\bar\varrho)^{-(n+1)} {\dd}r \\
& = R^{p \alpha + bps  + n} \int_{0}^{\frac{e^{-R}}{R}}  e^{-\frac{p t^{\alpha+1}}{\alpha+1} } t^{p \alpha + bps  + n-1} (t R+\bar\varrho)^{-(n+1)} {\dd}t \\
& \leq \bar\varrho^{-(n+1)} R^{p \alpha + bps  + n} \int_{0}^{\frac{e^{-R}}{R}} t^{p \alpha + bps  + n-1} {\dd}t
 =  \bar\varrho^{-(n+1)} (p \alpha + bps  + n)^{-1} e^{ - R (p \alpha + bps  + n)}.
\end{align*}
Hence,
\[ \int_{B_{e^{-R}}(\mathbf{0})} r^{bps} F^{*p}({\dd}v_\kappa)\, {\dd}x
 \leq (p \alpha + bps  + n)^{-1} \mathscr{I}_{\m}(\mathbf{0}) \kappa^{-p} e^{ - R (p \alpha + bps  + n)}, \]
which together with \eqref{exp_uiiHPW_ineq} implies
\[
\lim_{R\to +\infty} \frac{\int_{B_{e^{-R}}(\mathbf{0})} r^{bps} F^{*p}({\dd}v_\kappa)\, {\dd}x  }{\int_{\Omega} r^{bps} F^{*p}({\dd}v_\kappa) \, {\dd}x  } =0.
\]
This completes the proof.
\end{proof}

\begin{theorem}\label{thm:HPW_sharp}
Let $(\Omega, F, \mathscr{L}^n)$ be a forward complete but not backward complete projectively flat Finsler manifold with $\mathbf{K}=0$ and the generating Minkowski norm $\psi$ reversible, endowed with the Lebesgue measure. Let  $p >1$, $s>2$, and $0 \leq b < 1$ satisfying
\begin{align}\label{condition_b*HPW1}
b< \frac{p+s-1}{p s}.
\end{align}
Then for any \(f\in C_0^\infty(\Omega)\), the following HPW-type inequality holds:
\begin{align}\label{sbeerHolder_ineqHPW}
&\left( \int_{\Omega} r^{bps} F^{p}(\nabla f) {\dd}x \right)^{1/{p}} \left( \int_{\Omega} |f|^{p} r^{(1-b)p's} {\dd}x  \right)^{1/{p'}}
 \notag\\
&\qquad \qquad\geq \frac{1}{p} \left[(s-2) \int_{\Omega}  {|f|^{p}}{r^{s-1}}   {\dd}x  + (n+1) \int_{\Omega} {|f|^{p}} (1-\phi(x)) r^{s-1} {\dd}x\right],
\end{align}
where the constant \( 1/p\) is sharp.
\end{theorem}
\begin{proof}[Proof of Theorem \ref{thm:HPW_sharp}]
By \eqref{granotwoff}, replace $F(\nabla f)$ by $F^*({\dd}f)$.
The inequality follows by the proof of Theorem \ref{thm:CKNHPW_threshold}\eqref{1.2(i)} with $\lambda_\psi=1$, or by applying H\"older's inequality to Corollary \ref{maincorollary} with \(a=p\),
i.e.,
\begin{align*}
&\left( \int_{\Omega} r^{bps} F^{*p}({\dd}f) {\dd}x \right)^{1/{p}} \left( \int_{\Omega} |f|^{p} r^{(1-b)p's} {\dd}x  \right)^{1/{p'}}
 \geq
\int_{\Omega} {|f|^{p-1}}{r^{s}}  F^*({\dd}f) {\dd}x  \\
& \qquad \qquad \geq  \frac{1}{p} \left[(s-2) \int_{\Omega}  {|f|^{p}}{r^{s-1}}   {\dd}x  + (n+1) \int_{\Omega} {|f|^{p}} (1-\phi(x)) r^{s-1} {\dd}x\right].
\end{align*}
It remains to prove the sharpness.

By H\"older's equality condition, the first equality in the above formula holds if and only if
\begin{equation} \label{thm1.2_HolderFuncHPW}
r^{bps} F^{*p}({\dd}f) = c |f|^{p} r^{(1-b)p's},
\end{equation}
for some constant \(c>0\).
The function \(v_\kappa\) defined in Lemma \ref{lemma_u_exp_HPW},
 i.e.,
\begin{equation}\label{v_kappa}
v_\kappa(r) =  - e^{ -\frac{r^{\alpha+1}}{\kappa(\alpha+1) }},
\end{equation}
where
\[
\alpha: = \frac{s(1-bp)}{p-1},\quad \kappa = R^{\alpha+1}, \quad R> \max\{2, \hat{\varrho}\}.
\]
satisfies \eqref{thm1.2_HolderFuncHPW} with $c= \kappa^{-p} $, since \(v'_\kappa = \frac{1}{\kappa}|v_\kappa| r^\alpha\) and \(\alpha p = (1-b)p's - bps\). 
It can be verified directly that $v_{\kappa} < 0$, $v'_{\kappa} > 0$, and $\alpha +1 >0$ since $b< \frac{p+s-1}{p s}$ by the assumption.

Although $v_{\kappa}$ is not compactly supported and not smooth at $r=0$, we still can construct a smooth
function with compact support by $v_{\kappa}$ and two cut-off functions $\eta_R(r)$, $\varphi_R(r)$ as in \eqref{varphi_eta(r)} and \eqref{varphi_R(r)}.
Set
\begin{equation} \label{h_RHPW}
v_{\kappa, R}(r)= \eta_R(r)v_{\kappa}(0) + (1-\eta_R(r)) \varphi_R(r) v_{\kappa}(r).
\end{equation}
Then \(v_{\kappa,R}\in C_0^\infty(\Omega)\), \(v_{\kappa,R}\le0\), and \(v_{\kappa,R}'\ge0\). By Corollary \ref{maincorollary}, \(v_{\kappa,R}\) achieves equality in the underlying Hardy-type inequality \eqref{ineqmain1_reverpsi} for $a=p$, i.e.
 \begin{align*}
\int_\Omega |v_{\kappa,R}|^{p-1} r^s F(\nabla v_{\kappa,R}){\dd}x
= \frac{1}{p}\left[(s-2)\int_\Omega |v_{\kappa,R}|^p r^{s-1}{\dd}x
+ (n+1)\int_\Omega |v_{\kappa,R}|^p (1-\phi(x)) r^{s-1}{\dd}x\right].
\end{align*}

Define
\[
A_R:=\int_\Omega r^{bps}F^{*p}({\dd}v_{\kappa,R})\,{\dd}x,\quad
\tilde A_R:=\int_\Omega |v_{\kappa,R}|^{p}r^{(1-b)p's}\,{\dd}x,\quad
V_R:=\int_\Omega |v_{\kappa,R}|^{p-1}r^sF^*({\dd}v_{\kappa,R})\,{\dd}x.
\]
Therefore, to prove that $1/p$ is the sharp constant in \eqref{sbeerHolder_ineqHPW}, it suffices to show
\begin{equation}\label{Q_RHPW}
\lim_{R\to +\infty}\frac{(A_R)^{1/p}(\tilde A_R)^{1/p'}}{V_R}=1.
\end{equation}

We split \([0,\infty)\) into \(I_0=[0,e^{-R}]\), \(I_1=[e^{-R},e^R]\), \(I_2=[e^R,e^R+1]\), \(I_3=[e^R+1,\infty)\).
For simplicity, for $v_{\kappa, R}$, define
\begin{equation} \label{v_kappaRiHPW}
v_{\kappa, R, i} := v_{\kappa, R}|_{r\in I_i}, \quad v'_{\kappa, R, i} := v'_{\kappa, R}|_{r\in I_i},
\end{equation}
\begin{align*}
A^{(i)}_R&: = \int_{\Omega} |v'_{\kappa, R, i}|^{p}\, r^{bps} {\dd}x,\qquad
\tilde{A}^{(i)}_R: = \int_{\Omega}   |v_{\kappa, R, i}|^{p}\,r^{(1-b)p' s}\,{\dd}x,\\
V^{(i)}_R&: = \int_{\Omega} { |v_{\kappa, R, i}|^{p-1}}{r^{s}} |v'_{\kappa, R, i}| {\dd}x, \qquad i=0, 1, 2, 3.
\end{align*}
On \(I_3\), \(v_{\kappa,R}\equiv0\), so the corresponding integrals vanish.
Then
\begin{equation} \label{AAVHPW}
A_R =A^{(0)}_R + A^{(1)}_R + A^{(2)}_R, \quad \tilde{A}_R = \tilde{A}^{(0)}_R + \tilde{A}^{(1)}_R + \tilde{A}^{(2)}_R, \quad V_R =  V^{(0)}_R + V^{(1)}_R + V^{(2)}_R.
\end{equation}

The estimates on the remaining intervals are given in Lemmas \ref{lem:I0_estimatesHPW}, \ref{lem:I1_estimatesHPW}, and \ref{lem:I2_estimatesHPW} at the end of this subsection. They yield
\[
\lim_{R\to +\infty} \frac{A^{(0)}_R}{A}=\lim_{R\to +\infty}\frac{ \tilde A^{(0)}_R}{\kappa^p A}=\lim_{R\to +\infty} \frac{V^{(0)}_R }{\kappa^{p-1} A}=0,\qquad
\lim_{R\to +\infty} \frac{A^{(1)}_R}{A}=\lim_{R\to +\infty} \frac{\tilde A^{(1)}_R }{\kappa^p A}=\lim_{R\to +\infty}  \frac{ V^{(1)}_R}{\kappa^{p-1} A}=1,
\]
\[
\lim_{R\to +\infty} \frac{A^{(2)}_R}{A}=\lim_{R\to +\infty} \frac{\tilde A^{(2)}_R}{\kappa^p A}=\lim_{R\to +\infty} \frac{V^{(2)}_R}{ \kappa^{p-1} A }=0,
\]
where $A:=\int_\Omega |v_\kappa'|^p  r^{bps}{\dd}x$ is finite and positive by Lemma \ref{lemma_u_exp_HPW}\eqref{exp_uiHPW} and \eqref{exp_uiiHPW_ineq}. Hence
\[
\lim_{R\to +\infty} \frac{ A_R}{A} = \lim_{R\to +\infty} \frac{\tilde A_R}{\kappa^{p} A}   = \lim_{R\to +\infty}  \frac{V_R}{\kappa^{p-1} A}  = 1,
\]
and consequently
\begin{equation*}
\lim_{R\to +\infty}\frac{(A_R)^{1/p}(\tilde A_R)^{1/p'}}{V_R}=1.
\end{equation*}
This completes the proof.
\end{proof}

We now prove Theorem \ref{thm:CKNHPW_sharp}\eqref{1.3(ii)} using the same method and function as in the preceding theorem.
\begin{proof}[Proof of Theorem \ref{thm:CKNHPW_sharp}\eqref{1.3(ii)}]
Since  the assumptions are the same as those in Theorem \ref{thm:HPW_sharp} and $1-\phi(x) >0$, the inequality \eqref{CKNsharp} ($a=p$)
holds naturally by
\eqref{sbeerHolder_ineqHPW}, i.e.,
\begin{align}\label{sbeerHolder_ineqHPW_1}
\left( \int_{\Omega} r^{bps} F^{*p}({\dd}f) {\dd}x \right)^{1/{p}} \left( \int_{\Omega} |f|^{p} r^{(1-b)p's} {\dd}x  \right)^{1/{p'}}
 \geq \frac{s-2}{p}  \int_{\Omega}  {|f|^{p}} r^{s-1}   {\dd}x.
\end{align}
Moreover, for $a=p$, the condition on $b$ is now simplified into
\begin{align}\label{condition_b*HPW1_1}
b< \frac{p+s-1}{p s}.
\end{align}
Hence, we only need to show $(s-2)/p$ is the sharp constant.
To prove it, we still use the same function $v_{\kappa}$ and the two cut-off functions $\eta_R$, $\varphi_R$ to compose $v_{\kappa, R}$ as in \eqref{h_RHPW}.
By definition, since $|v_{\kappa}(0)| = 1 \geq |v_{\kappa}(r)|$,  the following estimate holds:
\begin{equation}\label{HPWvbound}
|v_{\kappa, R}| =
\begin{cases}
|\eta_R(r)v_{\kappa}(0)| + |(1-\eta_R(r)) v_{\kappa}(r)| \leq 1, &\qquad \text{if } r\in [0, e^{-R}], \\[3pt]
|v_{\kappa}(r)| = e^{ -\frac{r^{\alpha+1}}{\kappa(\alpha+1) }}, &\qquad \text{if } r\in [e^{-R}, e^R],\\[3pt]
|\varphi_R(r) v_{\kappa}(r)| \leq |v_{\kappa}(r)| = e^{ -\frac{r^{\alpha+1}}{\kappa(\alpha+1) }}, &\qquad \text{if } r\in [e^R, e^{R}+1], \\[3pt]
0, &\qquad \text{if } r\in [e^{R}+1, +\infty).
\end{cases}
\end{equation}
Then by the proof of Theorem \ref{thm:HPW_sharp}, the equation \eqref{Q_RHPW} still holds for $v_{\kappa, R}$.
Thus it remains to show
\begin{equation}\label{HPWsharptarget}
\lim_{R \rightarrow +\infty}  \frac{ \int_{\Omega} {| v_{\kappa, R}|^{p}} (1-\phi(x)) r^{s-1}\, {\dd}x}{ \int_{\Omega}  {|v_{\kappa, R}|^{p}} r^{s-1}\,  {\dd}x  }
=0.
\end{equation}

By \eqref{phi_rho_minmax}, we have
\begin{align}\label{HPW1-phi}
\frac{\hat{\varrho}}{\hat{\varrho} +r} \leq  1-\phi(x) \leq   \frac{\bar{\varrho}}{\bar{\varrho} +r}.
\end{align}
Now we divide the proof into two steps to estimate the two integrals in the target limit \eqref{HPWsharptarget}.

\medskip
\noindent
\paragraph{\emph{Step 1: Numerator estimate.}}
Combining \eqref{inte_ine} and \eqref{HPW1-phi}, we get
\begin{align}
& \quad \int_{\Omega} {| {v}_{\kappa, R}|^{p}} (1-\phi(x)) r^{s-1}\, {\dd}x
\leq \int_{\Omega} {| {v}_{\kappa, R}|^{p}} r^{s-1} \frac{\bar{\varrho}}{\bar{\varrho} +r} \, {\dd}x \notag \\
& \leq {\bar{\varrho}}^{n+2}  \mathscr{I}_{\m}(\mathbf{0})
\int^{+\infty}_0  {| {v}_{\kappa, R}|^{p}} r^{s+ n-2} (r+\bar{\varrho})^{-(n+2)} {\dd}r
\notag\\
& = {\bar{\varrho}}^{n+2}  \mathscr{I}_{\m}(\mathbf{0})
\int^{1}_0  {|v_{\kappa, R}|^{p}} r^{s+ n-2} (r+\bar{\varrho})^{-(n+2)} {\dd}r + {\bar{\varrho}}^{n+2} \mathscr{I}_{\m}(\mathbf{0})
\int^{e^R+1}_1  {|v_{\kappa, R}|^{p}} r^{s+ n-2} (r+\bar{\varrho})^{-(n+2)} {\dd}r
\notag\\
& \leq \mathscr{I}_{\m}(\mathbf{0})
\int^{1}_0  {|v_{\kappa, R}|^{p}}  {\dd}r
 + {\bar{\varrho}}^{n+2} \mathscr{I}_{\m}(\mathbf{0}) \int_1^{e^R+1} |v_{\kappa, R}|^{p} r^{s-4} \,{\dd}r.
\label{HPW_key1}
\end{align}
The estimate  \eqref{HPWvbound} yields that $|v_{\kappa, R}| \leq 1$ for $r\in [0, 1]$, and  $|v_{\kappa, R}| \leq  e^{ -\frac{r^{\alpha+1}}{\kappa(\alpha+1) }}$ for $r\in [1, e^{R+1}]$. Then \eqref{HPW_key1} implies
\[
\int_{\Omega} {| {v}_{\kappa, R}|^{p}} (1-\phi(x)) r^{s-1}\, {\dd}x  \leq \mathscr{I}_{\m}(\mathbf{0})  + {\bar{\varrho}}^{n+2} \mathscr{I}_{\m}(\mathbf{0}) \int_1^{e^R+1} e^{ -\frac{p r^{\alpha+1}}{\kappa(\alpha+1) }} r^{s-4} \,{\dd}r.
\]
To estimate the tail integral \[\int_1^{e^R+1} e^{ -\frac{p r^{\alpha+1}}{\kappa(\alpha+1) }} r^{s-4} \,{\dd}r,\]
we consider the two cases: $2<s \leq 3$ and $s>3$.

\medskip
	\subparagraph{\bf{Case 1: $2<s \leq 3$.}}
Since $R>2$, we have
\begin{align}
 \int_1^{e^R+1} e^{ -\frac{p r^{\alpha+1}}{\kappa(\alpha+1) }} r^{s-4} \,{\dd}r
& \leq  \int_1^{R} r^{s-4} \,{\dd}r  +  \int_R^{e^R+1} e^{ -\frac{p r^{\alpha+1}}{\kappa(\alpha+1) }} r^{s-4} \,{\dd}r \notag\\
& \leq  \ln R + \int_R^{e^R+1} e^{ -\frac{p r^{\alpha+1}}{\kappa(\alpha+1) }} r^{s-4} \,{\dd}r.
\end{align}
Noting that $\kappa = R^{\alpha+1}$, we set $t: = r/R$. Then
\[
\int_R^{e^R+1} e^{ -\frac{p r^{\alpha+1}}{\kappa(\alpha+1) }} r^{s-4} \,{\dd}r
= R^{s-3} \int_{1}^{\frac{e^R+1}{R}} e^{ -\frac{p t^{\alpha+1}}{\alpha+1 }} t^{s-4} \,{\dd}t
\leq R^{s-3} \int_{1}^{+\infty} e^{ -\frac{p t^{\alpha+1}}{\alpha+1 }} t^{s-4} \,{\dd}t = :   \mathcal{C}_1(b,p,s) R^{s-3},
\]
where the convergence of $\int_{1}^{+\infty} e^{ -\frac{p t^{\alpha+1}}{\alpha+1 }} t^{s-4} \,{\dd}t$ is used. Hence,
\begin{equation}\label{sharpHPWnumer}
\int_1^{e^R+1} e^{ -\frac{p r^{\alpha+1}}{\kappa(\alpha+1) }} r^{s-4} \,{\dd}r \leq  \ln R +  \mathcal{C}_1(b,p,s) R^{s-3}.
\end{equation}

\medskip
	\subparagraph{\bf{Case 2: $s> 3$.}}
Noting that $\kappa = R^{\alpha+1}$, we let $t: = r/R$. Then
\[
\int_1^{e^R+1} e^{ -\frac{p r^{\alpha+1}}{\kappa(\alpha+1) }} r^{s-4} \,{\dd}r
= R^{s-3} \int_{\frac{1}{R}}^{\frac{e^R+1}{R}} e^{ -\frac{p t^{\alpha+1}}{\alpha+1 }} t^{s-4} \,{\dd}t
\leq R^{s-3} \int_{0}^{+\infty} e^{ -\frac{p t^{\alpha+1}}{\alpha+1 }} t^{s-4} \,{\dd}t.
\]
Since $s-4 > -1$, the integral $\int_{0}^{+\infty} e^{ -\frac{p t^{\alpha+1}}{\alpha+1 }} t^{s-4} \,{\dd}t$ is convergent. Then by setting
$  \mathcal{C}_2(b,p,s): =\int_{0}^{+\infty} e^{ -\frac{p t^{\alpha+1}}{\alpha+1 }} t^{s-4} \,{\dd}t$, we have
\[
\int_1^{e^R+1} e^{ -\frac{p r^{\alpha+1}}{\kappa(\alpha+1) }} r^{s-4} \,{\dd}r \leq \mathcal{C}_2(b,p,s)  R^{s-3}.
\]
Thus, combining these two cases, we obtain
\begin{equation}\label{HPWsharpnumer}
\int_{\Omega} {| {v}_{\kappa, R}|^{p}} (1-\phi(x)) r^{s-1}\, {\dd}x
\leq
\begin{cases}
 \mathscr{I}_{\m}(\mathbf{0})\left[ 1 + {\bar{\varrho}}^{n+2}  \ln R + {\bar{\varrho}}^{n+2} \mathcal{C}_1(b,p,s)R^{s-3} \right],   & \text{if }2< s \leq 3,\\[3pt]
 \mathscr{I}_{\m}(\mathbf{0})\left[ 1 + {\bar{\varrho}}^{n+2} \mathcal{C}_2(b,p,s) R^{s-3} \right], &  \text{if } s> 3.
\end{cases}
\end{equation}

\medskip
\noindent
\paragraph{\emph{Step 2: Denominator estimate.}}
Combining  \eqref{inte_ine}  and \eqref{HPWvbound}, it follows that
\begin{align}
\int_{\Omega}  {| {v}_{\kappa, R}|^{p}} r^{s-1}\,  {\dd}x
&\geq {\hat{\varrho}}^{n+1} \mathscr{I}_{\m}(\mathbf{0}) \int^{+\infty}_0 {| v_{\kappa, R}|^{p}} r^{s+n-2} (r+\hat{\varrho})^{-(n+1)} \, {\dd}r \notag\\
& \geq {\hat{\varrho}}^{n+1} \mathscr{I}_{\m}(\mathbf{0}) \int^{e^R}_R {| v_{\kappa, R}|^{p}} r^{s+n-2} (r+\hat{\varrho})^{-(n+1)} \, {\dd}r \notag\\
& = {\hat{\varrho}}^{n+1} \mathscr{I}_{\m}(\mathbf{0}) \int^{e^R}_R e^{ -\frac{p r^{\alpha+1}}{\kappa(\alpha+1) }} r^{s+n-2} (r+\hat{\varrho})^{-(n+1)} \, {\dd}r \notag\\
&\geq  \left( \frac{\hat{\varrho}}{2} \right)^{n+1} \mathscr{I}_{\m}(\mathbf{0}) \int^{e^R}_R e^{ -\frac{p r^{\alpha+1}}{\kappa(\alpha+1) }} r^{s- 3} \, {\dd}r,
\label{HPW_key2}
\end{align}
where the last inequality follows from $R > \hat{\varrho}$, as defined in \eqref{v_kappa}.
Let $t: = r/R$. Then by $\kappa = R^{\alpha+1}$, we have
\[
\int^{e^R}_R e^{ -\frac{p r^{\alpha+1}}{\kappa(\alpha+1) }} r^{s-3} \, {\dd}r
=  R^{s-2} \int^{\frac{e^R}{R}}_1 e^{ -\frac{p t^{\alpha+1}}{\alpha+1 }} t^{s-3} \, {\dd}t
\geq  R^{s-2} \int^2_1 e^{ -\frac{p t^{\alpha+1}}{\alpha+1 }} t^{s-3} \, {\dd}t
=: \mathcal{C}_3(b,p,s)  R^{s-2},
\]
since $\int^{2}_1 e^{ -\frac{p t^{\alpha+1}}{\alpha+1 }} t^{s-3} \, {\dd}t  < + \infty$.
Hence, it follows from \eqref{HPW_key2} that
\begin{equation}\label{HPWsharpdenomi}
\int_{\Omega}  {| {v}_{\kappa, R}|^{p}} r^{s-1}\,  {\dd}x
\geq  \left( \frac{\hat{\varrho}}{2} \right)^{n+1} \mathscr{I}_{\m}(\mathbf{0}) \mathcal{C}_3(b,p,s) R^{s-2}.
\end{equation}

\medskip
Therefore, together with \eqref{HPWsharpnumer} and \eqref{HPWsharpdenomi}, it yields
\begin{equation*}
\lim_{R \rightarrow +\infty}  \frac{ \int_{\Omega} {| v_{\kappa, R}|^{p}} (1-\phi(x)) r^{s-1}\, {\dd}x}{ \int_{\Omega}  {|v_{\kappa, R}|^{p}} r^{s-1}\,  {\dd}x  }
\leq
\begin{cases}
\lim_{R \rightarrow +\infty} \frac{  1 + {\bar{\varrho}}^{n+2}  \ln R + {\bar{\varrho}}^{n+2}    \mathcal{C}_1(b,p,s)R^{s-3} }{\left( \frac{\hat{\varrho}}{2} \right)^{n+1}  \mathcal{C}_3(b,p,s) R^{s-2}} = 0,   & \text{if }2< s \leq 3,\\[18pt]
\lim_{R \rightarrow +\infty} \frac{ 1 + {\bar{\varrho}}^{n+2} \mathcal{C}_2(b,p,s) R^{s-3}}{\left( \frac{\hat{\varrho}}{2} \right)^{n+1} \mathcal{C}_3(b,p,s)  R^{s-2}} =0, &  \text{if } s> 3.
\end{cases}
\end{equation*}
This proves  \eqref{HPWsharptarget} and hence completes the proof.
\end{proof}

The three auxiliary lemmas for the interval estimates are stated below.
\begin{lemma}\label{lem:I0_estimatesHPW}
With the notation of Theorem \ref{thm:HPW_sharp}, we have
\[
\lim_{R\to +\infty} \frac{A^{(0)}_R}{A}=\lim_{R\to +\infty}\frac{ \tilde A^{(0)}_R}{\kappa^p A}=\lim_{R\to +\infty} \frac{V^{(0)}_R }{\kappa^{p-1} A}=0.
\]
\end{lemma}

\begin{proof}
	On $I_0=[0,e^{-R}]$,
	\[
	v_{\kappa,R}
	=\eta_R(r)v_\kappa(0)+(1-\eta_R(r))v_\kappa(r).
	\]
	Since $|\eta_R'|\le C_1 e^R$ for some constant $C_1>0$, and
	$|v_{\kappa,R}|\le |v_\kappa(0)|=1$, we have
	\[
	v_\kappa'(r)
	=
	\kappa^{-1}r^\alpha
	e^{-\frac{r^{\alpha+1}}{\kappa(\alpha+1)}}
	\le \kappa^{-1}r^\alpha,
	\quad \text{for } 0<r\le e^{-R}.
	\]
	Moreover, since $\eta_R(0)=1$ and $1-\eta_R(r)\le C_1e^Rr$ which follows from the mean value theorem,  using $1-e^{-t}\le t$ for $t\ge0$,
	\[
	|v_\kappa(0)-v_\kappa(r)|
	=
	1-e^{-\frac{r^{\alpha+1}}{\kappa(\alpha+1)}}
	\le
	\frac{r^{\alpha+1}}{\kappa(\alpha+1)}.
	\]
	Hence
	\begin{align}
		|v'_{\kappa,R}|
		&=
		\left|
		\eta_R'(r)(v_\kappa(0)-v_\kappa(r))
		+(1-\eta_R(r))v_\kappa'(r)
		\right| \notag\\
		&\le
		\frac{C_1}{\alpha+1}\kappa^{-1}e^Rr^{\alpha+1}
		+
		C_1\kappa^{-1}e^Rr^{\alpha+1} \notag \\
		&\le
		\bar C_1 \kappa^{-1}e^{-\alpha R}, \label{I_0Lem4.3u'}
	\end{align}
	where we used $r\le e^{-R}$, $\alpha+1>0$,
	$\kappa=R^{\alpha+1}$, and
	$
	\bar C_1:=\left(1+\frac{1}{\alpha+1}\right)C_1 >0.
	$
	
	Using the volume estimate \eqref{inte_ine} and \eqref{I_0Lem4.3u'}, we obtain
	\begin{align*}
		A^{(0)}_R & =\int_{B_{e^{-R}}(\mathbf{0})} |v'_{\kappa, R}|^{p}\, r^{bps} {\dd}x \\
		&\leq
		\bar{C}_1^p \bar\varrho^{n+1}
		\mathscr{I}_{\m}(\mathbf{0})
		\kappa^{-p} e^{-p\alpha R}
		\int_0^{e^{-R}}
		r^{bps+n-1}(r+\bar\varrho)^{-(n+1)}\,{\dd}r \\
		&\le
		\bar{C}_1^p\mathscr{I}_{\m}(\mathbf{0})
		\kappa^{-p}
		e^{-(p\alpha+bps+n)R}.
	\end{align*}
	Since $\alpha>-1$, $p<n$, and $bps \geq 0$, we have
	$
	-(p\alpha+bps+n) <0.
	$
	Thus, together with \eqref{exp_uiiHPW_ineq}, it yields
	\[
	\lim_{R\to+\infty}\frac{A^{(0)}_R}{A}=0.
	\]
	
	Similarly,
	\begin{align*}
		\tilde{A}^{(0)}_R & = \int_{B_{e^{-R}}(\mathbf{0})} |v_{\kappa, R}|^{p}\, r^{(1-b)p's} {\dd}x \\
		&\leq
		{\bar{\varrho}}^{n+1}\mathscr{I}_{\m}(\mathbf{0})
		\int^{e^{-R}}_0 |v_{\kappa, R}|^{p}
		r^{(1-b)p's+n-1}
		(r+\bar{\varrho})^{-(n+1)}
		\,{\dd}r \\
&\leq \mathscr{I}_{\m}(\mathbf{0})
		\int^{e^{-R}}_0
		r^{(1-b)p's+n-1}
		\,{\dd}r
		\leq
		\mathscr{I}_{\m}(\mathbf{0})
		e^{-[(1-b)p's+n]R},
	\end{align*}
	and
	\begin{align*}
		V^{(0)}_R  & = \int_{B_{e^{-R}}(\mathbf{0})} |v_{\kappa,R,0}|^{p-1}r^s
		|v'_{\kappa,R,0}| {\dd}x \leq \int_{B_{e^{-R}}(\mathbf{0})} r^s
		|v'_{\kappa,R,0}| {\dd}x\\
		&\leq
		{\bar{\varrho}}^{n+1}\bar{C}_1
		\mathscr{I}_{\m}(\mathbf{0})
		\kappa^{-1} e^{-\alpha R}
		\int^{e^{-R}}_0
		r^{s+n-1}(r+\bar\varrho)^{-(n+1)}
		\,{\dd}r \\
		&\leq
		\bar{C}_1\mathscr{I}_{\m}(\mathbf{0})
		\kappa^{-1}
		e^{-(\alpha+s+n)R}.
	\end{align*}
	Since $(1-b)p's+n>0$ and
	\[
	\alpha+s+n > s+n-1 > 0,
	\]
	together with \eqref{exp_uiiHPW_ineq}, it follows that
	\begin{align*}
0 \leq \lim_{R\to+\infty}
	\frac{\tilde A^{(0)}_R}{\kappa^p A}
	& \leq\lim_{R\to+\infty}  \frac{e^{-[(1-b)p's+n]R}}{C  \hat\varrho^{n+1} R^{p\alpha + bps -1}  } =0,  \\
0 \leq \lim_{R\to+\infty}
	\frac{V^{(0)}_R}{\kappa^{p-1}A}
	&\leq  \lim_{R\to+\infty}  \frac{\bar{C}_1
		e^{-(\alpha+s+n)R}}{C  \hat\varrho^{n+1} R^{p\alpha + bps -1}  } =0.
	\end{align*}
This completes the proof.
\end{proof}

\begin{lemma}\label{lem:I1_estimatesHPW}
With the notation of Theorem \ref{thm:HPW_sharp}, we have
\begin{align}\label{A/A1/VHPW}
\lim_{R\to +\infty} \frac{A^{(1)}_R}{A}=\lim_{R\to +\infty} \frac{\tilde A^{(1)}_R }{\kappa^p A}=\lim_{R\to +\infty}  \frac{ V^{(1)}_R}{\kappa^{p-1} A}=1.
\end{align}
\end{lemma}
\begin{proof}
On $I_1=[e^{-R}, e^{R}]$, $v_{\kappa,R}=v_\kappa$, with $v'_{\kappa,R}=v'_\kappa$.
Hence
\begin{align*}
A - A^{(1)}_R & = \int_{S_{\mathbf{0}}\Omega}\left( \int^{+\infty}_{e^{R}} r^{bps} |v'_{\kappa}|^p \hat{\sigma}_{\mathbf{0}}(r,y)\ {\dd}r \right){\dd}\nu_{\mathbf{0}}(y)
+ \int_{S_{\mathbf{0}}\Omega}\left( \int^{e^{-R}}_0 r^{bps} |v'_{\kappa}|^p \hat{\sigma}_{\mathbf{0}}(r,y)\ {\dd}r \right){\dd}\nu_{\mathbf{0}}(y) \\
& =\int_{\Omega \setminus B_{e^{R}}(\mathbf{0})} r^{bps} F^{*p}({\dd}v_\kappa)\, {\dd}x + \int_{B_{e^{-R}}(\mathbf{0})} r^{bps} F^{*p}({\dd}v_\kappa)\, {\dd}x,
\end{align*}
which together with \eqref{eqlemma_uiiiHPW} and \eqref{eqlemma_uivHPW} in Lemma \ref{lemma_u_exp_HPW} yields
\begin{align*}
\lim_{R\rightarrow +\infty} \left( \frac{A - A^{(1)}_R}{A} \right) =0 \quad \Rightarrow \quad \lim_{R\rightarrow +\infty} \left( \frac{ A^{(1)}_R}{A} \right) =1.
\end{align*}
Since $A^{(1)}_R  = \frac{1}{\kappa^p} \tilde{A}^{(1)}_R = \frac{1}{\kappa^{p-1}} V^{(1)}_R$, it follows that \eqref{A/A1/VHPW} holds.
\end{proof}

\begin{lemma}\label{lem:I2_estimatesHPW}
With the notation of Theorem \ref{thm:HPW_sharp}, we have
\[
\lim_{R\to +\infty} \frac{A^{(2)}_R}{A}=\lim_{R\to +\infty} \frac{\tilde A^{(2)}_R}{\kappa^p A}=\lim_{R\to +\infty} \frac{V^{(2)}_R}{ \kappa^{p-1} A }=0.
\]
\end{lemma}
\begin{proof}
On \(I_2=[e^{R},e^R+1]\), \(v_{\kappa,R}=v_\kappa\varphi_R\), with \(|\varphi_R'|\le C_2\), for some constant $C_2>0$. Since \( |v_\kappa| \) decays exponentially as \(e^{-\frac{r^{\alpha+1}}{\kappa(\alpha+1)}}\) = \(e^{-\frac{r^{\alpha+1}}{R^{\alpha+1}(\alpha+1)}}\), $R>2$, and $\alpha+1 >0$, we have
\[
|v_{\kappa,R}(r)|\le e^{-\frac{e^{(\alpha+1)R}}{\kappa(\alpha+1)}} ,\qquad
|v'_{\kappa,R}(r)|\le \left( \frac{r^\alpha}{R^{\alpha+1}} +C_2 \right) e^{-\frac{e^{(\alpha+1)R}}{\kappa(\alpha+1)}} < \left(e^{(\alpha+1)R}+C_2 \right) e^{-\frac{e^{(\alpha+1)R}}{\kappa(\alpha+1)}}.
\]
Thus
\begin{align*}
A^{(2)}_R & = \int_{B_{e^R+1}(\mathbf{0}) \backslash B_{e^R}(\mathbf{0}) } |v'_{\kappa, R, 2}|^{p}\, r^{bps} {\dd}x
 \\
& \leq  \left(e^{(\alpha+1)R}+C_2 \right)^p e^{-\frac{p e^{(\alpha+1)R}}{\kappa(\alpha+1)}}  \int_{B_{e^R+1}(\mathbf{0}) \backslash B_{e^R}(\mathbf{0}) } r^{bps} {\dd}x,
\end{align*}
which, combined with \eqref{inte_ine}, yields
\begin{align*}
A^{(2)}_R
& \leq {\bar{\varrho}}^{n+1} \mathscr{I}_{\m}(\mathbf{0})  \left(e^{(\alpha+1)R}+C_2 \right)^p e^{-\frac{p e^{(\alpha+1)R}}{\kappa(\alpha+1)}} \int^{e^R+1}_{e^R}   r^{bps + n-1} (r+\bar{\varrho})^{-(n+1)} {\dd}r\\
& \leq \mathscr{I}_{\m}(\mathbf{0})  \left(e^{(\alpha+1)R}+C_2 \right)^p e^{-\frac{p e^{(\alpha+1)R}}{\kappa(\alpha+1)}} (e^R+1)^{bps + n-1}.
\end{align*}
Together with \eqref{exp_uiiHPW_ineq}, it implies $\lim_{R\rightarrow +\infty} \frac{A^{(2)}_R}{A} =0$
because the exponential term decays faster than any positive power of \(R\).
The same argument applies to \(\tilde A^{(2)}_R\) and \(V^{(2)}_R\).
\end{proof}

\vskip 10mm

\section{Failure of the generalized CKN and HPW inequalities
}\label{section:K=0}

This section is devoted to the failure of classical functional inequalities on projectively flat Finsler manifolds with zero flag curvature, i.e.,
Theorem \ref{thm:CKNHPW_threshold}\eqref{1.2(ii)}.
It complements the results of Section \ref{sec:sharp} and \ref{SharpHPW}, where the validity of the generalized CKN and HPW inequalities was established for
$s>2$. Together with Section \ref{sec:sharp}  and \ref{SharpHPW}, this provides a complete characterization of the parameter ranges for which the generalized CKN and HPW inequalities hold with sharp constant or fail.
As consequences, we also obtain the failure of the
$L^p$-Hardy inequality and the generalized uncertainty principle in the same setting.

To avoid constructing compactly supported function sequences via cut-off functions as in the previous two sections (where such a construction was unavoidable), we first introduce a functional and prove Lemma \ref{lemma:Cpab_well}. This will make the proof of Theorem~\ref{thm:CKNHPW_threshold}\eqref{1.2(ii)} simpler. Of course, one could also prove the theorem by directly constructing a compactly supported sequence, but the present approach is technically simpler.

\begin{definition}\label{def:C_pab}
Let $(\Omega, F, \mathscr{L}^n)$ be a forward complete projectively flat Finsler manifold with $\mathbf{K}=0$  endowed with the Lebesgue measure.
Let $a>1$, $p>1$, and $b\ge 0$. Assume
$1-n < s \le 2$, $0 \le b \leq \frac{p+1}{2p}$, with
\[
bps+n>0,\qquad (1-b)p's+n>0.
\]
\noindent
Define a functional on $\Lip_0(\Omega)\backslash\{0\}$ by
\[
\mathscr{C}_{p,a,b,s}(f):= \frac{\left( \int_{\Omega} {r^{b p s} } {F}^{* p} ( {\dd}f)\, {\dd}x   \right)^{1/p}\left( \int_{\Omega}  {|f|^{p'(a-1)}}{r^{(1-b) p' s}}\,{\dd}x   \right)^{1/p'}  }{ \int_{\Omega}  {|f|^{a}}{r^{s-1}}\,{\dd}x  }.
\]
\end{definition}

\begin{lemma}\label{lemma:Cpab_well}
Under the conditions of Definition \ref{def:C_pab}, the functional $\mathscr{C}_{p,a,b,s}$ is well-defined on $\Lip_0(\Omega)\backslash\{0\}$, and
\begin{equation}
\inf_{f\in C_0^\infty(\Omega)\setminus\{0\}} \mathscr{C}_{p,a,b,s}(f)
=
\inf_{f\in \Lip_0(\Omega)\setminus\{0\}} \mathscr{C}_{p,a,b,s}(f). \label{Lipf=Cf}
\end{equation}
\end{lemma}
\begin{proof}
We first show that $\mathscr{C}_{p,a,b,s}$ is well-defined on $\Lip_0(\Omega)\backslash\{0\}$. Let $f\in \Lip_0(\Omega)\backslash\{0\}$ with support contained in $B_R(\mathbf{0})$ for some $R>1$.
A direct calculation combined with \eqref{dualff*} yields
\begin{align*}
\limsup_{z\rightarrow x}\frac{|f(z)-f(x )|}{d_F(x ,z)}&=\sup_{y\in S_x\Omega}\left(\lim_{t\rightarrow 0^+}\frac{|f(\gamma_y(t))-f(\gamma_y(0))|}{t}\right) \\
&=\sup_{y\in S_x\Omega}|\langle {\dd}f, y\rangle|=\max\{  F^*(x,{\dd}f), F^*(x,-{\dd}f) \},
\end{align*}
where $\gamma_y(t)$ denotes a unit-speed geodesic starting at $x$ with $\dot{\gamma}_y(0)=y$. Then
for $\mathscr{L}^n$-a.e. $x\in \Omega$, we have
\[
F^*(x,{\dd}f) \le \limsup_{z\to x}\frac{|f(z)-f(x)|}{d_F(x,z)} \le \dil(f) < +\infty,
\]
where $\dil(f)$ is the Lipschitz constant of $f$.
Then, since $\m(\supp f)<+\infty$ and \eqref{inte_ine},
\[ \int_\Omega r^{bps} F^{*p}({\dd}f)\, {\dd}x  \leq \dil(f)^{p} \int_{B_R(\mathbf{0})} r^{bps} \, {\dd}x < +\infty, \]
The latter integral is finite because $bps + n > 0$ and near $r=0$ the volume element behaves like $r^{n-1}{\dd}r$, so the integral converges near $0$.

For the second numerator integral,
\[
\int_\Omega |f|^{p'(a-1)} r^{(1-b)p's}\, {\dd}x
\le (\max |f|)^{p'(a-1)} \int_{B_R(\mathbf{0})} r^{(1-b)p's}\, {\dd}x.
\]
We check that this integral is finite. The assumption $(1-b)p's+n>0$ ensures convergence near $r=0$ (the volume element behaves like $r^{n-1}{\dd}r$); outside $0$ the integral is finite because the support is compact. Hence the integral is finite.

Similarly, for the denominator,
\[
\int_\Omega |f|^a r^{s-1}\, {\dd}x
\le (\max |f|)^a \int_{B_R(\mathbf{0})} r^{s-1}\, {\dd}x < +\infty,
\]
because  the assumption $s>1-n$  ensures integrability near $r=0$. Thus $\mathscr{C}_{p,a,b,s}(f)$ is finite and positive.

It remains to prove the equality of the infima. Since $C_0^\infty(\Omega)\subset \Lip_0(\Omega)$, we trivially have
\[
\inf_{f\in C_0^\infty}\mathscr{C}_{p,a,b,s}(f) \ge \inf_{f\in \Lip_0}\mathscr{C}_{p,a,b,s}(f).
\]
For the reverse inequality, let $h\in \Lip_0(\Omega)\backslash\{0\}$ be arbitrary. By Lemma \ref{lipsconverppax} (or by standard mollification, since the support is compact), there exists a sequence $\{h_k\}_{k\ge1}\subset C_0^\infty(\Omega)$ such that:
\begin{itemize}
\item $\operatorname{supp}(h_k)\subset B_{R+1}(\mathbf{0})$ for all sufficiently large $k$;
\item $h_k \to h$ uniformly on $\Omega$;
\item $F^*({\dd}h_k) \to F^*({\dd}h)$ in $L^p(\Omega)$.
\end{itemize}

We now show the convergence of the three integrals in the definition of $\mathscr{C}_{p,a,b,s}$. For the first integral, since the supports of $h_k$ and $h$ are contained in the common bounded set $B_{R+1}(\mathbf{0})$, and $r^{bps}$ is locally integrable on $B_{R+1}(\mathbf{0})$ (this follows from $bps+n>0$), we first note that, by the triangle inequality for $F^*$,
\[
F^*({\dd}h_k)\leq F^*({\dd}h)+F^*({\dd}h_k-{\dd}h),
\qquad
F^*({\dd}h)\leq F^*({\dd}h_k)+F^*({\dd}h-{\dd}h_k).
\]
Thus, by Lemma \ref{lipsconverppax}, $F^*({\dd}h_k)\to F^*({\dd}h) \text{ in } L^p(\Omega)$, and hence
\[
F^{*p}({\dd}h_k)\to F^{*p}({\dd}h)
\quad\text{in }L^1(B_{R+1}(\mathbf{0})).
\]
Since \(\{F^*({\dd}h_k)\}\) is uniformly bounded on the common support by Lemma~\ref{lipsconverppax}, the dominated convergence theorem yields the \(L^1\)-convergence of the \(p\)-th powers.
 Therefore, together with the local integrability of $r^{bps}$, a standard truncation argument yields
\[
\int_\Omega r^{bps} F^{*p}({\dd}h_k)\,{\dd}x \to \int_\Omega r^{bps} F^{*p}({\dd}h)\,{\dd}x.
\]

For the second integral, the uniform convergence $h_k\to h$ gives pointwise convergence of the integrands, and the sequence is uniformly bounded on the compact support; hence by the dominated convergence theorem,
\[
\int_\Omega |h_k|^{p'(a-1)} r^{(1-b)p's}\, {\dd}x
\to \int_\Omega |h|^{p'(a-1)} r^{(1-b)p's}\, {\dd}x.
\]
Similarly,
\[
\int_\Omega |h_k|^a r^{s-1}\, {\dd}x \to \int_\Omega |h|^a r^{s-1}\, {\dd}x.
\]
Therefore, $\mathscr{C}_{p,a,b,s}(h_k) \to \mathscr{C}_{p,a,b,s}(h)$ as $k\to +\infty$. Consequently, for every $h\in \Lip_0$,
\[
\inf_{f\in C_0^\infty}\mathscr{C}_{p,a,b,s}(f) \le \lim_{k\to +\infty}\mathscr{C}_{p,a,b,s}(h_k)=\mathscr{C}_{p,a,b,s}(h),
\]
so taking the infimum over $h$ gives
\[
\inf_{f\in C_0^\infty}\mathscr{C}_{p,a,b,s}(f) \le \inf_{f\in \Lip_0}\mathscr{C}_{p,a,b,s}(f).
\]
Combining the two inequalities yields \eqref{Lipf=Cf}.
\end{proof}

We are now ready to prove the failure result for the generalized CKN and HPW inequalities when $s\le2$.
The argument is explicit: we construct a family of test functions $w_\iota$ and show that the corresponding quotient in the definition of
$\mathscr{C}_{p,a,b,s}$ tends to $0$ as $\iota\to 0^+$. The volume estimates from Lemma \ref{lem_hat_sigma} and the approximation result from Lemma \ref{lemma:Cpab_well} will be used below.
We note that the following result holds even for $a=p$ (HPW), because it does not need the condition $a>p>1$ (CKN) as the theorems in Section \ref{sec:sharp}.


\begin{proof}[Proof of Theorem \ref{thm:CKNHPW_threshold}\eqref{1.2(ii)}]
	Let $(r,y)$ be the polar coordinate system around $\mathbf{0}$.
	Define a function $w_\iota$ on $\Omega$ by
	\[
	w_\iota(x) := -e^{-\iota r^{1 + \frac{\mu}{p}}},
	\]
	with constants $\iota, \mu$ satisfying
	\begin{equation}\label{condi_iota}
		0 < \iota < \min\left\{a^{-1}, a-1, p^{-1}, e^{-1} \right\}, \qquad
		 \mu = \frac{3}{2}.
	\end{equation}
	
	For the quotient $\mathscr{C}_{p,a,b,s}$, we need to estimate the integrals in it: two from the numerator and one from the denominator.
	
	The rest of the proof is organized into four steps. In Step 1 we estimate the denominator. In Step 2 we estimate the first numerator integral, which splits into Cases a1 and a2 according to the sign of \(s\). In Step 3 we estimate the second numerator integral, which splits into Cases b1 and b2 according to the exponent of \(r\) in the tail integral. In Step 4 we pass from the non-compactly supported functions \(w_\iota\) to compactly supported functions using Lemma \ref{lemma:Cpab_well}.
	
	\paragraph{\emph{Step 1: Denominator estimate.}}
	By the left inequality in \eqref{inte_ine}, we have
	\begin{align}
		\int_{\Omega} {|w_{\iota}|^a}{r^{s-1}}{\dd}x
		&\geq
		{\hat{\varrho}}^{n+1}\mathscr{I}_{\m}(\mathbf{0})
		\int^{+\infty}_0
		e^{-\iota a r^{1+\frac{\mu}{p}}}
		r^{n+s-2}(r+\hat{\varrho})^{-(n+1)}
		{\dd}r
		\notag\\
		&\geq
		{\hat{\varrho}}^{n+1}\mathscr{I}_{\m}(\mathbf{0})
		\int_{\frac12}^{\iota^{-\frac{p}{p+\mu}}} e^{-\iota a r^{1+\frac{\mu}{p}}}
		r^{s-3} \left(\frac{r}{r+\hat{\varrho}} \right)^{n+1}\,{\dd}r. \label{Thm5.2DenomEst}
\end{align}
For $\frac12\leq r\leq\iota^{-\frac{p}{p+\mu}},$
	we have
	\begin{align*}
	\iota r^{1+\frac{\mu}{p}}\leq1,
	\quad
	e^{-\iota a r^{1+\frac{\mu}{p}}}\geq e^{-a},
\quad
	\frac{r}{r+\hat{\varrho}}\geq \frac{\frac{1}{2}}{\frac{1}{2}+\hat{\varrho}} = \frac{1}{1+2\hat{\varrho}}.
	\end{align*}
	Moreover, since $s\leq2$, a direct estimate gives
	\[
	\int_{\frac12}^{1}r^{s-3}{\dd}r\geq\frac12,
	\qquad
	\int_1^{\iota^{-\frac{p}{p+\mu}}}r^{s-3}{\dd}r
	\geq
	\iota^{\frac{p(2-s)}{p+\mu}}
	\int_1^{\iota^{-\frac{p}{p+\mu}}}\frac{{\dd}r}{r}.
	\]
Then it follows from \eqref{Thm5.2DenomEst} that
\begin{align}
\int_{\Omega} {|w_{\iota}|^a}{r^{s-1}}{\dd}x
		&\geq
		\frac{{\hat{\varrho}}^{n+1}\mathscr{I}_{\m}(\mathbf{0})}
		{e^a(1+2\hat{\varrho})^{n+1}}
		\int_{\frac12}^{\iota^{-\frac{p}{p+\mu}}}
		r^{s-3}\,{\dd}r
		=
		\frac{{\hat{\varrho}}^{n+1}\mathscr{I}_{\m}(\mathbf{0})}
		{e^a(1+2\hat{\varrho})^{n+1}}
		\left(
		\int_{\frac12}^{1} r^{s-3}{\dd}r +
		\int_1^{\iota^{-\frac{p}{p+\mu}}} r^{s-3}{\dd}r
		\right)
		\notag\\
		&\geq
		\frac{{\hat{\varrho}}^{n+1}\mathscr{I}_{\m}(\mathbf{0})}
		{e^a(1+2\hat{\varrho})^{n+1}}
		\left(
		\frac12+
		\iota^{\frac{p(2-s)}{p+\mu}}
		\int_1^{\iota^{-\frac{p}{p+\mu}}}\frac{{\dd}r}{r}
		\right)
		\notag\\
		&=
		\frac{{\hat{\varrho}}^{n+1}\mathscr{I}_{\m}(\mathbf{0})}
		{e^a(1+2\hat{\varrho})^{n+1}}
		\left(
		\frac12+
		\frac{p}{p+\mu}
		\iota^{\frac{p(2-s)}{p+\mu}}
		(-\ln\iota)
		\right).
		\label{FinslerhARDdenome}
	\end{align}

	\medskip
	\paragraph{\emph{Step 2: First numerator integral.}}
	We now estimate the first numerator integral. Since \(F^*({\dd}r)=1\), \(\mu=3/2>1\), and \(bps+n>0\), together with \eqref{inte_ine}, a direct computation gives
	\begin{align}
		&\int_{\Omega} r^{bps} F^{*p}({\dd}w_{\iota}) {\dd}x\notag\\
		=& \, \iota^p \left(1+\frac{\mu}{p}\right)^p  \int_{\Omega} r^{\mu + bps} e^{- \iota p r^{1+\frac{\mu}{p}}} {\dd}x \notag\\
		\leq &\,  \iota^p \left(1+\frac{\mu}{p}\right)^p {\bar{\varrho}}^{n+1} \mathscr{I}_{\m}(\mathbf{0})  \left(   \int^{1}_0 r^{\mu+bps+n-1 } (r+\bar{\varrho})^{-(n+1)}  {\dd}r
		+ \int_1^{+\infty} e^{- \iota p r^{1+\frac{\mu}{p}}} r^{\mu+bps} (r+\bar{\varrho})^{-2}  {\dd}r \right) \notag \\
		\leq &\, \iota^p \left( 1 +\frac{\mu}{p}\right)^p {\bar{\varrho}}^{n+1} \mathscr{I}_{\m}(\mathbf{0})   \left( \frac{1}{\bar{\varrho}^{n+1}} +    \int_{1}^{+\infty}  e^{- \iota p r^{1+\frac{\mu}{p}}} r^{\mu +bps  -2}   {\dd}r\right). \label{K=0_v_numer_first0_general_lem}
	\end{align}
	It remains to estimate the tail integral
	\[
	\int_{1}^{+\infty}
	e^{- \iota p r^{1+\frac{\mu}{p}}}
	r^{\mu+bps-2}
	{\dd}r,
	\]
	which we treat in two cases according to the sign of $s$.
	
	\medskip
	\subparagraph{\bf{Case a1: $s \leq 0$.}}
	In this case, let $z:= \iota p r^{1+\frac{\mu}{p}}$.
	Since $\frac{p(\mu-2)-\mu}{p+\mu}>-1$, we have
	\begin{align*}
	&\quad	\int_{1}^{+\infty}
		e^{- \iota p r^{1+\frac{\mu}{p}}}
		r^{\mu+bps-2}{\dd}r
		\leq
		\int_{1}^{+\infty}
		e^{- \iota p r^{1+\frac{\mu}{p}}}
		r^{\mu-2}{\dd}r\\
		&\leq
		\iota^{\frac{p(1-\mu)}{p+\mu}}
		\frac{p^{\frac{p(2-\mu)+\mu}{p+\mu}}}{p+\mu}
		\int_0^{+\infty}
		e^{-z}
		z^{\frac{p(\mu-2)-\mu}{p+\mu}}
		{\dd}z
		=
		\iota^{\frac{p(1-\mu)}{p+\mu}}
		\frac{p^{\frac{p(2-\mu)+\mu}{p+\mu}}}{p+\mu}
		\Gamma\left(
		\frac{p(\mu-2)-\mu}{p+\mu}+1
		\right)\\
		&=:
		\iota^{\frac{p(1-\mu)}{p+\mu}}
		\mathcal{C}_1(p,\mu).
	\end{align*}
	Then it follows from \eqref{K=0_v_numer_first0_general_lem} and
	$\iota\in(0,1)$ that
	\begin{align*}
		\int_{\Omega}
		r^{bps}F^{*p}({\dd}w_\iota){\dd}x
		&\leq
		\iota^{\frac{p(p+1)}{p+\mu}}
		\left(
		1+{\bar{\varrho}}^{n+1}\mathcal{C}_1(p,\mu)
		\right)
		\mathscr{I}_{\m}(\mathbf{0})
		\left(1+\frac{\mu}{p}\right)^p\\
		&=:
		\iota^{\frac{p(p+1)}{p+\mu}}
		\mathcal{\bar C}_1(n,p,\mu,\bar{\varrho})
		\mathscr{I}_{\m}(\mathbf{0}).
	\end{align*}
	
	\medskip
	\subparagraph{\bf{Case a2: $s > 0$.}}
	In this case, again let
	$z:=\iota p r^{1+\frac{\mu}{p}}$.
	Since
	$\frac{p(\mu+bps-2)-\mu}{p+\mu}>-1$, we have
	\begin{align*}
		\int_{1}^{+\infty}
		e^{- \iota p r^{1+\frac{\mu}{p}}}
		r^{\mu+bps-2}{\dd}r
		&\leq
		\iota^{\frac{p(1-\mu-bps)}{p+\mu}}
		\frac{p^{\frac{p(2-\mu-bps)+\mu}{p+\mu}}}{p+\mu}
		\int_0^{+\infty}
		e^{-z}
		z^{\frac{p(\mu+bps-2)-\mu}{p+\mu}}
		{\dd}z\\
		&=
		\iota^{\frac{p(1-\mu-bps)}{p+\mu}}
		\frac{p^{\frac{p(2-\mu-bps)+\mu}{p+\mu}}}{p+\mu}
		\Gamma\left(
		\frac{p(\mu+bps-2)-\mu}{p+\mu}+1
		\right)\\
		&=:
		\iota^{\frac{p(1-\mu-bps)}{p+\mu}}
		\mathcal{C}_2(b,p,s,\mu).
	\end{align*}
	Then it follows from \eqref{K=0_v_numer_first0_general_lem}
	and $\iota\in(0,1)$ that
	\begin{align*}
		\int_{\Omega}
		r^{bps}F^{*p}({\dd}w_\iota){\dd}x
		&\leq
		\iota^{\frac{p(p+1-bps)}{p+\mu}}
		\left(
		1+{\bar{\varrho}}^{n+1}\mathcal{C}_2(b,p,s,\mu)
		\right)
		\mathscr{I}_{\m}(\mathbf{0})
		\left(1+\frac{\mu}{p}\right)^p\\
		&=:
		\iota^{\frac{p(p+1-bps)}{p+\mu}}
		\mathcal{\bar C}_2(n,b,p,s,\mu,\bar{\varrho})
		\mathscr{I}_{\m}(\mathbf{0}).
	\end{align*}
	
	Combining the two cases above, we obtain the following uniform estimate for the first numerator integral:
	\begin{align}\label{denominator1}
		\int_{\Omega}
		r^{bps}F^{*p}({\dd}w_\iota){\dd}x
		\leq
		\begin{cases}
			\iota^{\frac{p(p+1)}{p+\mu}}
			\mathcal{\bar C}_1(n,p,\mu,\bar{\varrho})
			\mathscr{I}_{\m}(\mathbf{0}),
			&\quad \text{if }s\leq0,\\[4pt]
			\iota^{\frac{p(p+1-bps)}{p+\mu}}
			\mathcal{\bar C}_2(n,b,p,s,\mu,\bar{\varrho})
			\mathscr{I}_{\m}(\mathbf{0}),
			&\quad \text{if }s>0.
		\end{cases}
	\end{align}
	where \(\mathcal{\bar{C}}_1\) and \(\mathcal{\bar{C}}_2\) are positive constants depending on the indicated parameters and on \(\bar{\varrho}\). In the sequel, any constant denoted by \(\mathcal{\bar{C}}\) or \(\mathcal C\) with a subscript is a positive constant depending only on the fixed parameters and geometric quantities; its value may change from line to line.
	
	\medskip
	\paragraph{\emph{Step 3: Second numerator integral.}}
	We now estimate the second numerator integral.
	By a direct calculation using \eqref{inte_ine}, we have
	\begin{align}
		\int_{\Omega}
		{|w_\iota|^{(a-1)p'}}r^{(1-b)p's}{\dd}x
		&\leq
		\bar{\varrho}^{n+1} \mathscr{I}_{\m}(\mathbf{0})
		\int_0^{+\infty}
		e^{-\iota(a-1)p'r^{1+\frac{\mu}{p}}}
		r^{(1-b)p's+n-1}
		(r+\bar{\varrho})^{-(n+1)}
		{\dd}r
		\notag\\
		&\leq
		\mathscr{I}_{\m}(\mathbf{0})
		\left(
		\int_0^1
		r^{(1-b)p's+n-1}{\dd}r
		+
		\bar{\varrho}^{n+1}
		\int_1^{+\infty}
		e^{-\iota(a-1)p'r^{1+\frac{\mu}{p}}}
		r^{(1-b)p's-2}{\dd}r
		\right)
		\notag\\
		&=
		\mathscr{I}_{\m}(\mathbf{0})
		\left(
		\frac{1}{(1-b)p's+n}
		+
		\bar{\varrho}^{n+1}
		\int_1^{+\infty}
		e^{-\iota(a-1)p'r^{1+\frac{\mu}{p}}}
		r^{(1-b)p's-2}{\dd}r
		\right).
		\label{K=0_v_numer_second1_general}
	\end{align}
	It remains to estimate the tail integral
	\[
	\int_1^{+\infty}
	e^{-\iota(a-1)p'r^{1+\frac{\mu}{p}}}
	r^{(1-b)p's-2}{\dd}r.
	\]
	We consider two cases according to the exponent of \(r\) in the integrand.
	
	\medskip
	\subparagraph{\bf{Case b1: $(1-b)p's-2\leq-1$.}}
	In this case, let
	$\bar z:=\iota(a-1)r^{1+\frac{\mu}{p}}$.
	By the exponent condition, we have
	\begin{align}
		\int_1^{+\infty}
		e^{-\iota(a-1)p'r^{1+\frac{\mu}{p}}}
		r^{(1-b)p's-2}{\dd}r
		&\leq
		\int_1^{+\infty}
		e^{-\iota(a-1)r^{1+\frac{\mu}{p}}}
		r^{-1}{\dd}r
		=
		\frac{p}{p+\mu}
		\int_{\iota(a-1)}^{+\infty}
		e^{-\bar z}\bar z^{-1}{\dd}\bar z
		\notag\\
		&\leq
		\frac{p}{p+\mu}
		\int_{\iota(a-1)}^1
		\bar z^{-1}{\dd}\bar z
		+
		\frac{p}{p+\mu}
		\int_1^{+\infty}
		e^{-\bar z}{\dd}\bar z
		\notag\\
		&=
		\frac{p}{p+\mu}
		\left[
		-\ln\iota-\ln(a-1)+e^{-1}
		\right].
		\label{K=0_v_numer_second2_general}
	\end{align}
	By the first inequality in \eqref{condi_iota},
	$-\ln\iota>-\ln(a-1)$ and
	$-\ln\iota>e^{-1}$, and hence
	\[
	\int_1^{+\infty}
	e^{-\iota(a-1)p'r^{1+\frac{\mu}{p}}}
	r^{(1-b)p's-2}{\dd}r
	\leq
	-\frac{3p}{p+\mu}\ln\iota.
	\]
	Then it follows from
	\eqref{K=0_v_numer_second1_general}
	and the first inequality in \eqref{condi_iota} that
	\begin{align}
		\int_{\Omega}
		{|w_\iota|^{(a-1)p'}}r^{(1-b)p's}{\dd}x
		&\leq
		(-\ln\iota)
		\left(
		\frac{1}{(1-b)p's+n}
		+
		\frac{3p{\bar{\varrho}}^{n+1}}{p+\mu}
		\right)
		\mathscr{I}_{\m}(\mathbf{0})
		\notag\\
		&=:
		(-\ln\iota)\,
		\mathcal{C}_3(b,p,s,\mu,\bar{\varrho})
		\mathscr{I}_{\m}(\mathbf{0}).
		\label{K=0_v_numer_second3_general}
	\end{align}

	Combining \eqref{denominator1} with
	\eqref{K=0_v_numer_second3_general}, we obtain
	\begin{align}\label{denominator2}
		&\left(
		\int_{\Omega}
		r^{bps}F^{*p}({\dd}w_\iota){\dd}x
		\right)^{1/p}
		\left(
		\int_{\Omega}
		{|w_\iota|^{(a-1)p'}}
		r^{(1-b)p's}{\dd}x
		\right)^{1/p'}
		\notag\\
		\leq&\,
		\begin{cases}
			\iota^{\frac{p+1}{p+\mu}}
			(-\ln\iota)^{\frac1{p'}}
			\mathscr{I}_{\m}(\mathbf{0})
			\mathcal{\bar C}_1^{\frac1p}
			(n,p,\mu,\bar{\varrho})
			\mathcal C_3^{\frac1{p'}}
			(b,p,s,\mu,\bar{\varrho}),
			&\quad\text{if }s\leq0,
			\\[5pt]
			\iota^{\frac{p+1-bps}{p+\mu}}
			(-\ln\iota)^{\frac1{p'}}
			\mathscr{I}_{\m}(\mathbf{0})
			\mathcal{\bar C}_2^{\frac1p}
			(n,b,p,s,\mu,\bar{\varrho})
			\mathcal C_3^{\frac1{p'}}
			(b,p,s,\mu,\bar{\varrho}),
			&\quad
			\text{if }0<s\leq\frac{1}{(1-b)p'}.
		\end{cases}
	\end{align}
	
	For $s\leq0$, the right-hand side of
	\eqref{denominator2} tends to $0$, and hence
	\eqref{FinslerhARDdenome} yields
	\[
	\lim_{\iota\rightarrow0^+}
	\mathscr{C}_{p,a,b,s}(w_\iota)=0.
	\]
	For $s>0$, since
	$
	b\leq\frac{p+1}{2p}$ and
	$s\leq2$,
	we have
	\[
	p+1-bps\geq0.
	\]
	If $p+1-bps>0$, the right-hand side of
	\eqref{denominator2} again tends to $0$, and
	\eqref{FinslerhARDdenome} yields
	\[
	\lim_{\iota\rightarrow0^+}
	\mathscr{C}_{p,a,b,s}(w_\iota)=0.
	\]
If $p+1-bps=0$, by $
	b\leq\frac{p+1}{2p}$ and
	$s\leq2$ again, we have
\[
p+1=bps \le b p\cdot 2 \le \frac{p+1}{2p}\cdot 2p = p+1.
\]
Thus equality must hold throughout, which forces $s=2$ and $b=\frac{p+1}{2p}$.
	\[
	b=\frac{p+1}{2p},
	\qquad
	s=2.
	\]
	Then, by \eqref{FinslerhARDdenome} and \eqref{denominator2}, we have
	\begin{align*}
		\mathscr{C}_{p,a,b,s}(w_\iota)
		&\leq
		\frac{e^a(1+2\hat{\varrho})^{n+1}}
		{\hat{\varrho}^{n+1}}
		\mathcal{\bar{C}}_2^{\frac{1}{p}}
		(n,b,p,s,\mu,\bar{\varrho})
		\mathcal{C}_3^{\frac{1}{p'}}
		(b,p,s,\mu,\bar{\varrho})
		\frac{(-\ln\iota)^{\frac{1}{p'}}}
		{\frac{1}{2}+\frac{p}{p+\mu}(-\ln\iota)}
		\longrightarrow 0
		\qquad\text{as }\iota\rightarrow0^+,
	\end{align*}
	since $\frac{1}{p'}-1=-\frac{1}{p}<0$. Hence, in Case {\rm b1},
	\[
	\lim_{\iota\rightarrow0^+}
	\mathscr{C}_{p,a,b,s}(w_\iota)=0.
	\]

	\medskip
	\subparagraph{\bf{Case b2: $(1-b)p's-2>-1$.}}
	In this case, let
	$\bar z:=\iota(a-1)r^{1+\frac{\mu}{p}}$.
	Since
	\[
	\frac{p[(1-b)p's-2]-\mu}{p+\mu}>-1,
	\]
	we have
	\begin{align*}
 \int_1^{+\infty} e^{-\iota(a-1)p'r^{1+\frac{\mu}{p}}}
		r^{(1-b)p's-2}{\dd}r
		&\leq
		\int_1^{+\infty}
		e^{-\iota(a-1)r^{1+\frac{\mu}{p}}}
		r^{(1-b)p's-2}{\dd}r
		\\
		&\leq
		\iota^{\frac{p[1-(1-b)p's]}{p+\mu}}
		(a-1)^{\frac{p[1-(1-b)p's]}{p+\mu}}
		\frac{p}{p+\mu}
		\int_0^{+\infty}
		e^{-\bar z}
		\bar z^{\frac{p[(1-b)p's-2]-\mu}{p+\mu}}
		{\dd}\bar z
		\\
		&=
		\iota^{\frac{p[1-(1-b)p's]}{p+\mu}}
		(a-1)^{\frac{p[1-(1-b)p's]}{p+\mu}}
		\frac{p}{p+\mu}
		\Gamma\left(
		\frac{p[(1-b)p's-2]-\mu}{p+\mu}+1
		\right)
		\\
		&=:
		\iota^{\frac{p[1-(1-b)p's]}{p+\mu}}
		\mathcal C_4(a,b,p,s,\mu).
	\end{align*}
	Then it follows from
	\eqref{K=0_v_numer_second1_general}
	and $\iota\in(0,1)$ that
	\begin{align}\label{denominator2_}
		\int_{\Omega}
		{|w_\iota|^{(a-1)p'}}
		r^{(1-b)p's}{\dd}x
		&\leq
		\iota^{\frac{p[1-(1-b)p's]}{p+\mu}}
		\left(
		1+{\bar{\varrho}}^{n+1}
		\mathcal C_4(a,b,p,s,\mu)
		\right)
		\mathscr I_{\m}(\mathbf0)
		\notag\\
		&=:
		\iota^{\frac{p[1-(1-b)p's]}{p+\mu}}
		\mathcal{\bar C}_4
		(a,b,p,s,\mu,\bar{\varrho})
		\mathscr I_{\m}(\mathbf0).
	\end{align}
	Together with \eqref{denominator1} and
	\eqref{denominator2_}, this yields
	\begin{align}\label{fenzi1}
		&\left(
		\int_{\Omega}
		r^{bps}F^{*p}({\dd}w_\iota){\dd}x
		\right)^{1/p}
		\left(
		\int_{\Omega}
		{|w_\iota|^{(a-1)p'}}
		r^{(1-b)p's}{\dd}x
		\right)^{1/p'}
		\notag\\
		\leq&\,
		\iota^{
			\frac{p+1-bps}{p+\mu}
			+
			\frac{p[1-(1-b)p's]}{p'(p+\mu)}
		} \mathscr I_{\m}(\mathbf0)
		\left[
		\mathcal{\bar C}_2
		(n,b,p,s,\mu,\bar{\varrho})
		\right]^{\frac1p} \left[
		\mathcal{\bar C}_4
		(a,b,p,s,\mu,\bar{\varrho})
		\right]^{\frac1{p'}}.
	\end{align}
	The power of $\iota$ on the right-hand side is
	\begin{align}\label{jisuan}
		\frac{p+1-bps}{p+\mu}
		+
		\frac{p[1-(1-b)p's]}{p'(p+\mu)}
		=
		\frac{p(2-s)}{p+\mu}
		\geq0.
	\end{align}
	Therefore, by \eqref{FinslerhARDdenome},
	\begin{align*}
		\mathscr C_{p,a,b,s}(w_\iota)
		&\leq
		\frac{
			e^a(1+2\hat{\varrho})^{n+1}
		}{
			\hat{\varrho}^{n+1}
		}
		\mathcal{\bar C}_2^{\frac1p}
		(n,b,p,s,\mu,\bar{\varrho})
		\mathcal{\bar C}_4^{\frac1{p'}}
		(a,b,p,s,\mu,\bar{\varrho})
		\\
		&\qquad\cdot
		\frac{
			\iota^{\frac{p(2-s)}{p+\mu}}
		}{
			\frac12+
			\frac{p}{p+\mu}
			\iota^{\frac{p(2-s)}{p+\mu}}
			(-\ln\iota)
		}
		\longrightarrow 0
		\qquad
		\text{as }\iota\rightarrow0^+.
	\end{align*}
	Indeed, if $s<2$, the numerator tends to $0$, while if $s=2$, the denominator tends to $+\infty$.
	Hence
	\begin{equation}\label{faillimfinal}
		\lim_{\iota\rightarrow0^+}
		\mathscr C_{p,a,b,s}(w_\iota)=0.
	\end{equation}
	
	\medskip
	\paragraph{\emph{Step 4: Passing to compactly supported functions.}}
	For each sufficiently small $\epsilon$, set
	\begin{equation*}
		w_{\iota,\epsilon}
		=
		\begin{cases}
			-e^{-\iota\epsilon^{1+\frac{\mu}{p}}},
			&\quad\text{if }r<\epsilon,\\
			-e^{-\iota r^{1+\frac{\mu}{p}}},
			&\quad\text{if }r\geq\epsilon,
		\end{cases}
		\qquad
		\tilde w_{\iota,\epsilon}
		:=
		\min\{0,w_{\iota,\epsilon}+\epsilon\}.
	\end{equation*}
For $r$ large, $w_{\iota,\epsilon} + \epsilon >0$, hence $\tilde w_{\iota,\epsilon}=0$;	thus
	$\tilde w_{\iota,\epsilon}\in
	\Lip_0(\Omega)\backslash\{0\}$.
	A direct calculation yields
	\begin{equation}\label{failfinallip}
		\lim_{\epsilon\rightarrow0^+}
		\mathscr C_{p,a,b,s}
		(\tilde w_{\iota,\epsilon})
		=
		\mathscr C_{p,a,b,s}(w_\iota)
		<+\infty.
	\end{equation}
	Now Lemma \ref{lemma:Cpab_well} combined with
	\eqref{faillimfinal} and \eqref{failfinallip}
	implies
	\[
	\inf_{f\in C_0^\infty(\Omega)\setminus\{0\}}
	\mathscr C_{p,a,b,s}(f)
	=
	\inf_{f\in\Lip_0(\Omega)\setminus\{0\}}
	\mathscr C_{p,a,b,s}(f)
	\leq
	\lim_{\iota\rightarrow0^+}
	\mathscr C_{p,a,b,s}(w_\iota)
	=
	0.
	\]
	This completes the proof.
\end{proof}

\begin{remark}\label{rem:CKN_fail_consequences}
As immediate consequences of Theorem~\ref{thm:CKNHPW_threshold}\eqref{1.2(ii)}, we obtain the failure of the classical $L^p$-Hardy inequality and of the HPW uncertainty principle in the same setting.
Indeed, taking $a=p$ and $b=0$ in Theorem~\ref{thm:CKNHPW_threshold}\eqref{1.2(ii)}, we have $bps+n=n>0$ and $(1-b)p's+n=p's+n>0$, while the conditions $s\le2$ and $1-n<s$ are satisfied. Then:
\begin{enumerate}
\item[(i)]
For $s=1-p$, the inequality in Theorem~\ref{thm:CKNHPW_threshold}\eqref{1.2(ii)} reduces to the failure of the $L^p$-Hardy inequality:
\[
\inf_{f\in C_0^\infty(\Omega)\setminus\{0\}}
\frac{\int_\Omega F^{*p}({\dd}f)\,{\dd}x}{\int_\Omega |f|^p r^{-p}\,{\dd}x}=0.
\]
Indeed, for \(a=p\), \(b=0\), and \(s=1-p\), the quotient in Theorem~\ref{thm:CKNHPW_threshold}\eqref{1.2(ii)} becomes
\[
\left(\frac{\int_\Omega F^{*p}(df)\,dx}{\int_\Omega |f|^p r^{-p}\,dx}\right)^{1/p},
\]
so its vanishing is equivalent to the failure of the \(L^p\)-Hardy inequality.
\item[(ii)] For $s\in(1-p,2]$, it reduces to the failure of the generalized HPW uncertainty principle:
\[
\inf_{f\in C_0^\infty(\Omega)\setminus\{0\}}
\frac{\left(\int_\Omega F^{*p}({\dd}f)\,{\dd}x\right)^{1/p}
\left(\int_\Omega |f|^p r^{p's}\,{\dd}x\right)^{1/p'}}
{\int_\Omega |f|^p r^{s-1}\,dx}=0.
\]
\end{enumerate}
\end{remark}

\vskip 5mm


\begin{thebibliography}{10}

\bibitem{BCS} D.~Bao, S.~S.~Chern, Z.~Shen, \textit{An Introduction to Riemannian--Finsler Geometry}, GTM, vol.~\textbf{200}, Springer--Verlag, 2000.

\bibitem{Berchio}
E.~Berchio, D.~Ganguly, G.~Grillo,
\textit{Sharp Poincar\'e--Hardy and Poincar\'e--Rellich inequalities on the hyperbolic space},
J.~Funct.~Anal., \textbf{272} (2017), 1661--1703.

\bibitem{Be1} L.~Berwald, \textit{Parallel\"ubertragung in allgemeinen R\"aumen}, Atti Congr.~Intern.~Mat.~Bologna, \textbf{4} (1928), 263--270.

\bibitem{Be2} L.~Berwald, \textit{\"Uber die n-dimensionalen Geometrien konstanter Kr\"ummung, in denen die Geraden die k\"urzesten sind}, Math.~Z., \textbf{30} (1929), 449--469.



\bibitem{Bryant1} R. Bryant, \textit{ Finsler structures on the 2-sphere satisfying $K = 1$}, Finsler Geometry,
Contemporary Mathematics \textbf{196}, Amer. Math. Soc., Providence,
RI, 1996, 27-42.
\bibitem{Bryant2} R. Bryant, \textit{ Projectively flat Finsler 2-spheres of constant curvature}, Selecta Math.,
New Series, \textbf{3}(1997), 161-204.

\bibitem{CKN} L.~Caffarelli, R.~Kohn, L.~Nirenberg, \textit{First order interpolation inequalities with
weights}, Compos. Math. \textbf{53} (1984) 259--275.

\bibitem{doCarmo}
M.~P. do Carmo, C.~Xia,
\textit{Complete manifolds with non-negative Ricci curvature and the Caffarelli--Kohn--Nirenberg inequalities},
Compos.~Math., \textbf{140} (2004), 818--826.

\bibitem{Carron}
G.~Carron,
\textit{In\'egalit\'es de Hardy sur les vari\'et\'es riemanniennes non-compactes},
J.~Math.~Pures Appl., \textbf{76} (1997), 883--891.

\bibitem{CFL} C.~Cazacu, J.~Flynn, N.~Lam, \textit{Sharp second order uncertainty principles}, J.~Funct.~Anal., \textbf{283}(10) (2022), 109659, 25~pp.

\bibitem{CFLL} C.~Cazacu, J.~Flynn, N.~Lam, G.~Lu, \textit{Caffarelli--Kohn--Nirenberg identities, inequalities and their stabilities}, J.~Math.~Pures~Appl., \textbf{182}(9) (2024), 253--284.



\bibitem{Ch1} S.~S.~Chern, \textit{Finsler geometry is just Riemannian geometry without the quadratic restriction}, Notices Amer.~Math.~Soc., \textbf{43}(9) (1996), 959--963.

\bibitem{CKZ} N.~Cui, A.~Krist\'aly, W.~Zhao, \textit{Sharp Hardy Inequalities Involving Distance Functions From Submanifolds of Riemannian Manifolds}, J.~Geom.~Anal., \textbf{35} (2025), 379.

\bibitem{Dambrosio}
L.~D'Ambrosio, S.~Dipierro,
\textit{Hardy inequalities on Riemannian manifolds and applications},
Ann.~Inst.~H.~Poincar\'e Anal.~Non Lin\'eaire, \textbf{31} (2014), 449--475.

\bibitem{FFK} F.~Faraci, C.~Farkas, A.~Krist\'aly, \textit{Multipolar Hardy inequalities on Riemannian manifolds}, ESAIM: COCV, \textbf{24}(2) (2018), 551--567.

\bibitem{HanXu2024}
B.-X.~Han, Z.-F.~Xu,
\textit{Sharp uncertainty principles on metric measure spaces},
Calc.~Var.~Partial Differential Equations, \textbf{63}(4) (2024), Paper No.~104, 16~pp.

\bibitem{HKZ} L.~Huang, A.~Krist\'aly, W.~Zhao, \textit{Sharp uncertainty principles on general Finsler manifolds}, Trans.~Amer.~Math.~Soc., \textbf{373} (2020), 8127--8161.


\bibitem{Ka} S.~Kaj\'ant\'o, \textit{Sharp Hardy and spectral gap inequalities on special irreversible Finsler manifolds}, Proc.~Roy.~Soc.~Edinburgh Sect.~A, \textsf{First~View} (2025), 1--36, {https://doi.org/10.1017/prm.2025.10072}.



\bibitem{KKPZ} S.~Kaj\'ant\'o, A.~Krist\'aly, I.~R.~Peter, W.~Zhao, \textit{A generic functional inequality and Riccati pairs: an alternative approach to Hardy-type inequalities}, Math.~Ann., \textbf{390} (2024), 3621--3663.


\bibitem{Kombe1}
I.~Kombe, M.~\"Ozaydin,
\textit{Improved Hardy and Rellich inequalities on Riemannian manifolds},
Trans.~Amer.~Math.~Soc., \textbf{361} (2009), 6191--6203.

\bibitem{Kombe2}
I.~Kombe, M.~\"Ozaydin,
\textit{Hardy--Poincar\'e, Rellich and uncertainty principle inequalities on Riemannian manifolds},
Trans.~Amer.~Math.~Soc., \textbf{365} (2013), 5035--5050.


\bibitem{Kristaly2015}
A.~Krist\'aly,
\textit{A sharp Sobolev interpolation inequality on Finsler manifolds},
J.~Geom.~Anal., \textbf{25} (2015), 2226--2240.

\bibitem{Kristaly2018}
A.~Krist\'aly,
\textit{Sharp uncertainty principles on Riemannian manifolds: the influence of curvature},
J.~Math.~Pures Appl., \textbf{119} (2018), 326--346.





\bibitem{KLZ} A.~Krist\'aly, B.~Li, W.~Zhao, \textit{Failure of famous functional inequalities on Finsler manifolds: the influence of $S$-curvature}, Math.\ Z.,\ \textbf{313}(4)(2026), Paper No. 81.

\bibitem{KristalyOhta}
A.~Krist\'aly, S.-I. Ohta,
\textit{Caffarelli--Kohn--Nirenberg inequality on metric measure spaces with applications},
Math.~Ann., \textbf{357} (2013), 711--726.

\bibitem{KristalyRepovs}
A.~Krist\'aly, D.~Repov\v{s},
\textit{Quantitative Rellich inequalities on Finsler--Hadamard manifolds},
Commun.~Contemp.~Math., \textbf{18} (2016), 1650020.

\bibitem{KR} A.~Krist\'aly, I.~J.~Rudas, \textit{Elliptic problems on the ball endowed with Funk-type metrics}, Nonlinear Anal., \textbf{199} (2015), 199--208.

\bibitem{KristalySzakal2019}
A.~Krist\'aly, P.~Szak\'al,
\textit{Interpolation between Brezis--V\'azquez and Poincar\'e inequalities on nonnegatively curved spaces: sharpness and rigidities},
J.~Differential Equations, \textbf{266} (2019), 6621--6646.

\bibitem{Li} B.~Li, \textit{On the classification of projectively flat Finsler metrics with constant flag curvature}, Adv.~in~Math., \textbf{257} (2014), 266--284.

\bibitem{LZ} B.~Li and W.~Zhao, \textit{Unexpected analytic phenomena on Finsler manifolds}, J. Differential Equations \textbf{481}(2026), 114659.





\bibitem{LiZhao_Hilbert4} B.~Li and W.~Zhao, \textit{Hilbert's fourth problem in the constant curvature setting}, preprint, 2025. arXiv:2510.25171v3

\bibitem{XiaLi2026}
R.~Li, Q.~Xia,
\textit{Sharp \(L^p\)-uncertainty principles on Finsler measure spaces}, preprint, 2026.
arXiv:2607.11227

\bibitem{MST2020}
A.~Mercaldo, M.~Sano, F.~Takahashi,
\textit{Finsler Hardy inequalities},
Math.~Nachr., \textbf{293}(12) (2020), 2370--2398.

\bibitem{MPV2020}
\'A.~Mester, I.~R.~Peter, C.~Varga,
\textit{Sufficient criteria for obtaining Hardy inequalities on Finsler manifolds},
Mediterr.~J.~Math., \textbf{18}(2) (2021), Paper No.~76, 18~pp.


\bibitem{Ohta1} S.-I.~Ohta, \textit{Comparison Finsler geometry}, Springer Monogr.~Math., Springer,  Cham, 2021.

\bibitem{Rade1} H.-B.~Rademacher, \textit{A sphere theorem for non-reversible Finsler metrics}, Math.~Ann., \textbf{328}(3) (2004), 373--387.

\bibitem{Rade} H.-B.~Rademacher, \textit{Non-reversible Finsler metrics of positive flag curvature}, in: \textit{A sampler of Riemann--Finsler geometry}, Math.~Sci.~Res.~Inst.~Publ., vol.~\textbf{50}, Cambridge Univ.~Press, Cambridge, 2004, pp.~261--302.




\bibitem{ShenSpray} Z.~Shen, \textit{Differential geometry of spray and Finsler spaces}, Kluwer Acad.~Publ., Dordrecht, 2001.

\bibitem{ShenLecture} Z.~Shen, \textit{Lectures on Finsler geometry}, World Sci.~Publ., Singapore, 2001.



\bibitem{Sh1} Z.~Shen, \textit{Projectively flat Finsler metrics of constant flag curvature}, Trans.~Amer.~Math.~Soc., \textbf{355}(4) (2003), 1713--1728.

\bibitem{Warner} F.~W.~Warner, \textit{Foundations of Differentiable Manifolds and Lie Groups}, Springer-Verlag, 1983.

\bibitem{WeiWu2020}
S.~W.~Wei, B.~Y.~Wu,
\textit{Generalized Hardy type and Caffarelli--Kohn--Nirenberg type inequalities on Finsler manifolds},
Internat.~J.~Math., \textbf{31}(13) (2020), 2050109, 27~pp.

\bibitem{Xia}
C.~Xia,
\textit{The Caffarelli--Kohn--Nirenberg inequalities on complete manifolds},
Math.~Res.~Lett., \textbf{14} (2007), 875--885.

\bibitem{Yang}
Q.~Yang, D.~Su, Y.~Kong,
\textit{Hardy inequalities on Riemannian manifolds with negative curvature},
Commun.~Contemp.~Math., \textbf{16} (2014), 1350043.

\bibitem{ZhaoHardy} W.~Zhao, \textit{Hardy inequalities with best constants on Finsler metric measure manifolds}, J.~Geom.~Anal., \textbf{31}(2) (2021), 1992--2032.

\bibitem{ZS} W.~Zhao, Y.~Shen, \textit{A universal volume comparison theorem for Finsler manifolds and related results}, Canad.~J.~Math., \textbf{65}(6) (2013), 1401--1435.














































\end{thebibliography}
\end{document}